\documentclass[11pt,a4paper,reqno]{amsart}
\usepackage[utf8]{inputenc}
\usepackage[T1]{fontenc}
\usepackage[francais,english]{babel}

\usepackage{amsmath,xcolor,color,dsfont}
\usepackage{amsfonts}
\usepackage{amssymb,amsthm}

\colorlet{darkgreen}{green!50!black}

\usepackage{hyperref} 
\hypersetup{	
colorlinks=true, 
breaklinks=true, 
urlcolor= green, 
linkcolor= blue,	
citecolor=darkgreen,	
pdftitle={Asymptotic expansion of stationary distribution for reflected Brownian motion in the quarter plane via analytic approach}, 
pdfauthor={Sandro Franceschi et Kilian Raschel}, 
pdfsubject={Simulation}	
}

\newtheorem{thm}{Theorem}
\newtheorem{prop}[thm]{Proposition}

\newtheorem{deff}[thm]{Definition}
\newtheorem{lem}[thm]{Lemma}

\renewcommand{\geq}{\geqslant}
\renewcommand{\leq}{\leqslant}

\usepackage{epsfig}
\usepackage{dsfont}

\usepackage{geometry}
\usepackage{enumitem}
\usepackage{mathtools}
\usepackage{eqnarray,cases}

\usepackage{indentfirst}

\usepackage{pstricks-add}
\usepackage[numbers]{natbib}

\usepackage{float}

\newcommand{\s}{\mathbf{S}}

\newcommand{\R}{\mathcal{R}}

\let\phi=\varphi
\newcommand{\cal}{\mathcal}

\title[Asymptotics of the stationary distribution for SRBM in \texorpdfstring{$\mathbb{R}_+^2$}{the quarter plane}]{Asymptotic expansion of stationary distribution for
 reflected Brownian motion 
 in the quarter plane
  \\ via analytic approach}

\author{S.\ Franceschi} \address{Laboratoire de Probabilit\'es et
        Mod\`eles Al\'eatoires, Universit\'e Pierre et Marie Curie,
        4 Place Jussieu, 75252 Paris Cedex 05, France
        \& Laboratoire de Math\'ematiques et Physique Th\'eorique, Universit\'e de Tours, Parc de Grandmont, 37200 Tours, France} \email{sandro.franceschi@upmc.fr}

 \author{I.\ Kurkova} \address{Laboratoire de Probabilit\'es et
        Mod\`eles Al\'eatoires, Universit\'e Pierre et Marie Curie,
        4 Place Jussieu, 75252 Paris Cedex 05, France} \email{Irina.Kourkova@upmc.fr}


\begin{document}

\begin{abstract}
Brownian motion in ${\bf R}_+^2$
  with covariance matrix $\Sigma$ and drift $\mu$ in the interior and reflection matrix $R$ from the axes is considered.
  The  asymptotic expansion of the stationary distribution density along all paths in ${\bf R}_+^2$  is found
  and its main term is identified depending on parameters $(\Sigma, \mu, R)$.
   For this purpose the analytic approach of  Fayolle, Iasnogorodski and Malyshev in \cite{fayolle_random_1999} and \cite{malyshev_asymptotic_1973},
   restricted essentially up to now to discrete random walks in ${\bf Z}_+^2$ with jumps 
  to the nearest-neighbors in the interior is developed in this article for diffusion processes on ${\bf R}_+^2$ with 
  reflections on the axes.
\end{abstract}

\maketitle
\date{\today}

\footnote{Version of \today}


\section{Introduction and main results}

\subsection{Context}

 Two-dimensional semimartingale  reflecting Brownian motion (SRBM) in the quarter plane received a lot of attention from the mathematical community.
  Problems such as SRBM existence \cite{reiman_boundary_1988,taylor_existence_1993}, stationary distribution conditions \cite{harrison_diffusion_1978,harrison_brownian_1987}, explicit forms of stationary distribution in special cases \cite{dai_stationary_2013,
  dieker_reflected_2009,harrison_diffusion_1978,harrison_multidimensional_1987,latouche_product-form_2013},
 large deviations \cite{avram_explicit_2001,dai_stationary_2013,majewski_large_1996,majewski_large_1998}
construction of Lyapunov functions
\cite{dupuis_lyapunov_1994},
and queueing networks approximations
\cite{harrison_diffusion_1978,harrison_brownian_1993,
lieshout_tandem_2007,lieshout_asymptotic_2008,williams_approximation_1996}
have been intensively studied in the literature.
References cited above are non-exhaustive, see also \cite{williams_semimartingale_1995} for a survey of some of these topics.
Many results on two-dimensional SRBM have been fully or partially generalized to higher dimensions.

\medskip
In this article we  consider stationary SRBMs in the quarter plane  and focus on the asymptotics  of their stationary distribution along any
  path in ${\bf R}_+^2$.  Let  $Z(\infty)=(Z_1(\infty), Z_2(\infty))$ be a random vector  that has the stationary distribution
 of the SRBM. In \cite{dai_reflecting_2011},  Dai et Myazawa obtain the following asymptotic result:
   for a given directional vector $c \in {\bf R}_+^2$ they  find the function $f_c(x)$ such that
   $$\lim_{ x \to \infty} \frac{\mathbf{P}( \langle c  \mid Z(\infty) \rangle \geq x)}{f_c(x)}=1$$
  where $\langle \cdot \mid \cdot \rangle$ is the inner product. In \cite{dai_stationary_2013} they  compute the exact asymptotics of two boundary stationary measures on the axes  associated with $Z(\infty)$. In this article  we solve a harder problem arisen in \cite[\S 8 p.196]{dai_reflecting_2011},  the one to compute the asymptotics of 
  $$\mathbf{P}(Z(\infty) \in  x c+B), \text{ as } x\to \infty,$$
  where $c \in {\bf R}_+^2$ is any directional vector and  $B \subset {\bf R}_+^2$ is a compact subset. Furthermore, our objective is to find {\it the full asymptotic expansion  of the density $\pi(x_1,x_2)$ of $Z(\infty)$  as $x_1,x_2 \to \infty$  and
   $x_2/x_1 \to  {\rm \tan } (\alpha)$ for any given angle  $\alpha \in ]0,\pi/2[$.}

\medskip
Our main tool is the analytic method developed by V.~Malyshev in  \cite{malyshev_asymptotic_1973} to compute the asymptotics of stationary
  probabilities for discrete random walks in ${\bf Z}_+^2$ with jumps to the nearest-neighbors in the interior and reflections
  on the axes.  This method proved to be  fruitful for the analysis of Green functions  and Martin boundary \cite{kourkova_random_2011,kurkova_martin_1998},  and also useful  for studying some joining the shortest queue
  models \cite{kurkova_malyshevs_2003}.
The  article  \cite{malyshev_asymptotic_1973}
has been a part of Malyshev's global analytic approach   
 to study discrete-time random walks in
 ${\bf Z}_+^2$  with four domains of spatial homogeneity  (the interior of ${\bf Z}_+^2$, the axes and the origin).
 Namely, in  the book \citep{malyshev_sluchainye_1970} he
 made explicit their stationary probability generating functions as solutions of boundary problems on the universal covering of the
  associated Riemann surface and studied the nature of these functions depending on parameters.
 G.~ Fayolle and R.~Iasnogorodski \cite{fayolle_two_1979} determined these generating functions
  as solutions of boundary problems of Riemann-Hilbert-Carleman type on the complex plane. Fayolle, Iansogorodski and Malyshev
   merged together and deepened  their methods in the book \cite{fayolle_random_1999}. The latter is entirely devoted to
  the explicit form of stationary probabilities generating functions for discrete random walks  in ${\bf Z}_+^2$ with nearest-neighbor jumps
  in the interior.  The analytic approach of this book  
has been further applied to the analysis of random walks absorbed on the axes in \cite{kourkova_random_2011}. It has been also especially efficient in combinatorics, where it allowed to study all models of walks in ${\bf Z}_+^2$ with small steps
 by making explicit the generating functions of the numbers of paths and clarifying their nature, see  \cite{raschel_counting_2012} and \cite{kourkova_functions_2012}.

\medskip
However, the methods of \cite{fayolle_random_1999} and \cite{malyshev_asymptotic_1973} seem to be essentially 
restricted to discrete-time models of  walks in the quarter plane
  with jumps in the interior only to the nearest-neighbors. They can hardly be extended to discrete models with bigger jumps, even at distance 2,  nevertheless some  attempts in this direction
  have been done in \cite{fayolle_about_2015}. 
In fact, while for jumps at distance 1 the Riemann surface associated with the random walk is the torus,
  bigger jumps lead to Riemann surfaces of higher genus, where the analytic procedures of \cite{fayolle_random_1999}  seem much more difficult to carry out.
       Up to now, as far as we know, neither the analytic approach of \cite{fayolle_random_1999}, nor the asymptotic results
   \cite{malyshev_asymptotic_1973} have been translated  to the continuous analogs of random walks in
${\bf Z}_+^2$, such as SRBMs in ${\bf R}_+^2$, except for some special cases in \cite{baccelli_analysis_1987} and in \cite{foddy_analysis_1984}.
     This article is the first one in this direction.
   Namely, the asymptotic expansion of the stationary distribution density
 for SRBMs  is  obtained  by methods strongly inspired by \cite{malyshev_asymptotic_1973}. 
The aim of this work goes beyond the solution of this particular problem. It provides the basis for the development  
of the analytic approach of \cite{fayolle_random_1999} for diffusion processes in  cones of ${\bf R}_+^2$  which is continued 
  in the next articles \cite{franceschi_tuttes_2016} and \cite{franceschi_explicit_2017}. 
     In \cite{franceschi_explicit_2017}  the first author  and K. Raschel  make explicit Laplace transform
of the invariant measure for SRBMs  in the quarter plane with general 
 parameters of the drift, covariance and reflection matrices.
      Following \cite{fayolle_random_1999},  they express it  in an integral form 
  as a solution of a  boundary value problem and then discuss possible simplifications
   of this integral formula for some particular sets of  parameters.
   The special case of orthogonal reflections from the axes is the subject of \cite{franceschi_tuttes_2016}.
    Let us note that the analytic approach   for SRBMs in  ${\bf R}_+^2$  which is 
  developed in the present paper and continued by the next ones \cite{franceschi_tuttes_2016} and \cite{franceschi_explicit_2017}, looks
  more transparent than the one for discrete models and deprived of many second order details.
  Last but not the least, contrary to random walks in  ${\bf Z}_+^2$ with jumps at distance 1, it can be easily 
  extended  to diffusions in any cones of ${\bf R}^2$ via linear transformations, as we observe in the concluding remarks, see  Section~\ref{subsec:concludingremarks}.

\subsection{Reflected Brownian motion in the quarter plane}

  We now define properly the two-dimensional SRBM  and present our results.  Let
\begin{align*}
  \begin{cases}
    \Sigma = \left(  \begin{array}{cc} \sigma_{11} & \sigma_{12} \\ \sigma_{12} & \sigma_{22} \end{array} \right) \in {\bf R}^{2 \times 2}
     \text{ be a non-singular covariance matrix,}
    \\
    \mu= \left(  \begin{array}{c} \mu_1 \\  \mu_2  \end{array} \right) \in {\bf R}^2
     \text{ be a drift,}
    \\
    R= (R^1,R^2)=\left(  \begin{array}{cc} r_{11} & r_{12} \\ r_{21} & r_{22} \end{array} \right) \in {\bf R}^{2 \times 2}
     \text{ be a reflection matrix.}
  \end{cases}
\end{align*}

\begin{deff}
The stochastic process $Z(t)=(Z^1(t), Z^2(t))$ is said to be
a reflected Brownian motion with drift in the quarter plane ${\bf R}_+^2$ associated with data $(\Sigma, \mu, R)$ if
$$ Z(t)=Z_0 + W(t) + \mu t + R L(t) \ \in {\bf R}_+^2,$$
where
\begin{itemize}
\item[(i)] $(W(t))_{t \in \mathbb{R^+}} $ is an unconstrained planar Brownian motion with covariance matrix $\Sigma$, starting from $0$;
\item[(ii)] $L(t)=(L^1(t), L^2(t))$;
 for $i=1,2$, $L^i(t)$  is a continuous and non-decreasing process that increases only at time $t$ such as $Z^i(t)=0$,
namely $\int_{0}^t 1_{\{Z^i(s) \ne 0 \}} \mathrm{d} L^i(s)=0$ $\forall t\geq 0$;
\item[(iii)] $Z(t)\in {\bf R}_+^2$  $\forall t\geq 0$.
\end{itemize}
\end{deff}

Process  $Z(t)$ exists if and only if $r_{11} > 0$, $r_{22} > 0$ and either  $r_{12},r_{21}>0$ or $r_{11} r_{22} - r_{12} r_{21} > 0$ (see \cite{taylor_existence_1993} and \cite{reiman_boundary_1988} which obtain an existence criterion in any dimension).
  In this case the process is unique in distribution for each given initial distribution of $Z_0$.

Columns $R^1$ and $R^2$ represent the directions where the Brownian motion is pushed when it reaches the axes, see Figure \ref{rebond}.

\begin{figure}[hbtp] 
\centering
\includegraphics[scale=1]{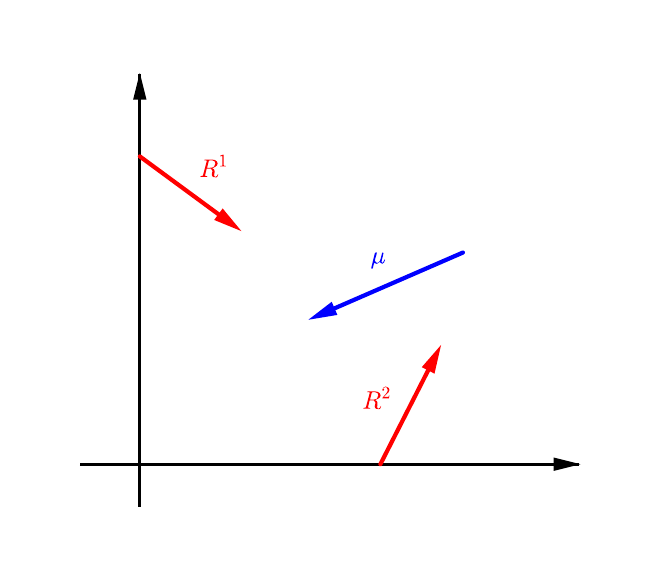}
\caption{Drift $\mu$ and reflection vectors $R^1$ and $R^2$}
\label{rebond}
\end{figure}

\begin{prop} The reflected Brownian motion $Z(t)$ associated with $(\Sigma, \mu, R)$ is well defined, and its stationary distribution $\Pi$ exists
and is unique if and only if the data satisfy the following conditions:
\begin{equation}
r_{11} > 0, \ \  r_{22} > 0, \ \ r_{11} r_{22} - r_{12} r_{21} > 0, \label{u}
\end{equation}
\begin{equation}
r_{22} \mu_1 - r_{12}  \mu_2 < 0, \ \ r_{11} \mu_2 - r_{21}  \mu_1 < 0 . \label{v}
\end{equation}
\end{prop}
 
The proof and some more detailed statements can be found in \cite{hobson_recurrence_1993,williams_recurrence_1985,harrison_reflected_2009}.
From now on we assume that conditions (\ref{u}) and  (\ref{v}) are satisfied.
The stationary
distribution $\Pi$  is absolutely continuous with respect to Lebesgue measure as it is shown  in
\cite{harrison_brownian_1987} and \cite{dai_steady-state_1990}.
We  denote its  density by $\pi(x_1,x_2)$.

\subsection{Functional equation for the stationary distribution}

Let $A$ be the generator of $(W_t +\mu t)_{t\geq 0}$. For each $f \in \mathcal{C}^2_b ({\bf R}_+^2)$ (the set of twice continuously differentiable functions $f$ on ${\bf R}_+^2$ such that $f$ and its first and second order derivatives are bounded) one has
$$Af(z)= \frac{1}{2} \sum_{i,j=1}^{2}  \sigma_{i,j} \frac{\partial^2 f}{\partial z_i \partial z_j}(z) + \sum_{i=1}^2 \mu_i \frac{\partial f}{\partial z_i} (z).$$
Let us define for $i=1,2$,
$$D_i f(x)= \langle R^i | \nabla f \rangle$$
that may be interpreted as generators on the axes.
We  define now  $\nu_1$ and $\nu_2$ two finite boundary measures with  their support on the axes:
 for any Borel set $B\subset {\bf R}_+^2$,
$$\nu_i (B) = \mathbb{E}_{\Pi} [ \int_0^1 1_{\{Z(u) \in B\}} \mathrm{d}L^i (u)].$$
By definition of  stationary distribution,
for all $t\geqslant 0$, $\mathbb{E}_{\Pi} [f(Z(t))]= \int_{{\bf R}_+^2} f(z) \Pi (\mathrm{d} z)$. A similar formula holds true for $\nu_i$:
$\mathbb{E}_{\Pi} [\int_0^t f(Z(u)) \mathrm{d}L^i(u)  ]= t \int_{{\bf R}_+^2} f(x) \nu_i (\mathrm{d} x)$.
Therefore $\nu_1$ and $\nu_2$ may be viewed as a kind of boundary invariant measures.
The basic adjoint relationship takes the following form:
for each $f \in \mathcal{C}^2_b ({\bf R}_+^2)$,
\begin{equation} \label{da}
 \int_{{\bf R}_+^2} Af(z) \Pi (\mathrm{d} z) + \sum_{i=1,2} \int_{{\bf R}_+^2} D_i f(z) \nu_i (\mathrm{d} z) =0 .
\end{equation}
The proof can be found in \cite{harrison_brownian_1987} in some particular cases and then has been extended to a general case, for example in \cite{dai_reflected_1992}.
We now define $\varphi(\theta)$  the two-dimensional Laplace transform of $\Pi$ also called its moment generating function. Let
$$
\varphi (\theta) = \mathbb{E}_{\Pi} [\exp (\langle \theta | Z\rangle)] = \iint_{{\bf R}_+^2} \exp (\langle \theta | z \rangle) \Pi (\mathrm{d} z)
$$
for all $\theta=(\theta_1,\theta_2) \in \mathbb{C}^2$ such that the integral converges.
 It does of course for any $\theta$ with ${\Re}\,\theta_1\leq 0, {\Re}\,\theta_2\leq 0$.  We have set
$\langle  \theta | Z \rangle=\theta_1 Z^1 + \theta_2 Z^2$.
Likewise we define the moment generating functions for $\nu_1(\theta_2)$ and  $\nu_2(\theta_1)$ on ${\bf C}$:
$$
\varphi_2 (\theta_1) =
\mathbb{E}_{\Pi} [ \int_0^1 e^{\theta_1 Z_t^1}  \mathrm{d} L^2(t)]
=\int_{{\bf R}_+^2} e^{\theta_1 z} \nu_2 (\mathrm{d} z),
\ 
\varphi_1 (\theta_2) =
\mathbb{E}_{\Pi} [ \int_0^1 e^{\theta_2 Z_t^2}  \mathrm{d} L^1(t)]
=\int_{{\bf R}_+^2} e^{\theta_2 z} \nu_1 (\mathrm{d} z).
$$
  Function $\phi_2(\theta_1)$ exists a priori for any $\theta_1$ with  ${\Re}\,\theta_1\leq 0$. It is  proved in \cite{dai_reflecting_2011}
 that it  also does for $\theta_1$  with ${\Re}\, \theta_1 \in [0,\epsilon_1]$,   up to its  first singularity $\epsilon_1>0$,
 the same is true for $\phi_1(\theta_2)$.
The following key functional equation (proven in \cite{dai_reflecting_2011}) results from the basic adjoint relationship (\ref{da}).
\begin{thm}
For any $\theta\in{\bf R}_+^2$ such $\varphi (\theta)<\infty$, $\varphi_2 (\theta_1)<\infty$ and  $\varphi_1 (\theta_2)<\infty$
we have the following fundamental functional equation:
\begin{equation}  \label{maineq}
\gamma (\theta) \varphi (\theta) =\gamma_1 (\theta) \varphi_1 (\theta_2) + \gamma_2 (\theta) \varphi_2 (\theta_1),
\end{equation}
where
\begin{align}
  \begin{cases}
     \gamma (\theta)=- \frac{1}{2} \langle \theta | \Sigma \theta \rangle- \langle \theta | \mu \rangle =-\frac{1}{2}(\sigma_{11} \theta_1^2 + \sigma_{22} \theta_2^2 +2\sigma_{12}\theta_1\theta_2)
-
(\mu_1\theta_1+\mu_2\theta_2),  \\
     \gamma_1 (\theta)= \langle R^1 | \theta \rangle=r_{11} \theta_1 + r_{21} \theta_2,  \\
     \gamma_2 (\theta)=\langle R^2 | \theta \rangle=r_{12} \theta_1 + r_{22} \theta_2.
  \end{cases}
\end{align}
\end{thm}
   This equation holds true a priori for any $\theta=(\theta_1, \theta_2)$ with ${\Re}\,\theta_1\leq 0, {\Re}\, \theta_2\leq 0$.
  It plays a crucial role in the analysis of the stationary distribution.

\subsection{Results}

    Our aim is to obtain  the asymptotic expansion of  the stationary distribution density $\pi(x)=\pi(x_1,x_2)$
as $x_1,x_2 \to \infty$ and $x_2/x_1 \to {\rm tan}\,(\alpha_0)$ for any given angle $\alpha_0 \in [0, \pi/2]$.

\medskip 

\noindent{\bf Notation.} We write the asymptotic expansion $f(x) \sim \sum_{k=0}^n g_k (x) $ as $x \to x_0$ if
   $g_k(x)=o(g_{k-1}(x))$  as $x \to x_0$ for all $k=1,\ldots, n$ and $f(x)-\sum_{k=0}^n g_k(x)=o(g_n(x))$ as 
  $x \to x_0$.

\medskip 

It will be more convenient to expand $\pi(r \cos \alpha, r \sin \alpha)$ 
  as $r \to \infty$ and $\alpha \to \alpha_0$. We give our final results in Section \ref{sec:asymptexpansion},
 Theorems \ref{thmresults1}--\ref{thmresults4}:  
  we find the  expansion of $\pi(r \cos \alpha, r \sin \alpha)$  as $r \to \infty$ and prove it uniform for $\alpha$ fixed in a small 
   neighborhood $\mathcal {O}(\alpha_0) \subset ]0, \pi/2[$ of $\alpha_0 \in ]0, \pi/2[$. 
   
In this section,  Theorem \ref{thmmain} below announces the main term of the expansion depending on parameters 
  $(\mu, \Sigma, R)$ and 
a  given direction $\alpha_0$.   
  Next, in Section \ref{subsec:sketch} we  sketch our analytic approach following the main lines of this paper  in order 
  to get the full asymptotic expansion of $\pi$.
 We  present at the same time the organization of the article. 
  
Now  we need to introduce some notations.
  The quadratic form $\gamma(\theta)$ is defined in (\ref{maineq})  via the covariance matrix $\Sigma$ and the drift $\mu$ of the process in the interior  of ${\bf R}_+^2$.
   Let us restrict ourselves on $\theta \in {\bf R}^2$.
    The equation $\gamma(\theta)=0$ determines an ellipse $\mathcal{ E}$ on ${\bf R}^2$ passing through the origin, its tangent in it is  orthogonal
  to vector $\mu$, see Figure \ref{ellipse}.
    Stability conditions (\ref{u}) and (\ref{v})  imply the negativity of at least one of coordinates of  $\mu$, see \cite[Lemma 2.1]{dai_reflecting_2011}.
     In this article, in order to shorten the number of pictures and cases of parameters to consider, we restrict ourselves
to the case
\begin{equation}
\label{mumu}
 \mu_1<0 \text{ and }  \mu_2<0,
\end{equation}
although our methods can be applied without any difficulty to other cases,
 we briefly sketch  some different details at the end of Section 2.4.    

\begin{figure}[hbtp]
\centering 
\includegraphics[scale=0.9]{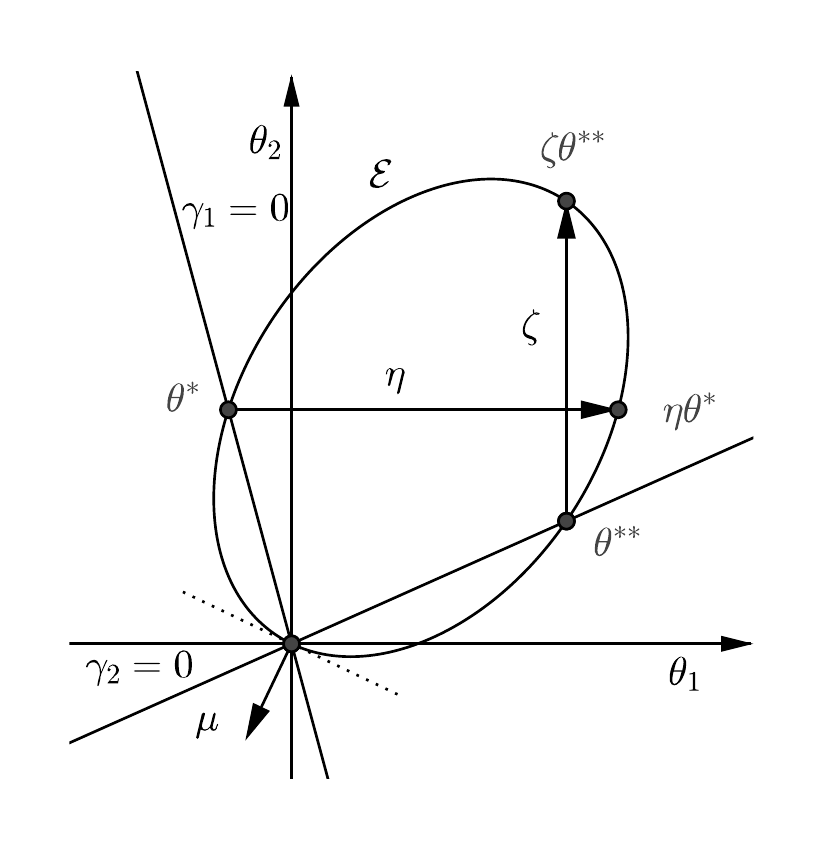} 
\caption{Ellipse $\mathcal{E}$,  straight lines $\gamma_1(\theta)=0$, $\gamma_2(\theta)= 0$, points 
$\theta^{*}, \theta^{**}, \eta \theta^*, \zeta \theta ^{**}$}
 \label{ellipse}
\end{figure}
Let us call $s_1^+=(\theta_1(s_1^+), \theta_2(s_1^+))\in {\cal E}$ the point of the ellipse with the maximal
first coordinate: $\theta_1(s_1^+)=\sup\{\theta_1 : \gamma(\theta_1,\theta_2)=0\}$.
  Let us call $s_2^+$ the point of the ellipse with the maximal second coordinate.
     Let ${\cal A}$ be the arc of the ellipse with endpoints  $s_1^+$,  $s_2^+$
  {\it not} passing through the origin, see Figure \ref{pointpolgeom}.
   For a given angle $\alpha \in [0,\pi/2]$  let us define the point $\theta(\alpha)$ on the arc ${\cal A}$ as
\begin{equation}
\label{thetaalpha}
\theta(\alpha)={\rm argmax}_{\theta \in {\cal A}}\langle \theta \mid e_\alpha  \rangle  \ \ \hbox{where }e_\alpha=(\cos \alpha, \sin \alpha).
\end{equation}
    Note that $\theta(0)=s_1^+$, $\theta(\pi/2)=s_2^+$, and $\theta(\alpha)$ is an isomorphism
  between $[0, \pi/2]$ and ${\cal A}$.
  Coordinates of $\theta(\alpha)$ are given explicitly in \eqref{cdpt}.
   One can also construct $\theta(\alpha)$ geometrically: first draw a ray $r(\alpha)$  on ${\bf R}_+^2$ that forms the angle $\alpha$ with $\theta_1$-axis,
 and then the straight line
$l(\alpha)$ orthogonal to this ray and tangent to the ellipse. Then $\theta(\alpha)$ is the point where $l(\alpha)$ is tangent to the ellipse, see Figure \ref{pointpolgeom}.

\begin{figure}[hbtp]
\centering
\includegraphics[scale=0.9]{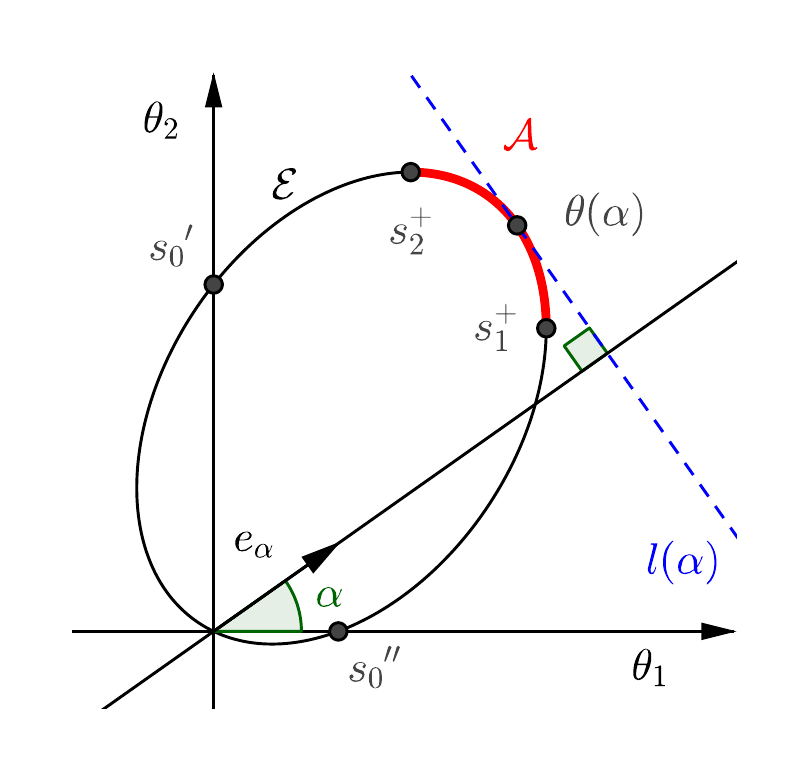}  
\caption{Arc $\cal{A}$ and  point $\theta(\alpha)$ on $\cal{E}$}
 \label{pointpolgeom}
\end{figure}

 Secondly, consider the straight lines $\gamma_1(\theta)=0$, $\gamma_2(\theta)=0$ defined in (\ref{maineq})  via the reflection matrix $R$.
  They cross  the ellipse ${\cal E}$ in the origin. Furthermore, due to stability conditions (\ref{u}) and  (\ref{v})  the  line $\gamma_1(\theta)=0$ [resp. $\gamma_2(\theta)=0$]  intersects the ellipse at the second point $\theta^{*}=(\theta_1^*, \theta_2^*)$
   (resp. $\theta^{**}=(\theta_1^{**}, \theta_2^{**})$)  where $\theta_2^*>0$ (resp. $\theta_1^{**}>0$).  Stability conditions also imply that 
  the ray $\gamma_1(\theta)=0$ is always "above" the ray $\gamma_2(\theta)=0$, see \cite[Lemma 2.2]{dai_reflecting_2011}.
   To present our results, we need to define the images of these points via the so-called Galois automorphisms $\zeta$ and $ \eta$ of ${\cal E}$.
      Namely, for  point $\theta^*=(\theta_1^*,\theta_2^*) \in {\cal E}$ there exists
 a unique point  $\eta \theta^*=(\eta \theta_1^*,   \theta_2^*) \in {\cal E}$  that has the same second coordinate.
 Clearly, $\theta_1^*$ and $ \eta \theta_1^*$
 are two roots of the second degree equation $\gamma(\cdot, \theta_2^{*})=0.$
   In the same way for point $\theta^{**}=(\theta_1^{**},\theta_2^{**}) \in {\cal E}$ there exists
 a unique point  $\zeta \theta^{**}=( \theta_1^{**},  \zeta \theta_2^{**}) \in {\cal E}$  with the same first coordinate.
 Points $\theta_2^{**}$  and $ \zeta \theta_2^{**}$
 are two roots of the second degree equation $\gamma(\theta_1^{**}, \cdot)=0.$
   Points  $\theta^*$, $\theta^{**}$,  $\eta \theta^*$ and $\zeta \theta^{**}$ are pictured  on Figure \ref{ellipse}.
  Their coordinates are made explicit in (\ref{zerostheta}) and (\ref{zerosthetaauto}).

   Finally  let $s'_0=(0,-2\frac{\mu_{22}}{\sigma_{22}}
   )$ be the point of intersection of the ellipse $\mathcal{E}$ with $\theta_2$-axis 
      and let $s''_0=(-2\frac{\mu_{11}}{\sigma_{11}},0)$ be the point of intersection of the ellipse with $\theta_1$-axis, see
Figure \ref{pointpolgeom}.
   The following theorem provides the main asymptotic term of $\pi(r\cos \alpha, r \sin \alpha)$.

\begin{thm}
\label{thmmain}
 Let $\alpha_0 \in ]0, \pi/2[$.
 Let $\theta(\alpha)$ be defined in  (\ref{thetaalpha}).
    Let $\red \{ \theta(\alpha_0), s'_0 \}$ {\rm (}resp. $\blue \{s''_0, \theta(\alpha_0)\}${\rm )} be the arc of the ellipse ${\cal E}$ with end points
     $s'_0$ and  $\theta(\alpha_0)$ {\rm  (}resp.  $s''_0$ and  $\theta(\alpha_0)${\rm )} \underline{not}  passing through the origin.
 We have the following results.
{\begin{itemize}
\item[(1)] If $\zeta\theta^{**}\notin \{ \theta(\alpha_0), s'_0 \}$ and
$\eta\theta^{*}\notin \{s''_0, \theta(\alpha_0) \}$,
then there exists a constant $c(\alpha_0)$ such that 
\begin{equation}
\label{aspointcol}
\pi(r \cos \alpha , r \sin \alpha)\sim \frac{c(\alpha_0)}{\sqrt{r} }
\exp\Big(- r \langle e_\alpha \mid \theta(\alpha)  \rangle \Big) \  \  \ r\to \infty, \alpha \to \alpha_0.
\end{equation}
  The function $c(\alpha)$ varies continuously on $[0, \pi/2]$, $\lim_{\alpha \to 0} c(\alpha) =\lim_{\alpha \to \pi/2} c(\alpha)=0$.
  \label{1}
\item[(2)] \label{2} If $\zeta\theta^{**}\in \}\theta(\alpha_0), s'_0 \}$ and
$\eta\theta^{*}\notin \{ s''_0,\theta(\alpha_0) \{$, then with some constant $c_1>0$ 
\begin{equation}
\label{aspol1}
\pi(r \cos \alpha , r \sin \alpha )\sim c_1
\exp\Big(-r \langle e_\alpha \mid \zeta \theta^{**} \rangle \Big) \  \  r\to \infty, \alpha \to \alpha_0.
\end{equation}
\item[(3)] \label{3} If $\zeta\theta^{**}\notin \} \theta(\alpha_0),  s'_0\}$ and $\eta\theta^{*}\in  \{ s''_0, \theta(\alpha_0) \{ $,
then  with some constant $c_2>0$
\begin{equation}
\label{aspol2}
\pi(r \cos \alpha , r \sin \alpha)\sim c_2
\exp\Big(- r \langle e_\alpha
 \mid  \eta \theta^{*} \rangle \Big) \ \  r\to \infty, \alpha \to \alpha_0.
\end{equation}
\item[(4)] \label{4} Let $\zeta\theta^{**}\in \} \theta(\alpha_0), s'_0 \}$ and
$\eta\theta^{*}\in \{ s''_0,\theta(\alpha_0) \{$.
   If $\langle \zeta \theta^{**} \mid e_{\alpha_0} \rangle < \langle \eta \theta^{*} \mid e_{\alpha_0} \rangle$, 
  then the asymptotics  (\ref{aspol1}) is valid with some constant $c_1>0$.
     If  $\langle \zeta \theta^{**} \mid e_{\alpha_0} \rangle > \langle \eta \theta^{*} \mid e_{\alpha_0} \rangle$, 
  then the asymptotics  (\ref{aspol2}) is valid with some constant $c_2>0$.
    If  $\langle \zeta \theta^{**} \mid e_{\alpha_0} \rangle = \langle \eta \theta^{*} \mid e_{\alpha_0} \rangle$, then 
  then with some constants $c_1>0$ and $c_2>0$
\begin{equation}
\label{aspol12}
\pi(r \cos \alpha , r \sin \alpha)\sim  c_1
\exp\Big(-r \langle e_\alpha \mid \zeta \theta^{**} \rangle \Big) +
c_2 \exp\Big(-r \langle e_\alpha  \mid \eta \theta^{*} \rangle \Big) \  \  \  r\to \infty, \alpha \to \alpha_0.
\end{equation}
\end{itemize}
}
\end{thm}
See Figure \ref{figthmprin} for the different cases.
(The arcs $\}a,  b\}$ or $\{a, b\{$ of $\mathcal{E}$ are those not passing through the origin where 
   the left  or the right end respectively  is excluded). 
\begin{figure}[hbtp]
\centering
\includegraphics[scale=0.8]{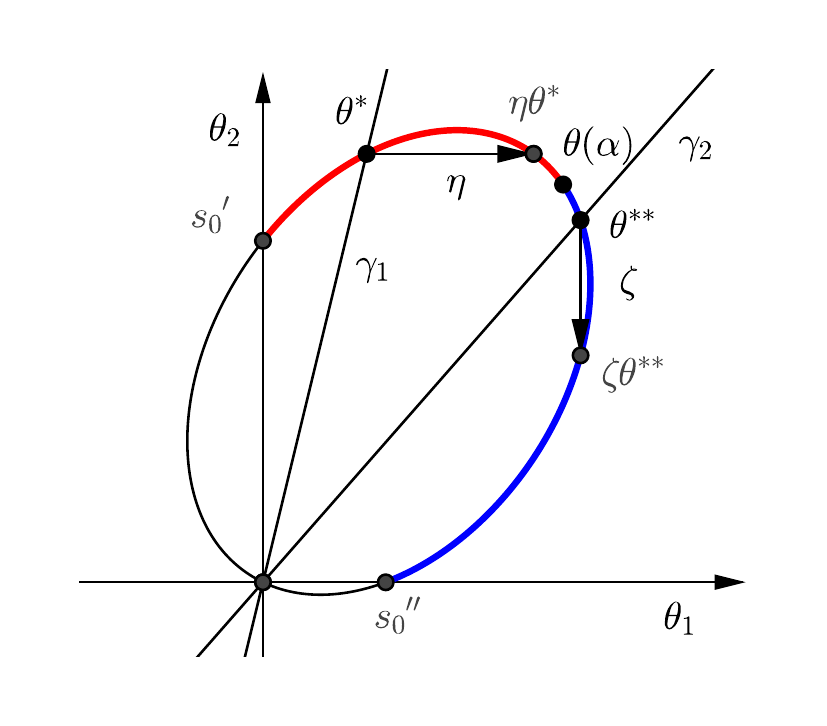}
\hfill
\includegraphics[scale=0.8]{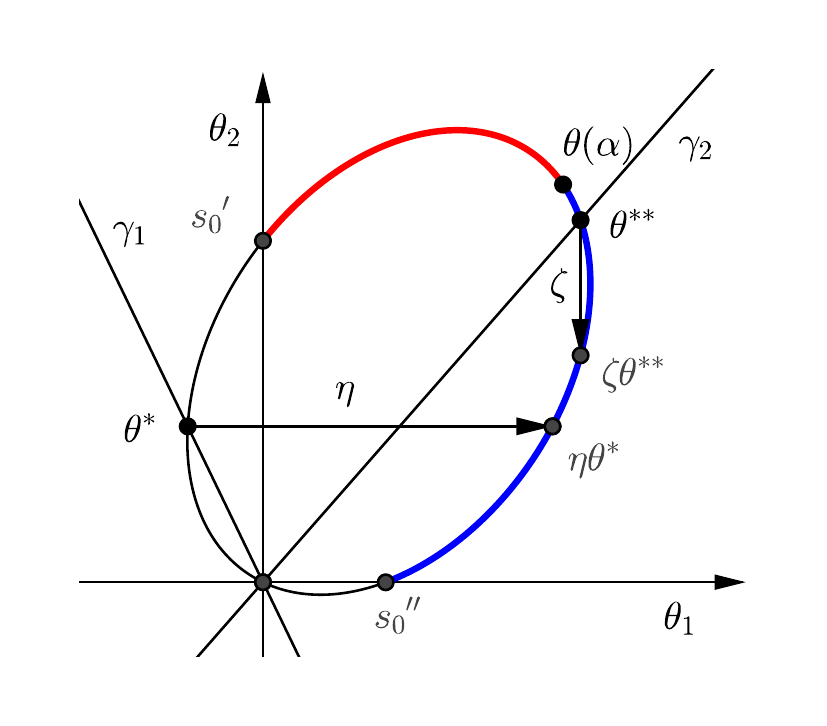}
\hfill
\includegraphics[scale=0.8]{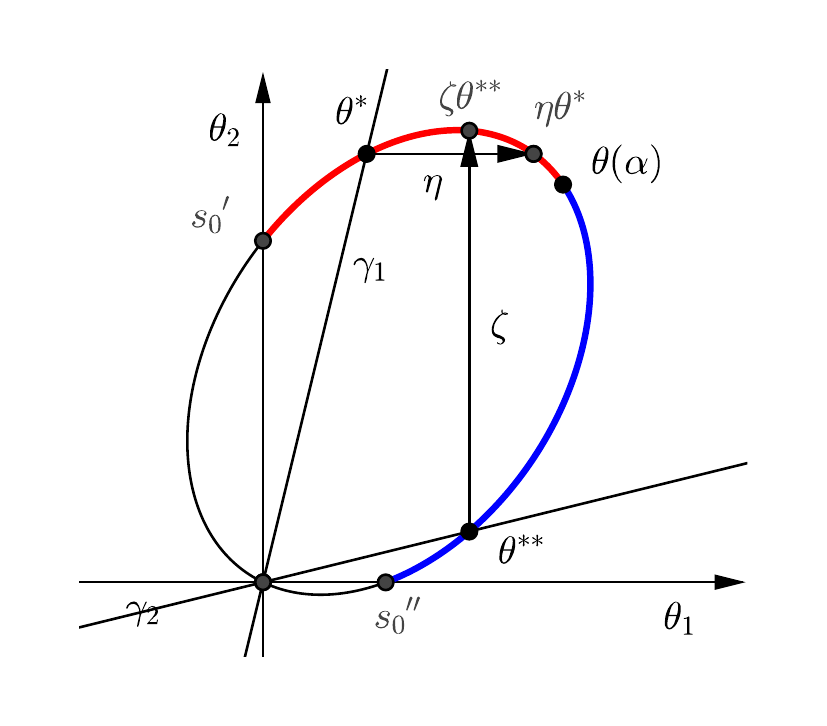}
\hfill
\includegraphics[scale=0.8]{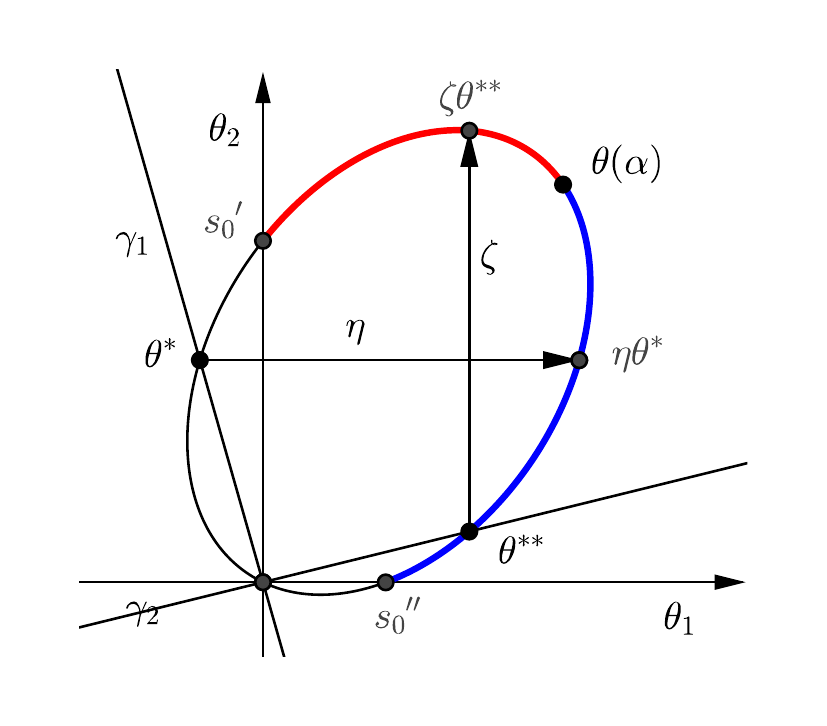}
\caption{Cases (1),(2),(3),(4)}
\label{figthmprin}
\end{figure}

  Let us  note  that the exponents in Theorem~\ref{thmmain} are the same as in the large deviation rate function found in \citep[Thm 3.2]{dai_stationary_2013}.
   The same phenomenon is observed for discrete random walks, cf. \cite{malyshev_asymptotic_1973} and \cite{ignatyuk_influence_1994}.

\subsection{Sketch of the analytic approach. Organization of the paper}
\label{subsec:sketch}

        The starting point of our approach is the main functional equation (\ref{maineq}) valid for any $\theta= (\theta_1, \theta_2) \in {\bf C}^2$ with
    ${\Re}\, \theta_1\leq 0$, ${\Re}\, \theta_2\leq 0$.
  The function $\gamma(\theta_1, \theta_2)$ in the left-hand side  is a polynomial of the second order of $\theta_1$ and $\theta_2$.
   The algebraic function $\Theta_1(\theta_2)$ defined by $\gamma(\Theta_1(\theta_2), \theta_2)\equiv 0$ is $2$-valued
  and its Riemann surface  ${\bf S}_{\theta_2}$  is of genus $0$. The same is true about  the $2$-valued algebraic function
  $\Theta_2(\theta_1)$ defined by $\gamma(\theta_1, \Theta_2(\theta_1))=0$ and its Riemann surface ${\bf S}_{\theta_1}$. 
   The surfaces ${\bf S}_{\theta_1}$ and ${\bf S}_{\theta_2}$ being  equivalent, we will consider just  one surface ${\bf S}$ 
   defined by the equation $\gamma(\theta_1, \theta_2)=0$
  with two different coverings.  Each point $s \in \s$ has two ``coordinates'' $(\theta_1(s), \theta_2(s))$, both of them are complex or 
  infinite and satisfy $\gamma(\theta_1(s), \theta_2(s))=0$.
    For any point $s=(\theta_1,\theta_2) \in \s$, there exits a unique point $s'=(\theta_1, \theta'_2)\in \s$ with the same first 
  coordinate and there exists a unique point $s''=(\theta''_1, \theta_2)\in \s$ with the same second coordinate.
  We say that $s'=\zeta s$, i.e. $s'$ and $s$ are related by  Galois automorphism $\zeta$ of ${\bf S}$
 that leaves untouched  the first coordinate, and that  $s''=\eta s$, i.e. $s''$ and $s$ are related by  Galois automorphism $\eta$ 
of ${\bf S}$ that leaves untouched  the second coordinate. Clearly $\zeta^2=Id$, $\eta^2=Id$ and the branch points of $\Theta_1(\theta_2)$
 and of $\Theta_2(\theta_1)$ are fixed points of $\zeta$ and $\eta$ respectively.
The ellipse $\mathcal{ E}$ is the set of points of ${\bf S}$ where both ``coordinates'' are real.
  The construction of ${\bf S}$ and definition of Galois automorphisms are carried out in Section~\ref{riemannsurfaceS}.
   
    Next, unknown functions $\varphi_1(\theta_2)$ and $\varphi_2(\theta_1)$ are lifted in the domains of $\s$ where 
   $\{s \in \s : {\Re}\, \theta_2(s)\leq 0\}$   and   $\{s \in \s : {\Re}\, \theta_1(s)\leq 0\}$ respectively.
   The intersection of these domains  on ${\bf S}$  is non-empty, both $\varphi_2$ and $\varphi_1$ are well defined in it. 
   Since for any $s=(\theta_1(s),\theta_2(s))\in \s$ we have $\gamma(\theta_1(s),\theta_2(s))=0$,
       the main functional equation  (\ref{maineq}) implies: 
  $$ \gamma_1(\theta_1(s), \theta_2(s)  ) \phi_1(\theta_2(s))+\gamma_2(\theta_1(s), \theta_2(s)  )\phi_2(\theta_1(s))=0     \  \  \forall 
s \in \s,{\Re}\,\theta_1(s)\leq 0, {\Re}\,\theta_2(s)\leq 0.$$
    Using this relation,  Galois automorphisms and the facts that  $\varphi_1$ and $\varphi_2$ depend just on one  ''coordinate''
  ($\phi_1$ depends on $\theta_2$ and $\phi_2$ on $\theta_1$ only),
  we continue $\varphi_1$ and $\varphi_2$ explicitly as meromorphic on the whole of $\s$.
  This meromorphic continuation procedure is the crucial step of our approach, it  is the subject of Section~\ref{meromcontinuation}.    
    It allows to recover $\varphi_1$ and $\varphi_2$ on the complex plane 
  as multivalued functions and determines all poles of all its branches. 
   Namely, it shows that  poles of $\phi_1$ and $\phi_2$ may be only at images of zeros of $\gamma_1$ and $\gamma_2$
  by automorphisms $\eta$ and $\zeta$ applied several times.
 We are in particular interested in the poles of their first (main) branch, 
  and more precisely in the most ''important'' pole (from the asymptotic point of view, to be explained below), that turns out to be  
 at one of points $\zeta \theta^{**}$ or $\eta \theta^*$  defined above.  The detailed analysis of the "main" poles of $\phi_1$ and $\phi_2$ 
  is furnished in Section~\ref{mainpole}.
             
   Let us now turn to the asymptotic expansion of the density $\pi(x_1,x_2)$.
  Its Laplace transform comes from the right-hand side of the main equation (\ref{maineq}) divided by the kernel $\gamma(\theta_1,\theta_2)$.
   By inversion formula  the density $\pi(x_1,x_2)$ is then represented as a double integral on $\{\theta : {\Re}\, \theta_1={\Re}\, \theta_2=0\}$.
   In Section~\ref{invlaplacetransf},  using the residues of the function $\frac{1}{\gamma(\theta_1, \cdot)}$ or $\frac{1}{\gamma(\cdot, \theta_2)}$
      we transform this double integral  into a sum of two single integrals along 
 two cycles on $\s$, those where ${\Re}\, \theta_1(s)=0$ or ${\Re}\, \theta_2(s)=0$. 
   Putting 
 $(x_1,x_2)=r e_\alpha$  we get the representation of the density as a sum of two 
  single integrals along some contours on ${\bf S}$:
\begin{equation}
\label{121}
\pi(r e_\alpha)=
\frac{1}{2\pi \sqrt{\det \Sigma}} 
\big(
\int_{\mathcal{I}_{\theta_1}^+}
\frac{\varphi_2(s) \gamma_2(\theta(s))}{s}  e^{- r \langle \theta(s) \mid e_\alpha \rangle }
{\mathrm{d}s} +
\int_{\mathcal{I}_{\theta_2}^+}
\frac{\varphi_1(s) \gamma_1(\theta(s))}{s} e^{- r \langle \theta(s) \mid e_\alpha \rangle }
\mathrm{d}s \big) .
\end{equation}
   We would like to compute their asymptotic expansion as $r \to \infty$ and prove it 
  to be uniform for $\alpha$ fixed  in a small neighborhood $\mathcal{O}(\alpha_0)$,  $\alpha_0\in ]0, \pi/2[$.

  These two integrals are typical to apply the saddle-point method, see \cite{fedoryuk_asymptotic_1989,pemantle_analytic_2013}.
   The point $\theta(\alpha) \in \mathcal{E}$ defined above is  the saddle-point for both of them, this is the subject of Section \ref{subsec:saddle}. 
 The integration contours on ${\bf S}$ are then shifted to  new ones 
  $\Gamma_{\theta_1, \alpha}$ and $\Gamma_{\theta_2,\alpha}$ which are 
       constructed in such a way that they pass  through the saddle-point  $\theta(\alpha)$,
   follow the steepest-descent curve in its neighborhood  $\mathcal{O}(\theta(\alpha))$
  and are ``higher'' than the saddle-point  w.r.t. the level curves of the function $\langle \theta(s) \mid e_\alpha  \rangle$ outside 
 $\mathcal{O}(\theta(\alpha))$, see  Section \ref{subsec:shiftingcontours}.  Applying Cauchy Theorem,  the density is finally represented 
  as a  sum of integrals along these new contours 
  and the sum of residues at poles of the integrands we encounter deforming the initial  ones: 
\begin{multline}
 \label{122}  
 \pi(r e_\alpha )= \sum_{p \in {\cal P'}_\alpha} {\rm res}_p \phi_2(\theta_1(s))
\frac{\gamma_2(p)}{\sqrt{d(\theta_1(p))}} e^{-r \langle \theta(p) \mid e_\alpha \rangle } +
     \sum_{p \in {\cal P''}_\alpha} {\rm res}_p \phi_1(\theta_2(s))
\frac{\gamma_1(p)}{\sqrt{\tilde d(\theta_2(p))}} e^{-r \langle \theta(p) \mid e_\alpha \rangle }  
\\
\frac{1}{2\pi \sqrt{\det \Sigma}} 
\left( 
\int_{\Gamma_{\theta_1, \alpha} }
\frac{\varphi_2(s) \gamma_2(\theta(s))}{s}  e^{- r \langle \theta(s) \mid e_\alpha \rangle }
{\mathrm{d}s} +
\int_{ \Gamma_{\theta_2, \alpha}   }
\frac{\varphi_1(s) \gamma_1(\theta(s))}{s} e^{- r \langle \theta(s) \mid e_\alpha \rangle }
\mathrm{d}s \right).
\end{multline}
   Here ${\cal P}'_\alpha$ (resp. ${\cal P''}_\alpha$) is  the set of  poles of the first order of $\phi_1$ (resp. $\phi_2$)
that are found  when  shifting  the initial contour $\mathcal{I}_{\theta_1}^+$ to the new one 
   $\Gamma_{\theta_1,\alpha}$ (resp. $\mathcal{I}_{\theta_2}^+   $  to  $ 
\Gamma_{\theta_2, \alpha} $), all of them are on the arc $\{s'_0, \theta(\alpha)\}$ (resp. $\{\theta(\alpha), s''_0\}$) 
    of ellipse $\mathcal{E}$.

   The asymptotic expansion of integrals along $\Gamma_{\theta_1,\alpha }$  and 
  $\Gamma_{\theta_2,\alpha}$ is made explicit by the standard saddle-point  method 
  in Section \ref{subsec:asymptintalongshift}.
   The set of poles ${\cal P}'_\alpha \cup {\cal P''}_\alpha$ is analyzed in Section \ref{subsec:contributionpole} .  
       In Case (1) of Theorem \ref{thmmain}  this set is empty, thus the asymptotic expansion of the density 
  is determined by the saddle-point, its first term is given in Theorem \ref{thmmain}.
     In Cases (2), (3) and (4) this set is not empty.  The  residues at  poles 
   over  ${\cal P}'_\alpha \cup {\cal P''}_\alpha $  
   in (\ref{122})  bring  all more important contribution to the asymptotic expansion of $\pi(r e_\alpha)$
    than the integrals  along  $\Gamma_{\theta_1,\alpha}$  and 
   $\Gamma_{\theta_2,\alpha}$.  
       Taking into  account the monotonicity of function $\langle \theta  \mid  e_\alpha  \rangle$  
on the arcs  $\{s''_0, \theta(\alpha)\}$ and on $\{\theta(\alpha), s'_0\}$, they can be ranked in order of their importance: clearly, 
    the term associated with a pole $p'$  is more important than the one with  $p''$ if 
  $\langle p' \mid e_\alpha  \rangle <   \langle  p'' \mid  e_\alpha  \rangle$.
  In  Case (2) (resp. (3)) the most important pole  is   $\zeta \theta^{**}   $ (resp. $\eta \theta^* $), as announced in Theorem \ref{thmmain}.
    In  Case (4) the most important of them is among  $\zeta \theta^{**}   $  and  $ \eta \theta^{*} $, as stated in Theorem \ref{thmmain} as well.
       The expansion of integrals in (\ref{122})  along $\Gamma_{\theta_1,\alpha}$  and $\Gamma_{\theta_2,\alpha} $ 
  via the saddle-point method provides all smaller asymptotic terms than those coming from the  poles.
   Section \ref{sec:asymptexpansion} is devoted to the results: they are stated from two points of view in Sections \ref{subsec:givenangle} and \ref{subsec:givenparameters} respectively. 
  First,  given an angle $\alpha_0$, we find  
  the  uniform asymptotic expansion of the density $\pi(r \cos (\alpha), r \sin (\alpha))$
   as $r \to \infty$ and $\alpha \in \mathcal{O}(\alpha_0)$ 
  depending on parameters $(\Sigma, \mu, R)$: Theorems \ref{thmresults1} --\ref{thmresults4} of Section \ref{subsec:givenangle}
  state it in all cases of parameters (1)--(4).
  Second, in Section \ref{subsec:givenparameters}, given a set parameters $(\Sigma, \mu, R)$, we compute the asymptotics of the density for all angles $\alpha_0 \in ]0, \pi/2[$,  see Theorems \ref{thmresultsnew1}-\ref{thmresultsnew3}.

\smallskip 

\noindent{\bf Remark.} 
     The constants mentioned  in Theorem \ref{thmmain} and all those in asymptotic expansions of Theorems \ref{thmresults1}--\ref{thmresultsnew3}
  are  specified in terms of functions $\varphi_1$ and $\varphi_2$.   In the  present paper we leave unknown these functions
  in their initial domains of definition although we carry out explicitly  their meromorphic continuation procedure and  find all their poles.
   In  \cite{franceschi_explicit_2017} the first author and K.~Raschel make explicit  these functions solving some boundary value problems.
  This determines the constants in asymptotic expansions in Theorems \ref{thmmain}, \ref{thmresults1}--\ref{thmresultsnew3} .

\medskip   

\noindent{\bf Future works.}
    The case of parameters such that $\zeta\theta^{**}=\theta(\alpha)$ and
$\eta \theta^{*} \not \in \{s'_0, \theta(\alpha)\{$ or the case  such that
   $\eta \theta^*=\theta(\alpha)$   and
$\zeta \theta^{**} \not\in \{s''_0, \theta(\alpha)\{$
     are not treated in  Theorem \ref{thmmain}.
Theorem \ref{thmresults4} gives a partial result but not at all as satisfactory as in all other cases.
 In fact, in these cases the saddle-point $\theta(\alpha)$ coincides 
 with  the  ``main'' pole of $\phi_1$ or $\phi_2$.  
 The analysis  is then reduced to a technical problem of computing the asymptotics of an integral whenever the saddle-point coincides with a pole
  of the integrand or approaches to it. We leave it for the future work. 
    
 In the cases $\alpha=0$ and $\alpha=\pi/2$,
   the tail asymptotics of the boundary measures $\nu_1$ and $\nu_2$ has been found in \cite{dai_stationary_2013}
  and the constants have been specified in \cite{franceschi_explicit_2017}.
  It would be also possible to find the asymptotics of $\pi(r \cos \alpha, r \sin \alpha)$ where $r \to \infty$ and $\alpha \to 0$ or $\alpha \to \pi/2$.      
   This problem is reduced  to obtaining the asymptotics of an integral when the saddle-point $\theta(0)$ or $\theta(\pi/2)$ coincides 
  with a branch point of the integrand $\phi_1$ or $\phi_2$. It can be solved by the same methods as in \cite{kourkova_random_2011} for discrete random walks.

\section{Riemann surface \texorpdfstring{$\mathbf{S}$}{S}}
\label{riemannsurfaceS}

\subsection{Kernel \texorpdfstring{$\gamma(\theta_1,\theta_2)$}{gamma}}

     The kernel of the main functional equation  
$$
\gamma(\theta_1,\theta_2)=
\frac{1}{2}(\sigma_{11} \theta_1^2 + \sigma_{22} \theta_2^2 +2\sigma_{12}\theta_1\theta_2)
+
\mu_1\theta_1+\mu_2\theta_2
$$
  can be written as
$$
\gamma(\theta_1, \theta_2)= \tilde{a}(\theta_2)\theta_1^2+\tilde{b}(\theta_2)\theta_1+\tilde{c}(\theta_2)
=a(\theta_1)\theta_2^2+b(\theta_1)\theta_2+c(\theta_1)
$$
where
\begin{center}
\begin{tabular}{ccc}
$\tilde{a}(\theta_2)=\frac{1}{2}\sigma_{11}$, & $\tilde{b}(\theta_2)=\sigma_{12}\theta_2+\mu_1$, & $\tilde{c}(\theta_2)=\frac{1}{2}\sigma_{22}\theta_2^2+\mu_2\theta_2$, \\ 
$a(\theta_1)=\frac{1}{2}\sigma_{22}$, & $b(\theta_1)=\sigma_{12}\theta_1+\mu_2$, & $c(\theta_1)=\frac{1}{2}\sigma_{11}\theta_1^2+\mu_1\theta_1$. \\ 
\end{tabular} 
\end{center}
The equation $\gamma(\theta_1, \theta_2)\equiv 0$ defines a two-valued algebraic function $\Theta_1(\theta_2)$  such that 
 $\gamma(\Theta_1(\theta_2),\theta_2)\equiv 0$
 and a two-valued algebraic function 
 $\Theta_2(\theta_1)$ such that $\gamma(\theta_1, \Theta_2(\theta_1)\equiv 0$.
    These functions have two branches:
\begin{center}
\begin{tabular}{cc}
$\Theta_1^+(\theta_2)=\frac{-\tilde{b}(\theta_2)+\sqrt{\tilde{d}(\theta_2)}}{2\tilde{a}(\theta_2)}$,
&
$\Theta_1^-(\theta_2)=\frac{-\tilde{b}(\theta_2)-\sqrt{\tilde{d}(\theta_2)}}{2\tilde{a}(\theta_2)}$,
\end{tabular}
\end{center}
and
\begin{center}
\begin{tabular}{cc}
$\Theta_2^+(\theta_1)=\frac{-b(\theta_1)+\sqrt{d(\theta_1)}}{2a(\theta_1)}$,
&
$\Theta_2^-(\theta_1)=\frac{-b(\theta_1)-\sqrt{d(\theta_1)}}{2a(\theta_1)}$.
\end{tabular}
\end{center}
   where 
\begin{center}
\begin{tabular}{c}
$\tilde{d}(\theta_2)=\theta_2^2(\sigma_{12}^2-\sigma_{11}\sigma_{22})+2\theta_2(\mu_1\sigma_{12}-\mu_2\sigma_{11})+\mu_1^2$,
 \\ $d(\theta_1)=\theta_1^2(\sigma_{12}^2-\sigma_{11}\sigma_{22})+2\theta_1(\mu_2\sigma_{12}-\mu_1\sigma_{22})+\mu_2^2$. \\ 
\end{tabular} 
\end{center}
 The  discriminant ${d}(\theta_1)$  (resp. $\tilde{d}(\theta_2)$) has two zeros $\theta_1^+$,  $\theta_1^{-}$ 
  (resp. $\theta_2^{+}$ and $\theta_2^{-}$)  that  are 
   both real and of opposite signs:
\begin{center}
\begin{tabular}{cc}
$\theta_1^-= \frac{(\mu_2\sigma_{12}-\mu_1\sigma_{22}) - \sqrt{D_1}}{\det{\Sigma}}<0$,
&
$\theta_1^+= \frac{(\mu_2\sigma_{12}-\mu_1\sigma_{22}) + \sqrt{D_1}}{\det{\Sigma}}>0$,
\\
\end{tabular}
\end{center}
\begin{center}
\begin{tabular}{cc}
$\theta_2^- = \frac{(\mu_1\sigma_{12}-\mu_2\sigma_{11}) - \sqrt{D_2}}{\det{\Sigma}} <0$,
&
$\theta_2^+ = \frac{(\mu_1\sigma_{12}-\mu_2\sigma_{11}) + \sqrt{D_2}}{\det{\Sigma}} >0$,
\\
\end{tabular}
\end{center}
 with notations 
 $D_1 =(\mu_2\sigma_{12}-\mu_1\sigma_{22})^2 +\mu_2^2 \det{\Sigma}$
  and
 $D_2 =(\mu_1\sigma_{12}-\mu_2\sigma_{11})^2 +\mu_1^2 \det{\Sigma}$.
  Then  $\Theta_2(\theta_1)$  (resp. $\Theta_1(\theta_2)$) has two branch  points: $\theta_1^-$ and 
 $\theta_1^+$   (resp  $\theta_2^{-}$ and $\theta_2^{+}$). 
We can compute: 
$$
\Theta_2^\pm(\theta_1^-)
=\frac{\mu_1\sigma_{12}-\mu_2\sigma_{11} + \frac{\sigma_{12}}{\sigma_{22}} \sqrt{D_1}}{\det \Sigma}, \   \   
\Theta_2^\pm(\theta_1^+)
=\frac{\mu_1\sigma_{12}-\mu_2\sigma_{11} - \frac{\sigma_{12}}{\sigma_{22}} \sqrt{D_1}}{\det \Sigma},$$ 
 $$\Theta_1^\pm(\theta_2^-)
 = 
\frac{\mu_2\sigma_{12}-\mu_1\sigma_{22} + \frac{\sigma_{12}}{\sigma_{11}} \sqrt{D_2}}{\det \Sigma},\    \ 
 \Theta_1^\pm(\theta_2^+)
 = 
\frac{\mu_2\sigma_{12}-\mu_1\sigma_{22} - \frac{\sigma_{12}}{\sigma_{11}} \sqrt{D_2}}{\det \Sigma}.$$
    Furthermore,  $d(\theta_1)$  (resp. $\tilde d(\theta_2)$)
 being positive on $]\theta_1^-, \theta_1^+[$ 
(resp.   $]\theta_2^-, \theta_2^+[$)  
and negative on ${\bf R} \setminus [\theta_1^-, \theta_1^+]$ 
(resp.  ${\bf R}\setminus [\theta_2^-, \theta_2^+]$) , 
both branches $\Theta_2^{\pm}(\theta_1)$
  (resp. $\Theta_1^\pm (\theta_2)$) 
 take real values on $[\theta_1^-, \theta_1^+]$ 
(resp.   $[\theta_2^-, \theta_2^+]$)  
 and complex values on   ${\bf R} \setminus [\theta_1^-, \theta_1^+]$
(resp.  ${\bf R}\setminus [\theta_2^-, \theta_2^+]$).

\begin{figure}[hbtp]
\centering
\includegraphics[scale=0.7]{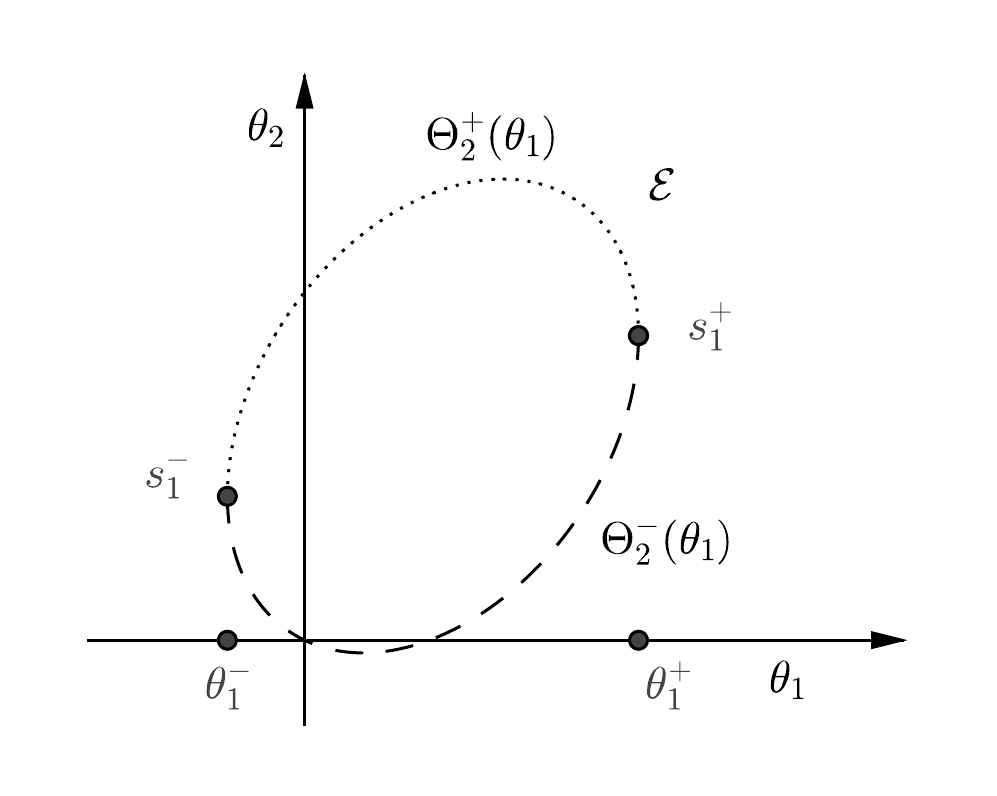}
\includegraphics[scale=0.7]{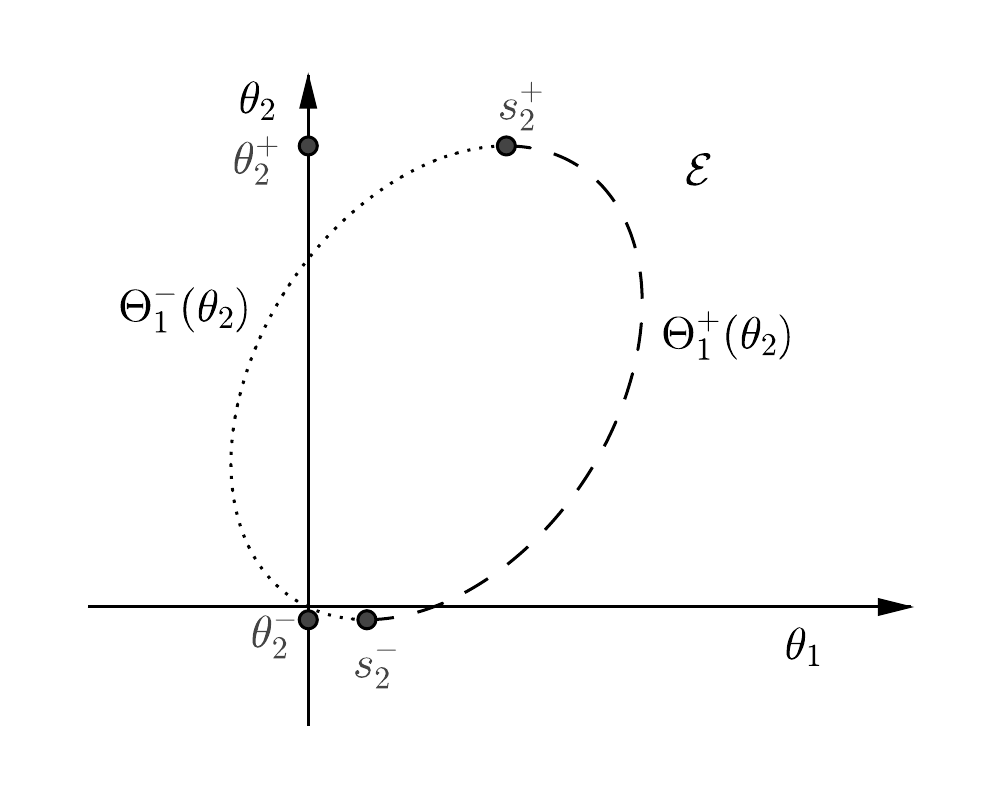}
\caption{Functions $\Theta_2^\pm (\theta_1)$ and $\Theta_1^\pm (\theta_2)$ on $[\theta_1^-,\theta_1^+]$ and $[\theta_2^-,\theta_2^+]$}
\label{fonctionellipse}
\end{figure}

\subsection{Construction of the Riemann surface \texorpdfstring{${\bf S}$}{S}}
\label{constructionS}

We now construct the Riemann surface $\mathbf{S}$ of the algebraic function $\Theta_2(\theta_1)$. For this purpose we take two Riemann spheres $\mathbb{C}_{\theta_1}^1\cup\{\infty\}$ and $\mathbb{C}_{\theta_1}^2\cup\{\infty^\prime\}$, say $\mathbf{S}_{\theta_1}^1$ and $\mathbf{S}_{\theta_1}^2$, cut along $([-\infty^{(\prime)},\theta_1^{-}] \cup [\theta_1^{+},\infty^{(\prime)}])$, and we glue them together along the borders of these cuts, joining the lower border of the cut on $\mathbf{S}_{\theta_1}^1$ to the upper border of the same cut on $\mathbf{S}_{\theta_1}^2$ and vice versa. This procedure can be viewed as gluing together two half-spheres, see Figure \ref{constrsurfa}. The resulting surface $\mathbf{S}$ is homeomorphic to a sphere (i.e., a compact Riemann surface of genus 0) and is projected on the Riemann sphere $\mathbb{C}\cup\{\infty\}$ by a canonical covering map $h_{\theta_1}: \mathbf{S} \to \mathbb{C}\cup\{\infty\}$.
In a standard way, we can lift the function $\Theta_2(\theta_1)$ to $\mathbf{S}$, by setting $\Theta_2(s)=\Theta_2^+(h_{\theta_1}(s))$ if $s\in\mathbf{S}_{\theta_1}^1 \subset \mathbf{S}$ and $\Theta_2(s)=\Theta_2^-(h_{\theta_1}(s))$ if $s\in\mathbf{S}_{\theta_1}^2 \subset \mathbf{S}$. 

In a similar way one constructs the Riemann surface of the function $\Theta_1(\theta_2)$, by gluing together two copies $\mathbf{S}_{\theta_2}^1$ and $\mathbf{S}_{\theta_2}^2$ of the Riemann sphere $\mathbf{S}$ cut along $([-\infty^{(\prime)},\theta_2^{-}] \cup [\theta_2^{+},\infty^{(\prime)}])$.
We obtain again a surface homeomorphic to a sphere where we  lift function  $\Theta_1(\theta_2)$.

\begin{figure}[hbtp]
\centering
\includegraphics[scale=1.2]{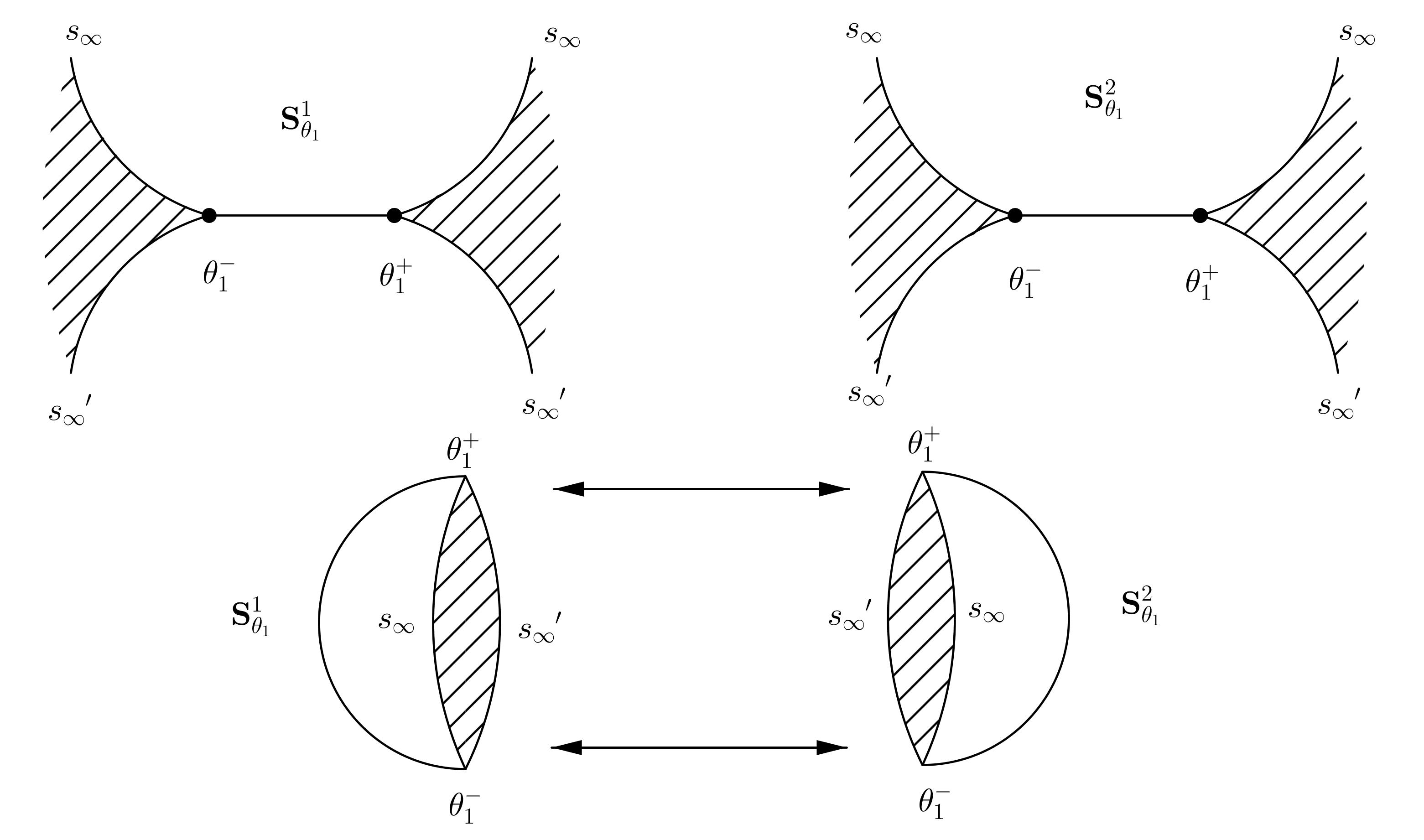}
\caption{Construction of the Riemann surface $\s$}
\label{constrsurfa}
\end{figure}

Since the Riemann surfaces of $\Theta_1(\theta_2)$ and $\Theta_2(\theta_1)$ are  equivalent, we can and will work on a {\it single} Riemann surface $\mathbf{S}$, with two different covering maps $h_{\theta_1}, h_{\theta_2} : \mathbf{S} \to \mathbb{C}\cup\{\infty\}$. Then, for $s\in\mathbf{S}$, we set $\theta_1(s)=h_{\theta_1}(s)$ and $\theta_2(s)=h_{\theta_2}(s)$. We will often represent a point $s\in\mathbf{S}$ by the pair of its \textit{coordinates} $(\theta_1(s),\theta_2(s))$. These coordinates are of course not independent, because the equation $\gamma(\theta_1(s),\theta_2(s))=0$ is valid for any $s\in\mathbf{S}$.  
    One can see ${\bf S}$  with points 
$
s_1^\pm=(\theta_1^\pm,\frac{\mu_1\sigma_{12}-\mu_2\sigma_{11} \mp \frac{\sigma_{12}}{\sigma_{22}} \sqrt{D_1}}{\det \Sigma})$,
$
s_2^\pm=(\frac{\mu_2\sigma_{12}-\mu_1\sigma_{22} \mp \frac{\sigma_{12}}{\sigma_{11}} \sqrt{D_2}}{\det \Sigma},\theta_2^\pm)$,
$
s_\infty=(\infty,\infty)$,
$
s_{\infty^\prime}=(\infty^\prime,\infty^\prime)$
on Figure \ref{realpointss}.
 It is the union of ${\bf S}_{\theta_1}^1$ and ${\bf S}_{\theta_1}^2$ glued along the contour 
${\cal R}_{\theta_1}=\{s : \theta_1(s) \in {\bf R} \setminus ]\theta_1^-, \theta_1^+[\}$ that goes from $s_\infty$ to $s_{\infty^\prime}$
  via $s_1^-$  and back to $s_\infty$ via $s_1^+$. 
It is also the union   of ${\bf S}_{\theta_2}^1$ and ${\bf S}_{\theta_2}^2$ glued along the contour 
${\cal R}_{\theta_2}=\{s : \theta_2(s) \in {\bf R} \setminus ]\theta_2^-, \theta_2^+[\}$. This contour
   goes from $s_\infty$ to $s_{\infty^\prime}$ and back as well,  but via $s_2^-$  and $s_2^+$. 
     Let ${\cal E}$ be the set of points of ${\bf S}$ where both coordinates $\theta_1(s)$ and $\theta_2(s)$ are real.
  Then 
$$
\mathcal{E}=\{ s\in\mathbf{S}:\theta_1(s)\in[\theta_1^-,\theta_1^{+}] \}
=
\{ s\in\mathbf{S}:\theta_2(s)\in[\theta_2^-,\theta_2^{+}] \}.
$$
   One can see $\mathcal{ E}$ on Figures \ref{fonctionellipse} and \ref{realpointss}, it contains all branch  points $s_1^{\pm}$ and $s_2^{\pm}$.

\begin{figure}[hbtp]
\centering
\includegraphics[scale=1]{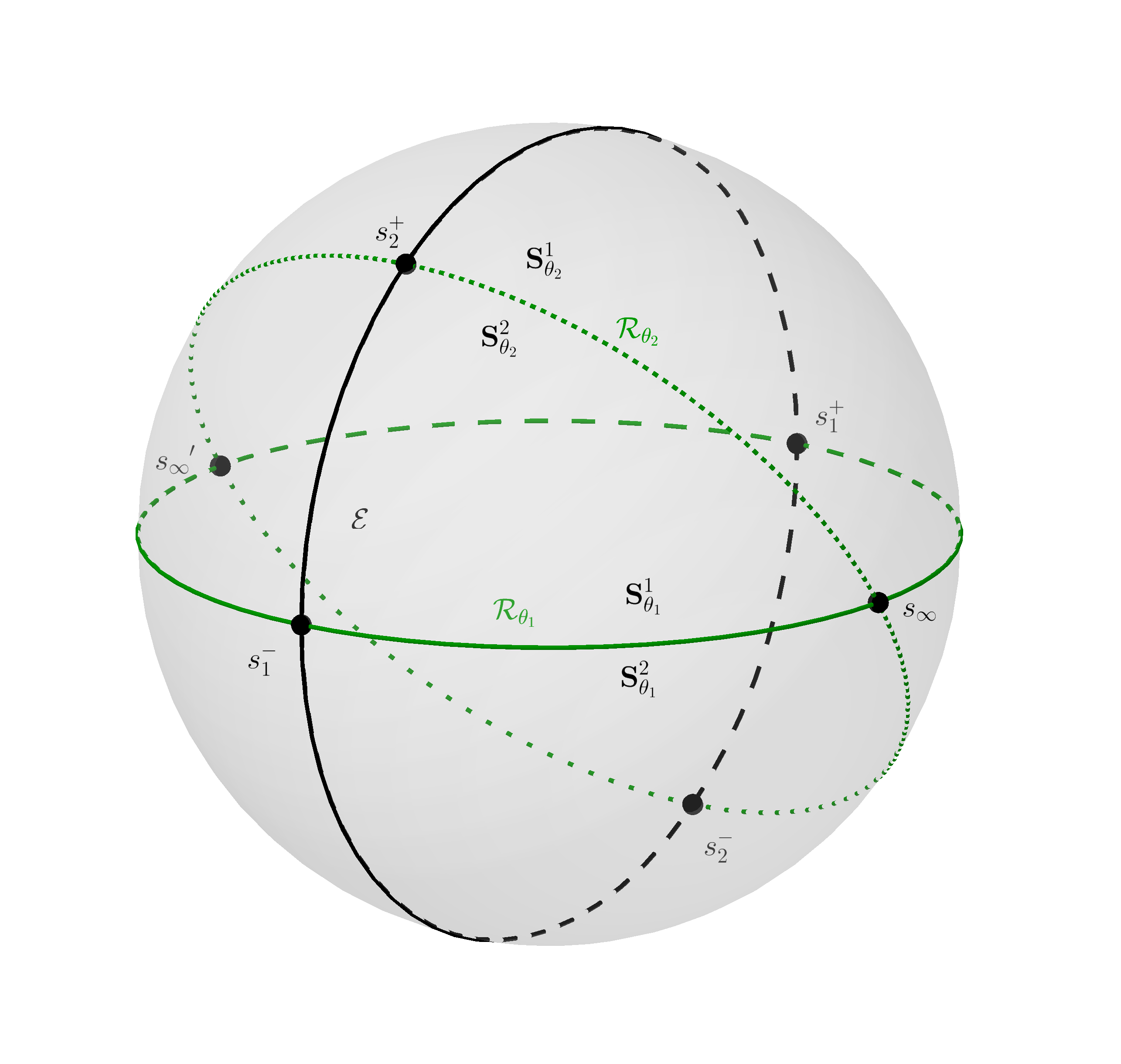}
\caption{Points of $\s$ with $\theta_1(s)$ or $\theta_2(s)$ real}
\label{realpointss}
\end{figure}

\subsection{Galois automorphisms \texorpdfstring{$\eta$}{eta} and  \texorpdfstring{$\zeta$}{zeta}}

Now we need to introduce Galois automorphisms on ${\bf S}$.
 For any $s\in \mathbf{S} \setminus s_1^{\pm}$   there is a unique $s^\prime \ne s \in \mathbf{S} \setminus s_1^{\pm}$ such that 
$\theta_1(s)=\theta_1(s^\prime)$. Furthermore, if $s \in {\mathbf{S}}_{\theta_1}^1$ then $s^\prime \in {\mathbf{S}}_{\theta_1}^2$ and vice versa. On the other hand, whenever $s=s_1^-$ or $s=s_1^+$  (i.e. $\theta_1(s)=\theta_1^\pm$  is one of branch  points of $\Theta_2(\theta_1)$) 
 we have $s=s^\prime$. Also, since $\gamma(\theta_1(s),\theta_2(s))=0$, $\theta_2(s)$ and $\theta_2(s^\prime)$ represent both values of function $\Theta_2(\theta_1)$ at $\theta_1=\theta_1(s)=\theta_1(s^\prime)$. By Vieta's theorem we obtain $\theta_2(s)\theta_2(s^\prime)=\frac{c(\theta_1(s))}{a(\theta_1(s))}$.

Similarly, for any $s\in\mathbf{S} \setminus s_2^{\pm}$, there exists a unique $s^{\prime\prime} \neq s \in\mathbf{S}
 \setminus s_2^{\pm} $ such that $\theta_2(s)=\theta_2(s^{\prime\prime})$. If $s\in {\mathbf{S}}_{\theta_2}^1$ then $s^\prime \in {\mathbf{S}}_{\theta_2}^2$ and vice versa. On the other hand, if $s=s_2^-$  or $s=s_2^+$  (i.e. $\theta_2(s)=\theta_2^\pm$ is one of branch  points of $\Theta_1(\theta_2)$) 
    we have $s=s^{\prime\prime}$. Moreover, since $\gamma(\theta_1(s),\theta_2(s))=0$, $\theta_1(s)$ and $\theta_1(s^{\prime\prime})$ give both values of function $\Theta_1(\theta_2)$ at $\theta_2=\theta_2(s)=\theta_2(s^{\prime\prime})$. Again, by Vieta's theorem $\theta_1(s)\theta_1(s^{\prime\prime})=\frac{\tilde{c}(\theta_2(s))}{\tilde{a}(\theta_2(s))}$.

With the previous notations we now define the mappings $\zeta:\mathbf{S}\to\mathbf{S}$ and $\eta:\mathbf{S}\to\mathbf{S}$ by
\begin{center}
$\left\{
\begin{tabular}{cc}
$\zeta s = s^\prime$ & if $\theta_1(s)=\theta_1(s^\prime)$,
\\ 
$\eta s = s^{\prime\prime}$ & if $\theta_2(s)=\theta_2(s^{\prime\prime})$
\\ 
\end{tabular} 
\right. $
\end{center}
Following \cite{malyshev_sluchainye_1970} we call them \textit{Galois automorphisms} of $\mathbf{S}$. Then $\zeta^2=\eta^2=\text{Id}$, and
\begin{center}
\begin{tabular}{cc}
$\theta_2(\zeta s)= \frac{c(\theta_1(s))}{a(\theta_1(s))} \frac{1}{\theta_2(s)}$,
&
$\theta_1(\eta s)=\frac{\tilde{c}(\theta_2(s))}{\tilde{a}(\theta_2(s))} \frac{1}{\theta_1(s)}$.
\\
\end{tabular}
\end{center}
Points $s_1^-$ and $s_1^+$  (resp. $s_2^-$ and $s_2^+$)  are fixed points for $\zeta$ (resp. $\eta$). 

  It is known that conformal automorphisms of a sphere (that can be identified to ${\bf C} \cup \infty$)
     are transformations of type $z\mapsto \frac{az+b}{cz+d}$ where $a,b,c,d$ are any complex numbers satisfying $ad-bc\neq 0$.
  The automorphisms $\zeta$ and $\eta$, which are conformal automorphisms of $\s$, have each two fixed points  and are involutions (because 
$\zeta^2=\eta^2=\text{Id}$).
  We can deduce from it that $\zeta$ (resp. $\eta$) is a symmetry w.r.t. the axis $A_1$ (resp. $A_2$) that passes through fixed points $s_1^-$ and $s_1^+$ (resp. $s_2^-$ and $s_2^+$). In other words $\zeta$ (resp. $\eta$)
   is a  rotation of angle $\pi$, around $D_1$  (resp. $A_2$), see Figure \ref{sphereaxes}.
  Let us draw the axis $A$  orthogonal to the plane generated by the axes $A_1$ and $A_2$ and passing through the intersection point of $A_1$ and $A_2$.
   We denote by  $\beta$ the angle between the axes $A_1$ and $A_2$. Automorphisms $\eta\zeta$ and $\zeta\eta$ are then {\it rotations of angle $ 2\beta$ and $-2\beta$ around the axis  $A$.} This axis  goes through  points $s_\infty$ and $s_{\infty\prime}$ which are fixed points for $\eta\zeta$ and $\zeta\eta$,  see Figure \ref{sphereaxes}.

In the particular case  $\Sigma=\text{Id}$,  we have $\eta\zeta=\zeta\eta$, the axes $A_1$ and $A_2$ are orthogonal. We deduce that $\beta=\frac{\pi}{2}$ and that $\eta\zeta$ and $\zeta\eta$ are symmetries w.r.t. the axis $A$.

\begin{figure}[hbtp]
\centering
\includegraphics[scale=1]{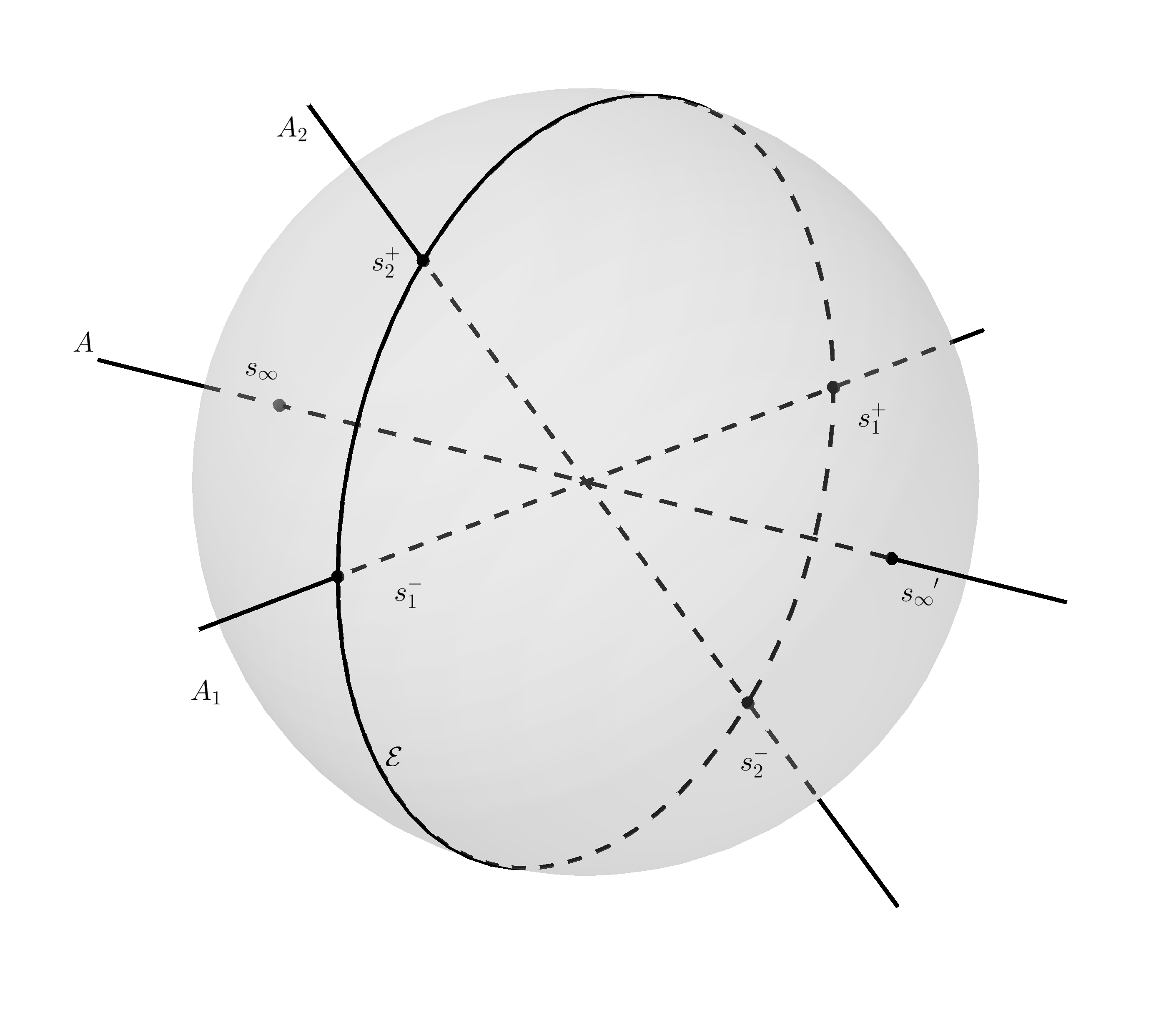}
\caption{Axes $A_1$, $A_2$ and $A$ of Galois automorphisms $\zeta$, $\eta$ and $\zeta\eta$ respectively}
\label{sphereaxes}
\end{figure}

\subsection{Domains of initial definition of \texorpdfstring{$\phi_1$}{phi1} and \texorpdfstring{$\phi_2$}{phi2} on  \texorpdfstring{${\bf S}$}{S}}  

    We would like to lift  functions  $\phi_1(\theta_2)$ and $\phi_2(\theta_1)$ on ${\bf S}$ naturally as 
$\phi_1(s)=\phi_1(\theta_2(s))$ and $\phi_2(s)=\phi_2(\theta_1(s))$. But it can not be done 
for all $s \in {\bf S}$,  $\phi_1(\theta_2)$ and $\phi_2(\theta_1)$ being not defined on the whole of ${\bf  C}$.  
 Nevertheless, we  are able to do it  for points $s$  where $\theta_2(s)$ or $\theta_1(s)$ respectively  have non-positive real parts.
  Therefore, in this section we study the domains on ${\bf S}$ where it holds true.   

 For any $\theta_1 \in {\bf C}$ with ${\R}(\theta_1)=0$, $\Theta_2(\theta_1)$ takes two values  
$\Theta_2^{\pm}(\theta_1)$. Let us observe that under assumption  that the second coordinate of the interior drift is negative, i.e\  $\mu_2<0$ we have  ${\R} \Theta_2^-(\theta_1)\leq 0$ and ${\R} \Theta_2^+(\theta_1)>0$. 
   Furthermore ${\R}\Theta_2^-(\theta_1)=0$ only at $\theta_1=0$, and then  $\Theta_2^-(\theta_1)=0$. 
  The domain 
$$ \Delta_1=\{s \in \s :   \ {\R}\theta_1(s)<0\}$$ 
     is simply connected and bounded by the contour  
$\mathcal{I}_{\theta_1}=\{s :   {\R}\theta_1(s)=0\}$.
 
   The contour  
$\mathcal{I}_{\theta_1}$ can be represented as the union of 
 $\mathcal{I}_{\theta_1}^-\cup \mathcal{I}_{\theta_1}^+$ , where 
  $\mathcal{I}_{\theta_1}^-=\{s :   {\R}\theta_1(s)=0, {\R} \theta_2(s) \leq 0\}$, 
    $\mathcal{I}_{\theta_1}^+=\{s :   {\R}\theta_1(s)=0, {\R} \theta_2(s) > 0\}$, see Figure~\ref{contIth1}.

The contour $\mathcal{I}_{\theta_1}^-$ goes from $s_\infty$ to $s_{\infty\prime}$ crossing the set of real points $\mathcal{  E}$
 at $s_0=(0,0)$, while   $\mathcal{I}_{\theta_1}^{+}$ goes from $s_\infty$ to $s_{\infty\prime}$ crossing  $\mathcal{  E}$
  at $s'_0=(0, -2 \frac{\mu_2}{\sigma_{22}})$ where the second coordinate is positive.

In the same way, under assumption that the first coordinate of the interior drift is negative, i.e. $\mu_1<0$, 
 for any $\theta_2 \in {\bf C}$ with ${\R}(\theta_2)=0$, $\Theta_1(\theta_2)$ takes two values  
$\Theta_1^{\pm}(\theta_2)$, where ${\R} \Theta_1^-(\theta_2)\leq 0$ and ${\R} \Theta_1^+(\theta_2)>0$,   
  moreover ${\R}\Theta_1^-(\theta_2)=0$ only at $\theta_2=0$, and then  $\Theta_1^-(\theta_2)=0$. 
The domain 
$$ \Delta_2=\{s \in \s :    {\R}\theta_2(s)<0\}$$ 
     is simply connected and bounded by the contour  
$\mathcal{I}_{\theta_2}=\{s :   {\R}\theta_2(s)=0\}$. 
   The contour  
$\mathcal{I}_{\theta_2}$ can be represented as the union of 
 $\mathcal{I}_{\theta_2}^-\cup \mathcal{I}_{\theta_2}^+$ , where 
  $\mathcal{I}_{\theta_2}^-=\{s :   {\R}\theta_2(s)=0, {\R} \theta_1(s) \leq 0\}$, 
    $\mathcal{I}_{\theta_2}^+=\{s :   {\R}\theta_2(s)=0, {\R} \theta_1(s) > 0\}$.
The contour $\mathcal{I}_{\theta_2}^-$ goes from $s_\infty$ to $s_{\infty\prime}$ crossing the set of real points $\mathcal{ E}$
 at $s_0=(0,0)$, while   $\mathcal{I}_{\theta_2}^+$ goes from $s_\infty$ to $s_{\infty\prime}$ crossing  $\mathcal{ E}$
  at $s''_
0=(-2 \frac{\mu_1}{\sigma_{11}}, 0)$, see Figure~\ref{contIth1}.

     Assume now that the interior drift has both coordinates negative, i.e. (\ref{mumu}).   From what said above, $\mathcal{I}_{\theta_1}^-\setminus s_0 \subset \Delta_2$ and 
  $\mathcal{I}_{\theta_2}^-\setminus s_0 \subset \Delta_1$. The intersection $\Delta_1 \cap \Delta_2$ consists of two connected 
  components, both bounded by  $\mathcal{I}_{\theta_1}^-
$ and $\mathcal{I}_{\theta_2}^-$. The union $\Delta_1 \cup \Delta_2$ is a connected domain, but not simply connected because of the point 
  $s_0$. The domain $\Delta_1 \cup \Delta_2 \cup s_0$ is  open, simply connected and bounded by  $\mathcal{I}_{\theta_1}^+
$ and $\mathcal{I}_{\theta_2}^+$, see Figure~\ref{contIth1}. We set $\Delta=\Delta_1\cup\Delta_2$.

    Note that in the cases  of stationary SRBM  with  drift $\mu$ having one of coordinates non-negative, the location of
  contours $\mathcal{I}_{\theta_1}^+$, $\mathcal{I}_{\theta_1}^-$, $\mathcal{I}_{\theta_2}^+$, $\mathcal{I}_{\theta_2}^-$
   on ${\bf S}$ is different. For example, assume that $\mu_2>0$. Then ${\R} \Theta_2^-(\theta_1)<0$ and ${\R} \Theta_2^+(\theta_1)\geq 0$,
    the contour $\mathcal{I}_{\theta_1}^-$ goes from $s_\infty$ to $s_{\infty\prime}$ crossing the set of real points $\mathcal{  E}$
 at $s'_0=(0, -2 \frac{\mu_2}{\sigma_{22}})$ where the second coordinate is negative, while   $\mathcal{I}_{\theta_1}^{+}$ goes from $s_\infty$ to $s_{\infty\prime}$ crossing  $\mathcal{  E}$
  at $s_0=(0, 0)$.  In order to shorten the number of cases and pictures, we restrict ourselves in this paper to the case (\ref{mumu}) of both coordinates of $\mu$  negative,  although all our methods work in these other cases as well.     
  
\begin{figure}[hbtp]
\centering
\includegraphics[scale=1]{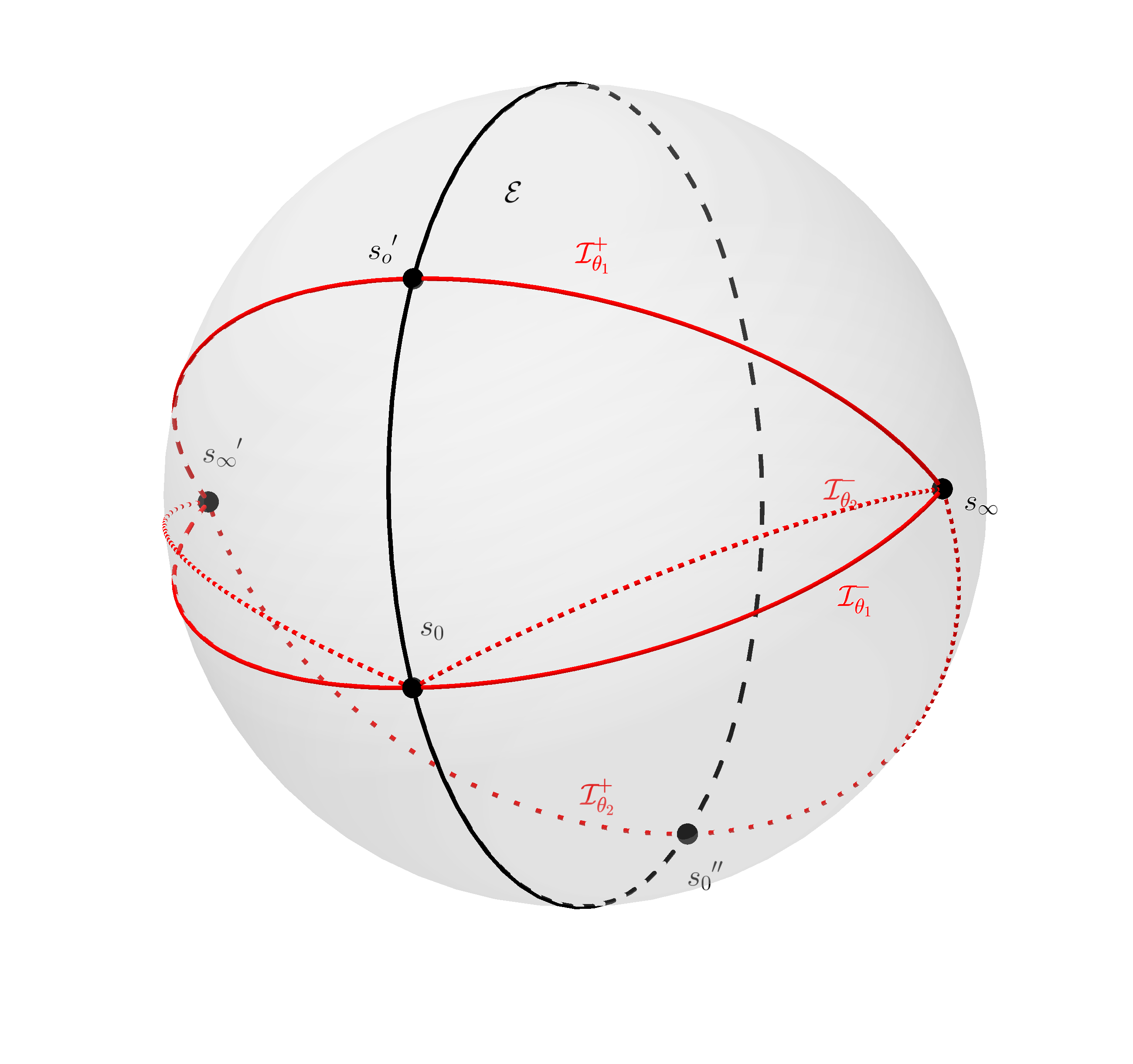}
\caption{Pure imaginary points of $\s$}
\label{contIth1}
\end{figure}

\subsection{Parametrization of \texorpdfstring{$\s$}{S}}\label{param}

      It is difficult to visualize on three-dimensional  sphere 
 different points, contours, automorphisms and domains  introduced above 
  that will be used in  future steps. For this reason  
  we propose here an explicit and practical parametrisation of $\s$.  Namely     
  we identify $\s$ to $\mathbb{C}\cup\{\infty\}$  and  in the next proposition we explicitly define $h_{\theta_1}$ and $h_{\theta_2}$ 
two recoveries introduced in Section \ref{constructionS}.
Such a parametrisation  allows to visualize better in two dimensions the sphere $\s\equiv \mathbb{C}\cup\{\infty\}$ and all sets we are interested in, as we can see in Figure \ref{recouvrS}.

\begin{prop}
We set the following covering maps
$$
\begin{array}{lrcl}
h_{\theta_1}: &  \mathbb{C}\cup\{\infty\}\equiv\s & \longrightarrow & \mathbb{C}\cup\{\infty\} 
\\
    & s & \longmapsto &  h_{\theta_1}(s)=\theta_1(s):=\frac{\theta_1^{+}+\theta_1^{-}}{2}+\frac{\theta_1^{+}-\theta_1^{-}}{4} (s+\frac{1}{s})
\end{array}
$$
and
$$
\begin{array}{lrcl}
h_{\theta_2}: &  \mathbb{C}\cup\{\infty\}\equiv\s & \longrightarrow & \mathbb{C}\cup\{\infty\} 
\\
    & s & \longmapsto &  h_{\theta_2}(s)=\theta_2(s):=\frac{\theta_2^{+}+\theta_2^{-}}{2}+\frac{\theta_2^{+}-\theta_2^{-}}{4} (\frac{s}{e^{i\beta}}+\frac{e^{i\beta}}{s}),
\end{array}
$$
where $$\beta= \arccos {-\frac{\sigma_{12}}{\sqrt{\sigma_{11}\sigma_{22}}}}.$$
The equation $\gamma(\theta_1(s),\theta_2(s))=0$ is valid for any $s\in\mathbf{S}$.
 Galois automorphisms can be written
\begin{center}
\begin{tabular}{cc}
$\zeta(s)=\frac{1}{s}$, 
& 
$\eta(s)=\frac{e^{2i\beta}}{s}$, \\ 
\end{tabular} 
\end{center}
and $\eta\zeta$ (resp. $\zeta\eta$) is a rotation around $s_\infty\equiv0$ of angle $2\beta$ (resp. $-2\beta$) according to counterclockwise direction.
\end{prop}

\begin{proof}
We set $ h_{\theta_1}(s)=\theta_1(s):=\frac{\theta_1^{+}+\theta_1^{-}}{2}+\frac{\theta_1^{+}-\theta_1^{-}}{4} (s+\frac{1}{s})$.
One can notice that $h_{\theta_1}(1)=\theta_1^+$, $h_{\theta_1}(-1)=\theta_1^-$, $h_{\theta_1}'(1)=0$, $h_{\theta_1}'(-1)=0$.
This parametrization is practical because it leads to a similar rational recovery $h_{\theta_2}$.
In order to make the equation $\gamma(\theta_1(s),\theta_2(s))=0$ valid for any $s\in\mathbf{S}$ we naturally set 
$$\theta_2(s)= \Theta_2^+ (\theta_1(s)):=\frac{-b(\theta_1(s))+\sqrt{d(\theta_1(s))}}{2a(\theta_1(s))}$$
and we are going to show that $\theta_2(s)=
\frac{\theta_2^{+}+\theta_2^{-}}{2}+\frac{\theta_2^{+}-\theta_2^{-}}{4} (\frac{s}{e^{i\beta}}+\frac{e^{i\beta}}{s})$ where $\beta= \arccos {-\frac{\sigma_{12}}{\sqrt{\sigma_{11}\sigma_{22}}}}$.
We note that $d(\theta_1(s))$ is the opposite of the square of a rational fraction
\begin{align*}
d(\theta_1(s))&=-\det \Sigma (\theta_1(s)-\theta_1^+)(\theta_1(s)-\theta_1^-)
=-\det \Sigma (\frac{\theta_1^{+}-\theta_1^{-}}{4})^2(-2+(s+\frac{1}{s}))(2+(s+\frac{1}{s}))
\\ &= -\det \Sigma
(\frac{\theta_1^+-\theta_1^-}{4})^2 (s-\frac{1}{s})^2 \leqslant 0.
\end{align*}
Then we have
\begin{equation}
\label{232}
\theta_2(s)= \Theta_2^+ (\theta_1(s))
:=
\frac{-\sigma_{12}\frac{\theta_1^{+}+\theta_1^{-}}{2}+\frac{\theta_1^{+}-\theta_1^{-}}{4} (s+\frac{1}{s})-\mu_2+i \sqrt{\det \Sigma}(\frac{\theta_1^+-\theta_1^-}{4})(s-\frac{1}{s})}{\sigma_{22}}.
\end{equation}
Furthermore this parametrization leads to simple expressions for Galois automorphisms $\eta$ and $\zeta$. We derive immediately  that $\theta_1(s)=\theta_1(\frac{1}{s})$ and $\theta_2 (\frac{1}{s})=\Theta_2^-(\theta_1(s))$. Then we have $$\zeta (s)=\frac{1}{s}.$$
Next we search $\eta$ as an automorphism of the form $\eta s =\frac{K}{s}$.
 Since $\theta_2(s)$ is of the form $\theta_2(s)=us+\frac{v}{s}+w$ with  constants 
   $u,v,w$  defined by (\ref{232}),  then  $\theta_2(s)=\theta_2(\frac{K}{s})$  with $K=\frac{u}{v}$. This leads to 
$$
\eta (s) =\frac{K}{s} \text{ with }
K=\frac{-\sigma_{12}-i\sqrt{\det \Sigma}}{-\sigma_{12}+i\sqrt{\det \Sigma}}.
$$
After setting
$$
K=e^{2i\beta} \text{ with }
\beta= \arccos {-\frac{\sigma_{12}}{\sqrt{\sigma_{11}\sigma_{22}}}}
$$
we have
\begin{center}
\begin{tabular}{cc}
${\zeta}(s)=\frac{1}{s}$, & ${\eta}(s)=\frac{e^{2i\beta}}{s}$ \\ 
\end{tabular}
\end{center}
and then
\begin{center}
\begin{tabular}{cc}
$\eta\zeta (s) = e^{2i\beta}s$,
&
$\zeta\eta (s) = e^{-2i\beta}s$.
\\
\end{tabular}
\end{center}
It follows that ${\eta}{\zeta}$ and ${\zeta}{\eta}$ are just rotations for angles $2\beta$ et $-2\beta$ respectively.
By symmetry considerations we can now rewrite $$\theta_2 (s)=\sqrt{uv}(\frac{s}{\sqrt{K}}+\frac{\sqrt{K}}{s})+w=\frac{\theta_1^+-\theta_1^-}{4}\sqrt{\frac{\sigma_{11}}{\sigma_{22}}}(\frac{s}{\sqrt{K}}+\frac{\sqrt{K}}{s})+\frac{-\sigma_{12}(\frac{\theta_1^++\theta_1^-}{2})-\mu_2}{\sigma_{22}}.$$ For $i=1,2$ we have $\theta_i^+-\theta_i^- = 2 \frac{\sqrt{D_i}}{\det \Sigma} $ and $\sigma_{11}D_1=\sigma_{22}D_2$. Then we obtain $\frac{\theta_1^+-\theta_1^-}{4}\sqrt{\frac{\sigma_{11}}{\sigma_{22}}}=\frac{\theta_2^+-\theta_2^-}{4}$. Moreover $\frac{-\sigma_{12}(\frac{\theta_1^++\theta_1^-}{2})-\mu_2}{\sigma_{22}}=\frac{\Theta_2^\pm(\theta_1^+)+\Theta_2^\pm(\theta_1^-)}{2}=\frac{\theta_2^++\theta_2^-}{2}$ (the last equality  follows from elementary geometric properties of an ellipse).
It implies
$$
h_{\theta_2}(s)=\theta_2(s)=
\frac{\theta_2^++\theta_2^-}{2}+
\frac{\theta_2^+-\theta_2^-}{4} (\frac{s}{\sqrt{K}}+\frac{\sqrt{K}}{s})
$$ concluding the proof.
\end{proof}

\begin{figure}[hbtp]
\begin{center}
\includegraphics[scale=1]{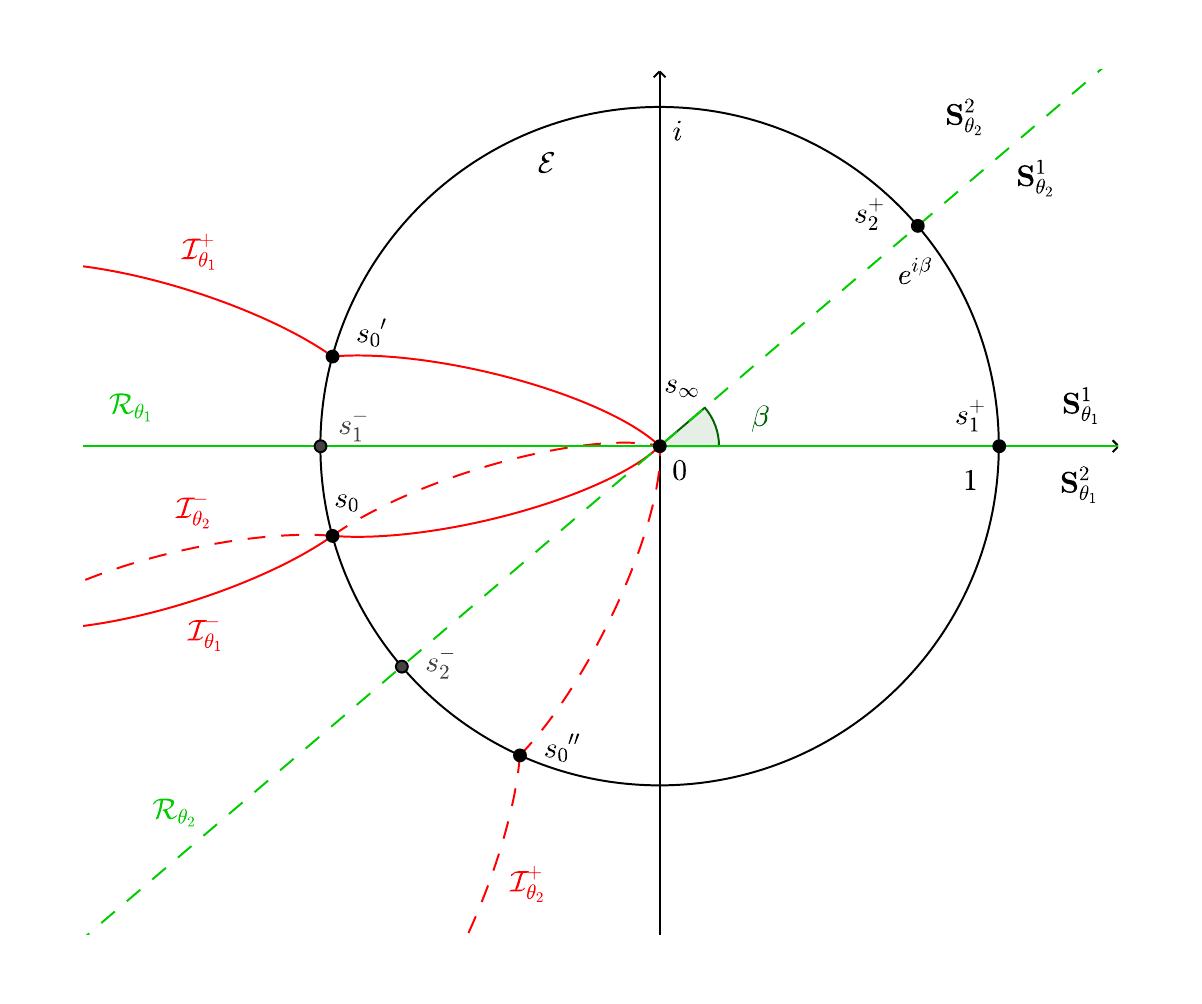} 
\end{center}
\caption{Parametrization of $\s$}
\label{recouvrS}
\end{figure}

Figure \ref{recouvrS} shows  different sets we are interested in according to the  parametrization we have just introduced.
We have $\theta_1(1)=\theta_1^+$, $\theta_1(-1)=\theta_1^-$, $\theta_2 (e^{i\beta})=\theta_2^+$ et $\theta_2 (e^{i(\pi+\beta)})=\theta_2^-$, $\theta_1(0)=\theta_2(0)=\infty$, $\theta_1(\infty)=\theta_2(\infty)=\infty$. Then we write $s_1^+=1$, $s_1^-=-1$, $s_2^+=e^{i\beta}$, $s_2^-=e^{i(\pi+\beta)}$, $s_\infty=0$, $s_{\infty'}=\infty$.
It is easy to see that $$\mathcal{E}=\{s\in \mathbb{C}| \ |s|=1\},$$ and
$$\mathcal{R}_{\theta_1}={\bf R}, \ \mathcal{R}_{\theta_2}=e^{i\beta}{\bf R}.$$
We can determine the equation of the analytic curves of pure imaginary points of $\theta_i$. We have ${\mathcal{I}}_{\theta_1}=\{ s \in \s | \theta_1(s) \in i{\bf R} \}$.
If we write $s=e^{i\omega}$ with $\omega=a+ib\in\mathbb{C}$ we find that $\Re (\theta_1(s))=\frac{\theta_1^++\theta_1^-}{2}+\frac{\theta_1^+-\theta_1^-}{2} \cos (a) \cosh (b)$. It follows that
$${\mathcal{I}}_{\theta_1}=\{s=e^{i\omega}\in\s |\omega=a+ib, 
a\in {\bf R}, b \in{\bf R}, \cos (a) \cosh (b)=\frac{\theta_1^++\theta_1^-}{\theta_1^--\theta_1^+}\}.$$
Similarly we have
$${\mathcal{I}}_{\theta_2}=\{s=e^{i\omega}\in\s |\omega=a+ib, 
a\in {\bf R}, b \in{\bf R}, \cos (a) \cosh (b)=\frac{\theta_2^++\theta_2^-}{\theta_2^--\theta_2^+}\}.$$
We can easily notice that
$$
 {\zeta} {\mathcal{I}}_{\theta_1}^-= {\mathcal{I}}_{\theta_1}^+, \
 {\zeta} {\mathcal{I}}_{\theta_1}^+= {\mathcal{I}}_{\theta_1}^-
\text{ and }
 {\eta} {\mathcal{I}}_{\theta_2}^-= {\mathcal{I}}_{\theta_2}^+, \
 {\eta} {\mathcal{I}}_{\theta_2}^+= {\mathcal{I}}_{\theta_2}^-.
$$

\section{Meromorphic continuation of \texorpdfstring{$\varphi_1$}{phi1} and \texorpdfstring{$\varphi_2$}{phi2} on \texorpdfstring{$\s$}{S}}

\subsection{Lifting of \texorpdfstring{$\phi_1$}{phi1} and \texorpdfstring{$\phi_2$}{phi2} on \texorpdfstring{$\s$}{S} and their meromomorphic continuation}
\label{meromcontinuation}

\paragraph{Lifting of $\varphi_1$ and  $\varphi_2$  on $\s$.}

Since the function $\theta_1 \to \varphi_2(\theta_1)$ is holomorphic on the set $\{\theta_1\in \mathbb{C} : \Re \theta_1< 0\}$ and continuous up to its boundary,  we can lift it to $\bar \Delta_1 = \{s\in\s:\Re \theta_1(s) \leq  0 \}$ as
$$
\varphi_2(s)=\varphi_2(\theta_1(s)), \ \forall s \in \bar \Delta_1.
$$
In the same way we can lift $\varphi_1$ to $\bar\Delta_2$ as
$$
\varphi_1(s)=\varphi_1(\theta_2(s)), \ \forall s \in \bar\Delta_2.
$$
Moreover, by definition of  Galois  automorphisms, functions $\phi_1$ and $\phi_2$ are invariant w.r.t. 
$\eta$ and $\zeta$ respectively: 
\begin{equation}
\label{dhf}
\varphi_2(\zeta s)=\varphi_2(\theta_1(\zeta s))=\varphi_2(\theta_1(s))=\varphi_2(s), \ \ \forall s\in\bar \Delta_1,$$ 
$$ 
\varphi_1(\eta s)=\varphi_1(\theta_2(\eta s))=\varphi_1(\theta_2(s))=\varphi_1(s), \ \ \forall s\in \bar \Delta_2.
\end{equation} 
 Functions $\gamma_1$ and  $\gamma_2$ can be lifted naturally on the whole of ${\bf S}$ 
  as 
$$ \gamma_1(s)=\gamma_1(\theta_1(s), \theta_2(s)), \  \  \   \   \gamma_2(s)=\gamma_2(\theta_1(s), \theta_2(s)) \ \  \ \forall s\in {\bf S}.$$
Since $\gamma(\theta_1(s), \theta_2(s))=0$,  then the right-hand side in the main functional equation (\ref{maineq}) equals zero  
for any  $\theta=(\theta_1(s), \theta_2(s))$ such that  $s \in \bar\Delta_1 \cap \bar\Delta_2$.  Thus we have 
\begin{equation}
\label{eqoo}
\gamma_1(s)\varphi(s)+\gamma_2(s)\varphi_2(s)=0, \  \  \ \forall s\in  \bar \Delta_1 \cap \bar \Delta_2.
\end{equation}

\paragraph{\bf Continuation of  $\varphi_1$ and  $\varphi_2$  on $\Delta$.}

\begin{lem} \label{lempro}
Functions $\varphi_1$ and $\varphi_2$ (defined on $\bar{\Delta}_2$ and $\bar{\Delta}_1$  respectively)
  can be meromorphically continued on $\Delta\cup\{s_0\}$ by setting
$$\varphi_1(s)= 
\begin{tabular}{cc}
$-\frac{\gamma_2(s)}{\gamma_1(s)} \varphi_2(s)$ & if $s\in\Delta_1,$
\\
\end{tabular} 
$$
and
$$\varphi_2(s)= 
\begin{tabular}{cc}
$-\frac{\gamma_1(s)}{\gamma_2(s)} \varphi_1(s)$ & if $s\in\Delta_2.$
\\
\end{tabular}
$$
Furthermore,
\begin{equation}
\label{dfg}
\gamma_1(s)\varphi_1(s)+\gamma_2(s)\varphi_2(s)=0
\ \ \forall s\in\Delta\cup\{s_0\},
\end{equation}
\begin{equation}
\label{invf}
\varphi_1(s)=\varphi(\eta s), \ \   \  \  \varphi_2(s)=\varphi(\zeta s)  \ \ \forall s\in\Delta\cup\{s_0\}.
\end{equation}
\end{lem}

\begin{proof}
The open set $\Delta_1 \cap \Delta_2$ is non-empty and bounded by the curve $\mathcal{I}_{\theta_1}^- \cup \mathcal{I}_{\theta_2}^-$. Functional equation 
(\ref{eqoo})  is valid for $s\in \Delta_1 \cap \Delta_2$. It allows us to continue functions $\varphi_1$ and $\varphi_2$ as meromorphic on $\Delta$ 
  as stated in this lemma. The functional equation (\ref{eqoo})  is then valid on the whole of $\Delta$, as well as the invariance formulas (\ref{invf}). 
\end{proof}

 The function $\phi_1(s)$  is defined in a neighborhood $\mathcal{O}(s_0)$ 
   of $s_0$ as  $\phi_1(\theta_2(s))$ for any $s \in \Delta_2 \cap \mathcal O(s_0)$ and
  $  -\frac{\gamma_2(s)}{\gamma_1(s)} \varphi_2(\theta_1(s))$  for any $s \in \Delta_1 \cap \mathcal O(s_0)$.  
Furthermore, 
$$\lim_{s \to s_0, s\in \Delta_2} \phi_1(s)= \mathbb{E}_{\pi}(\int_0^1 \mathrm{d}L_t^1)$$ by definition of the function $\phi_1.$
   It is easy to see that  function $\frac{\gamma_2(s)}{\gamma_1(s)}$  has a removable singularity at $s_0$ and to compute
  $\lim_{s\to s_0} \frac{\gamma_2(s)}{\gamma_1(s)}=\frac{r_{12}\mu_2-r_{22}\mu_1}{r_{11}\mu_2-r_{21}\mu_1}$.
      Hence 
$$\lim_{s \to s_0, s\in \Delta_1} \phi_1(s)= \lim_{s \to s_0, s\in \Delta_1}
   -\frac{\gamma_2(s)}{\gamma_1(s)} \phi_2(\theta_1(s))= 
   \frac{r_{22}\mu_1 -r_{12}\mu_2  }{r_{11}\mu_2-r_{21}\mu_1 } \mathbb{E}_{\pi}(\int_0^1 \mathrm{d}L_t^2).$$ 
 For any $s \in \Delta_1 \cap \Delta_2 \cap \mathcal{O}(s_0)$, by (\ref{eqoo}) 
   $\phi_1(s)=-\frac{\gamma_2(s)}{\gamma_1(s)}\phi_2(s)$, from where 
    $\lim_{s \to s_0, s\in \Delta_2} \phi_1(s)= \lim_{s \to s_0, s\in \Delta_1} \phi_1(s)$.
       Hence,  function $\phi_1(s)$  has a removable singularity at $s_0$, and so is $\phi_2(s)$ by the same arguments.
       

  Functions $\phi_1$ and $\phi_2$ can be then of course continued to $\bar \Delta$.
   Moreover we have the following lemma.
\begin{lem} \label{lemrev}  The domains
$ \bar \Delta  \cup   \eta \zeta \bar \Delta$ and $\bar \Delta  \cup  \zeta \eta\bar  \Delta$ are simply connected.  
\end{lem}

\begin{proof}
   Since $\eta \zeta$ and $\zeta \eta$ are just rotations for a certain angle $2\beta$ or $-2\beta$, it suffices to check that $\eta \zeta \mathcal{I}_{\theta_1}^+ \subset  \bar \Delta$ 
  and that   $\zeta \eta \mathcal{I}_{\theta_2}^+ \in \bar \Delta$.  In fact, $ \zeta \mathcal{I}_{\theta_1}^+ = \mathcal I_{\theta_1}^- \subset
   \bar \Delta_2.$
   Since $\eta \bar \Delta_2=\bar \Delta_2$, it follows that $ \eta \mathcal I_{\theta_1}^- \subset \bar \Delta_2 \subset \bar \Delta$.
  By the same arguments $\zeta \eta \mathcal{I}_{\theta_2}^+\in \bar \Delta$. One can refer to Figure \ref{etazetadelta}.
  \end{proof}
  \begin{figure}[hbtp]
  \centering
  \includegraphics[scale=0.7]{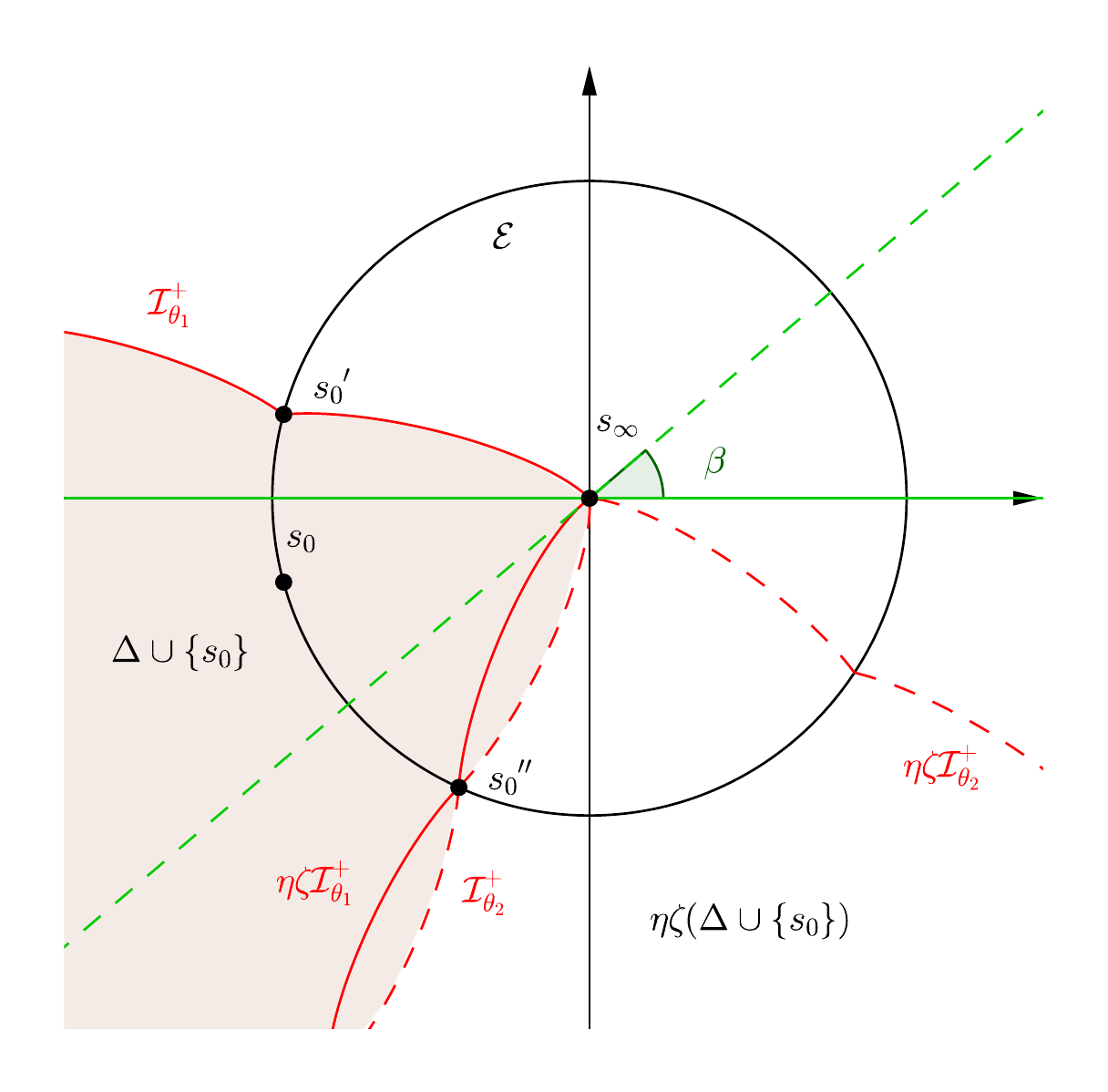}
  \caption{$\Delta$ and $\eta\zeta\Delta$}
  \label{etazetadelta}
  \end{figure}

   Now  we would like  to continue  function $\phi_1$ (resp. $\phi_2$)  on $\eta \zeta \bar \Delta$  (resp. $\zeta \eta \bar \Delta$)  as 
  $\phi_1(s)=G(s) \phi_1(\zeta \eta s)$ for all $s\in  \eta \zeta \bar \Delta$, where $G(s)$ is a known function and $\phi_1 (\zeta \eta s)$ is well defined 
   since $\zeta \eta s \in \bar \Delta$.  We could then continue this procedure for  $(\eta \zeta)^2 \bar \Delta$,   $(\eta \zeta)^3 \bar \Delta$,
   (resp.    $(\zeta \eta)^2 \bar \Delta$,   $(\zeta \eta)^3 \bar \Delta$)  etc
  and hence to define $\phi_1$  (resp. $\phi_2$) 
  on the whole of ${\bf S}$.  Unfortunately, the domain $\bar \Delta$ is closed, from where it will be difficult to establish 
  that the function is meromorphic. From the other hand,  neither $ \Delta  \cup   \eta \zeta \Delta$ nor $  (\Delta \cup s_0) \cup   \eta \zeta  
  (\Delta \cup s_0)$ are simply connected, there is a ``gap''  at $s''_0$. See figure \ref{etazetadelta}.
    To avoid this technical complication, we will first continue $\phi_1$ and $\phi_2$ on a slightly bigger open domain $\Delta^\epsilon$ defined as follows.
Let 
\begin{equation}
\label{de1}
\Delta_1^\epsilon=\{s : {\R}\theta_1(s)<\epsilon\},   \ \ \  \Delta_2^\epsilon=\{s : {\R}\theta_2(s)<\epsilon\}
\end{equation}
  and 
\begin{equation}
\label{de}
\Delta^\epsilon=\Delta_1^\epsilon\cup \Delta_2^\epsilon
\end{equation}

  Let us  fix any $\epsilon>0$ small enough.
   For any $\theta_1 \in {\bf C}$   with ${\R} \theta_1 =\epsilon$,    the function $\Theta_2(\theta_1)$ takes two values    $\Theta_2^{\pm} (\theta_1)$   where    ${\R} (\Theta_2^{-}(\theta_1)) <0$  and  ${\R}(\Theta_2(\theta_1)) >0$. 
      The domain $\Delta_1^\epsilon$ is bounded by the contour $\mathcal{I}_{\theta_1}^{\epsilon}=   \mathcal{I}_{\theta_1}^{\epsilon, -}\cup \mathcal{I}_{\theta_1}^{\epsilon,+}$  where 
     $\mathcal{I}_{\theta_1}^{\epsilon, -}$,   $\mathcal{I}_{\theta_1}^{\epsilon, +}$ both go from $s_\infty$ to $s_{\infty\prime}$,  
 $\mathcal{I}_{\theta_1}^{\epsilon, -} \subset \Delta_2$ and 
    $\mathcal{I}_{\theta_1}^{\epsilon,+ }\cap \Delta =\emptyset$. See Figure~\ref{deltaepsilon}.
    The same is true about the contour   
 $\mathcal{I}_{\theta_2}^{\epsilon}=   \mathcal{I}_{\theta_2}^{\epsilon, -}\cup \mathcal{I}_{\theta_2}^{\epsilon,+}$
    limiting $\Delta_2^\epsilon$, namely  $\mathcal{I}_{\theta_2}^{\epsilon, -} \subset \Delta_1$ and 
    $\mathcal{I}_{\theta_2}^{\epsilon,+ }\cap \Delta =\emptyset$.
  \begin{figure}[hbtp]
  \centering
  \includegraphics[scale=0.7]{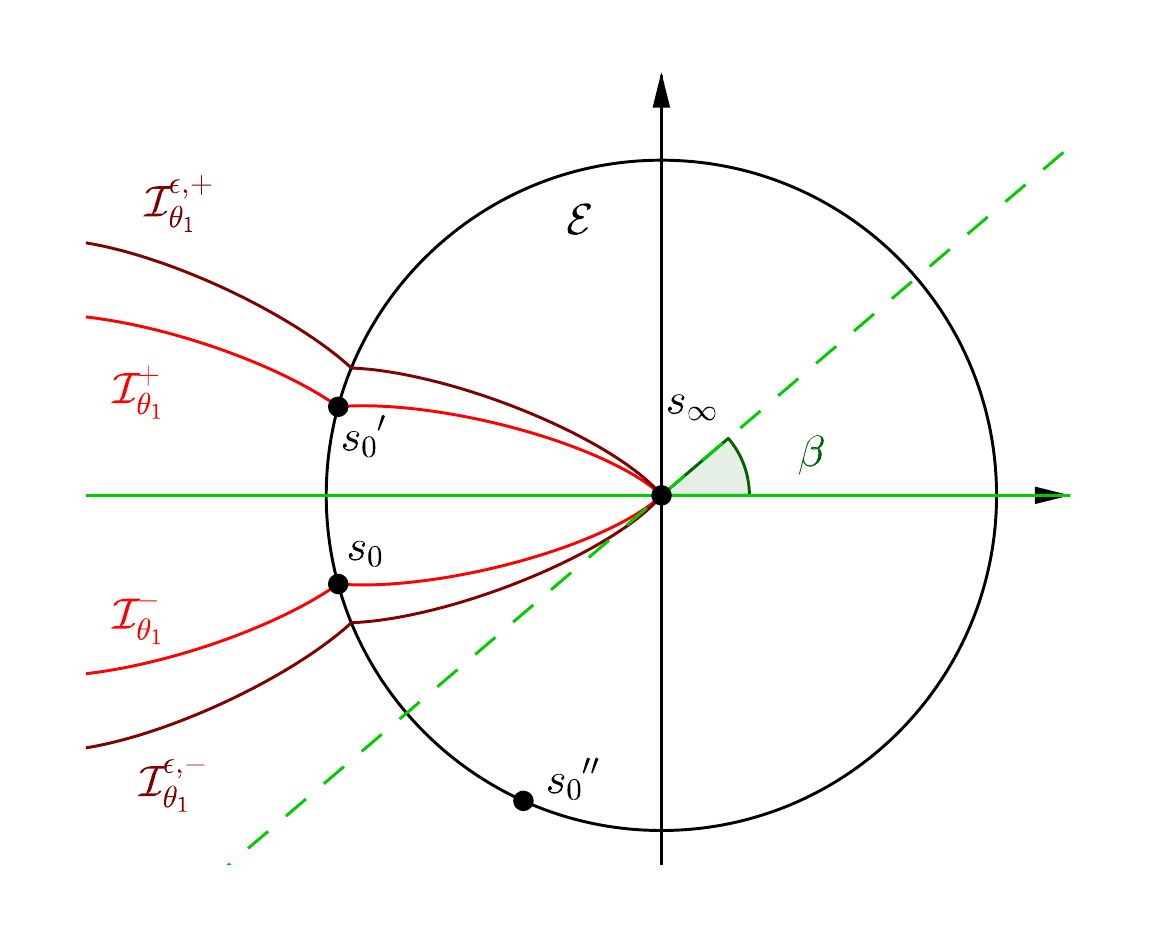}
  \caption{$\mathcal{I}_{\theta_1}^{\epsilon, -}$ and $\mathcal{I}_{\theta_1}^{\epsilon, +}$}
  \label{deltaepsilon}
  \end{figure}

\begin{lem} \label{lemproepsilon}
   Functions $\phi_1(s)$ and $\phi_2(s)$
can be continued as meromorphic functions on $\Delta^\epsilon$. Moreover  
    equation (\ref{dfg}) and the invariance formulas (\ref{invf}) remain valid.
     \end{lem}

\begin{proof}
   For any $s\in \Delta_1^\epsilon\setminus \Delta$,  we have $\zeta s \in \Delta_2 \subset \Delta$, except for $s =s'_0$, for which $\zeta s'_0=s_0$. 
   Anyway, function $\phi_2(s)$ can be continued  as meromorphic function on   $\Delta_1^\epsilon/\Delta$ as :
   $$ \phi_2(s)=\phi_2(\zeta s), \  \  \  \forall s\in \Delta_1^\epsilon \setminus \Delta.$$
  Then $\phi_1(s)$ can be continued on the same domain by (\ref{dfg}):
$$ \phi_1(s)=-\frac{\gamma_2(s)}{\gamma_1(s)} \varphi_2(s)   \  \   \  \forall s \in \Delta_1^\epsilon\setminus\Delta.$$
   Similarly, the formulas 
   $$ \phi_1(s)=\phi_1(\eta s), \  \     \phi_2(s)=-\frac{\gamma_1(s)}{\gamma_2(s)} \varphi_1(s)   \  \   \   \forall  s \in \Delta_2^\epsilon\setminus\Delta$$
  determine the meromorphic continuation of $\phi_1(s)$  and $\phi_2(s)$ on  $\Delta_2^\epsilon\setminus\Delta$.
\end{proof}

\begin{lem}
\label{lemconteps}
     The domains
   $\Delta^\epsilon \cap \eta\zeta \Delta^\epsilon$ and $\Delta^\epsilon \cap  \zeta\eta \Delta^\epsilon$
 are open simply connected 
  domains.  Function $\phi_1(s)$ can be continued as meromorphic   on $ \Delta^\epsilon \cup \eta\zeta  \Delta^\epsilon$ by the formula :
 \begin{equation}
 \varphi_1(s)= 
\frac{\gamma_1( \zeta \eta s)\gamma_2( \eta s)}{\gamma_2( \zeta \eta s)\gamma_1( \eta s)} \varphi_1( \zeta \eta s), \ \ \forall s \in  \eta\zeta  \Delta^\epsilon\ \hbox{ continuation by rotation of } 2\beta. \label{pro+trans}
\end{equation}
 Function $\phi_2(s)$ can be continued as meromorphic on $ \Delta^\epsilon \cup \zeta\eta  \Delta^\epsilon$ by the formula :
 \begin{equation}
\varphi_2(s)= 
\frac{\gamma_2( \eta \zeta s)\gamma_1( \zeta s)}{\gamma_1( \eta \zeta s)\gamma_2( \zeta s)} \varphi_2( \eta \zeta s), \ \ \forall s\in 
\zeta\eta  \Delta^\epsilon  \hbox{ continuation by rotation of} -2\beta. \label{pro-trans1}
\end{equation}
\end{lem}

\begin{proof} We have shown in the proof of Lemma \ref{lemrev}
 that $\eta\zeta \mathcal I_{\theta_1}^{+} \subset \bar \Delta \subset \Delta^\epsilon$, 
  and  that   $\zeta\eta \mathcal I_{\theta_2}^{+} \subset \bar \Delta \subset \Delta^\epsilon$. Since $\zeta \eta$ and $\eta \zeta$ are just rotations, this implies that  $\Delta^\epsilon \cap \eta\zeta \Delta^\epsilon$ and $\Delta^\epsilon \cap  \zeta\eta \Delta^\epsilon$  are non-empty open simply connected domains, and that $\Delta^\epsilon \cup \eta\zeta \Delta^\epsilon$ and $\Delta^\epsilon \cup  \zeta\eta \Delta^\epsilon$ are simply connected.

 Let us take $s \in \Delta^\epsilon \cap  \eta\zeta \Delta^\epsilon$. Then $\zeta \eta s \in \Delta^\epsilon \cap  \zeta\eta \Delta^\epsilon $  
and we can write by (\ref{dfg})
\begin{equation}
\label{efty}
 \gamma_1(\zeta \eta s   ) \phi_1(\zeta \eta s   )+ \gamma_2( \zeta \eta s  ) \phi_2( \zeta \eta s  )=0.
\end{equation}  
   Furthermore, we have shown in the proof of Lemma \ref{lemrev} that $\zeta \eta \mathcal{I}_{\theta_2}^+  \in \bar \Delta_1$.  It follows that for all $\epsilon$ small enough  $\zeta \eta \mathcal{I}_{\theta_2}^{ \epsilon, +} \in  \Delta_1$, and hence    $ \Delta^\epsilon \cap  \zeta\eta \Delta^\epsilon \subset \Delta_1^\epsilon $.  Since $\zeta \Delta_1^\epsilon=\Delta_1^\epsilon$, then 
  $ \zeta ( \Delta^\epsilon \cap  \zeta\eta \Delta^\epsilon) \subset \Delta_1^\epsilon \subset \Delta_\epsilon$. Then $\zeta(\zeta \eta s)=\eta s \in \Delta^\epsilon$ and we can write (\ref{dfg})  and (\ref{invf}) at this point as well:
\begin{equation}
\label{smh}
  \gamma_1( \eta s   ) \phi_1 (\eta s   )+ \gamma_2(  \eta s  ) \phi_2(  \eta s  )=0.
\end{equation}
\begin{equation}
\label{smhh}
  \phi_1 (\eta s   )=\phi_1(s)
\end{equation}
\begin{equation}
\label{dooo} 
\phi_2 (\eta s)=\phi_2(\zeta \eta s).
\end{equation}
     Combining  (\ref{smhh}) and  (\ref{smh}) we get $\phi_1(s)=-\gamma_2(\eta s) \phi_2(\eta s)/\gamma_1(\eta s)$
  from where  by  (\ref{dooo}) 
\begin{equation}
\label{sdnew}
 \phi_1(s)=  -\frac{\gamma_2(\eta s) }{\gamma_1(\eta s)} \phi_2(\zeta \eta s).
\end{equation} 
 Due to  (\ref{efty})  
\begin{equation}
\label{sdnew1}
\phi_2(\zeta \eta s)=-\frac{\gamma_1(\zeta\eta s)}{\gamma_2(\zeta\eta s)} \phi_1(\zeta \eta s).
\end{equation}     
   Substituting  (\ref{sdnew1}) into (\ref{sdnew}), 
 we obtain the formula (\ref{pro+trans})  
valid for any $s \in \Delta^\epsilon \cap  \eta\zeta \Delta^\epsilon$. By principle of analytic continuation 
this allows to  continue  $\phi_1$ on $\eta\zeta \Delta^\epsilon$ as meromorphic function.
   The proof is completely analogous for $\varphi_2$.  
\end{proof}
   
   We may now in the same way, using  formulas (\ref{pro+trans}) and  (\ref{pro-trans1}),  continue function $\phi_1(s)$  (resp. $\phi_2(s)$) as meromorphic  on  
$(\eta\zeta)^2  \Delta^\epsilon$, $(\eta\zeta)^3  \Delta^\epsilon$ (resp.  $(\zeta\eta)^2  \Delta^\epsilon$, $(\zeta\eta)^3  \Delta^\epsilon$)  etc proceeding each  time by rotation for the angle $2\beta$ 
[resp. $-2\beta$].  Namely we have the following lemma.

 \begin{lem}
\label{lemcontepse}
   For any $n \geq 1$
     the domains $\Delta^\epsilon \cup \eta\zeta \Delta^\epsilon \cup \cdots \cup  (\eta \zeta)^{n}\Delta^\epsilon$ and 
    $\Delta^\epsilon \cup  \zeta\eta \Delta^\epsilon \cup \cdots (\zeta \eta)^n \Delta^\epsilon $ 
 are open simply connected 
  domains.  Function $\phi_1(s)$ can be continued as meromorphic subsequently  on $\eta \zeta \Delta^\epsilon, (\eta \zeta)^2 \Delta^\epsilon, \cdots 
   (\eta \zeta) ^n \Delta^\epsilon $ by the formulas :
 \begin{equation}
 \varphi_1(s)= 
\frac{\gamma_1( \zeta \eta s)\gamma_2( \eta s)}{\gamma_2( \zeta \eta s)\gamma_1( \eta s)} \varphi_1( \zeta \eta s), \ \ \forall s \in  (\eta\zeta)^k  \Delta^\epsilon, k =1,2\ldots, n,\ \hbox{ continuation by rotation of } 2\beta. \label{pro+transe}
\end{equation}
 Function $\phi_2(s)$ can be continued as meromorphic on $\zeta \eta \Delta^\epsilon, (\zeta \eta)^2 \Delta^\epsilon, \cdots 
   (\zeta \eta) ^n \Delta^\epsilon  $ by the formulas :
 \begin{equation}
\varphi_2(s)= 
\frac{\gamma_2( \eta \zeta s)\gamma_1( \zeta s)}{\gamma_1( \eta \zeta s)\gamma_2( \zeta s)} \varphi_2( \eta \zeta s), \ \ \forall s\in 
(\zeta\eta)^k  \Delta^\epsilon, k=1,2,\ldots n,  \hbox{ continuation by rotation of} -2\beta. \label{pro-trans1e}
\end{equation}
\end{lem}

\begin{proof}
   We proceed by induction on $k=1,2,\ldots n$. For $k=1$, this is the subject of the previous  lemma. 
   For any $k=2,\ldots,n$, assume the formula (\ref{pro+transe}) for any $s\in (\eta \zeta)^{k-1} \Delta$. 
 The domain  $(\eta \zeta)^{k-1}  \Delta^\epsilon \cap (\eta \zeta)^k \Delta^\epsilon=(\eta \zeta)^{k-1}(\Delta^\epsilon \cap \eta \zeta \Delta^\epsilon)$ is a non empty open domain by Lemma~\ref{lemconteps}, $(\eta \zeta)^{k-1}$ being just the rotation for the angle $2(k-1)\beta$.  
    The formula (\ref{pro+transe}) is valid for any $s \in  (\eta \zeta)^{k-1}  \Delta^\epsilon \cap (\eta \zeta)^k \Delta^\epsilon$ by induction assumption.
   Hence, by the principle of meromorphic continuation it is valid for any $s\in   (\eta \zeta)^k \Delta^\epsilon$.
   The same is true for the formula (\ref{pro-trans1e}). 
\end{proof}

 Proceeding as in Lemma \ref{lemcontepse} by rotations,
 we will continue $\phi_1$  soon on the  first half of ${\bf S}$, that is ${\bf S}_{\theta_2}^1$, then the whole of ${\bf S}$ and go further, turning around ${\bf S}$ for the second time, for the third, etc up to infinity. 
  In fact, each  time we complete this procedure on one of two halves of ${\bf S}$, we recover a new branch of the function $\phi_1$ as function of $\theta_2 \in {\bf C}$.  So, going back to
  the complex plane, we continue this function as multivalued and determine all its branches.
   The same is true for $\phi_2$ if we proceed by rotations in  the opposite direction.
 This procedure could  be presented better on the universal 
      covering of ${\bf S}$,  but  for the purpose of the present paper  it is enough to complete 
  it only on one-half of ${\bf S}$, that is to study just the first (main) branch of $\phi_1$ and $\phi_2$. 
    We summarize this result in the following theorem.
   We recall that ${\bf S} ={\bf  S}_{\theta_1}^1 \cup  {\bf  S}_{\theta_1}^2$ and we denote by   ${\bf  S}_{\theta_1}^1$  the half that contains 
   $s'_0$ (and not $s_0$, as $\zeta s_0=s'_0$). In the same way   
   ${\bf S} ={\bf  S}_{\theta_2}^1 \cup  {\bf  S}_{\theta_2}^2$ and we denote by   ${\bf  S}_{\theta_2}^1$  the half that contains 
   $s''_0$ (and not $s_0$, as $\eta s_0=s''_0$), see Figure \ref{recouvrS}.

\begin{thm} \label{thmpro} 
  For any $s\in {\bf S}_{\theta_2}^1$  there exists $n\geq 0$ such that $(\zeta \eta)^n s \in \bar \Delta$.
Let us define 
\begin{equation}
\label{qs}
\phi_1(s)= \frac{\gamma_1( (\zeta \eta)^n s)\dots \gamma_1( \zeta \eta s)}{\gamma_2( (\zeta \eta)^n s)\dots \gamma_2( \zeta \eta s)} 
\frac{\gamma_2( \eta (\zeta \eta)^{n-1} s)\dots\gamma_2( \eta s)}{\gamma_1( \eta \zeta \eta)^{n-1} s)\dots \gamma_1( \eta s)}
\phi_1((\zeta\eta)^n s)
\end{equation}
 Then the function $\phi_1(s)$ is meromorphic on  ${\bf S}_{\theta_2}^1$. 
   For any $s\in {\bf S}_{\theta_1}^1$, there exists $n\geq 0$ such that $(\eta \zeta)^n s \in \bar \Delta$.
  Let us define
\begin{equation}
\label{qse}
\phi_2(s)= 
\frac{\gamma_2( (\eta \zeta)^n s)\dots \gamma_2( \eta \zeta s)}{\gamma_1( (\eta \zeta)^n s)\dots \gamma_1( \eta \zeta s)} 
\frac{\gamma_1( \zeta (\eta \zeta)^{n-1} s)\dots\gamma_1( \eta s)}{\gamma_2( \zeta \eta \zeta)^{n-1} s)\dots \gamma_2( \zeta s)} 
\phi_1((\eta\zeta)^n s)
\end{equation}
 Then the function $\phi_2(s)$ is meromorphic on  ${\bf S}_{\theta_1}^1$. 
  \end{thm} 

\begin{proof} It is a direct corollary of  Lemma \ref{lemrev} and Lemma \ref{lemcontepse}.
\end{proof}  

\subsection{Poles of functions \texorpdfstring{$\phi_1$}{phi1} and  \texorpdfstring{$\phi_2$}{phi2} on \texorpdfstring{${\bf S}$}{S}}
\label{mainpole}

     It follows from meromorphic continuation procedure that all poles of $\phi_1(s)$ and $\phi_2(s)$ on
   ${\bf S}$ are located on the ellipse ${\cal E}$,  they are  images of zeros of $\gamma_1$ and 
$\gamma_2$ by automorphisms $\eta$ and $\zeta$ applied several times. 
   Then all poles of all branches of $\phi_2(s)$  (resp. $\phi_2(s)$) on ${\bf C}_{\theta_1}$ 
 (resp. ${\bf C}_{\theta_2}$)  are on the real segment $[\theta _1^-, \theta_1^+]$  
   (resp.  $ [\theta _2^-, \theta_2^+]$).  

\smallskip 

\noindent {\bf Notations of arcs on $\mathcal{E}$.} 
         Let us remind that  we denote  by $\{s_1, s_2\}$ an arc of the ellipse ${\cal E}$ with ends 
  at $s_1$ and $s_2$ \underline{not} passing through the origin, see Theorem \ref{thmmain}.
     From now on, we will  denote in square brackets $]s_1, s_2[$  or $[s_1,s_2]$  an arc  of  ${\cal E}$ 
   going in the anticlockwise direction  from $s_1$ to $s_2$.
 
 In order to compute the asymptotic expansion of stationary distribution density, we are interested in  poles  of
  $\phi_1$  on the arc $]s''_0, s_2^+[$   and  in those of  $\phi_2$ on the arc $]s_1^+, s'_0[$.  
   To determine the main asymptotic term, we are particularly interested in the pole of
 $\phi_1(\theta_2(s))$ on $]s''_0, s_2^+[$
  closest to $s''_0$ and in the one of $\phi_2(\theta_1(s))$ on $]s_1^+, s'_0[$ closest to $s'_0$.
      We identify them in this section.

   We remind that $\theta^{*}$  is a zero of $\gamma_1(s)$ on $\mathcal{E}$ different from $s_0$ and 
that  $\theta^{**}$  is a zero of $\gamma_2(s)$ on $\mathcal{E}$ different from $s_0$. Their coordinates are
\begin{eqnarray} 
\theta^*&=&
2 \frac{r_{21}\mu_1-r_{11}\mu_2}{r_{21}^2\sigma_{11}-2r_{11}r_{21}\sigma_{12}+r_{11}^2\sigma_{22}}
\Big(-r_{21} , r_{11} \Big),\nonumber  \\ 
\theta^{**}&=& 
2 \frac{r_{12}\mu_2-r_{22}\mu_1}{r_{22}^2\sigma_{11}-2r_{22}r_{12}\sigma_{12}+r_{12}^2\sigma_{22}}
\Big(r_{22} , -r_{12} \Big)
\label{zerostheta}
\end{eqnarray}
   Their images by automorphisms $\eta$ and $\zeta$ have the following coordinates:
\begin{eqnarray}
\eta \theta^*&=&\Big( -\frac{r_{11}}{\sigma_{11}r_{21}}(\sigma_{22}\theta_2^{*}+2\mu_2) , \theta_2^* \Big),\nonumber\\   
\zeta \theta^{**}&=& \Big(\theta_1^{**} ,-\frac{r_{22}}{\sigma_{22}r_{12}}(\sigma_{11}\theta_1^{**}+2\mu_1) \Big).\label{zerosthetaauto}
\end{eqnarray}

\begin{lem}
\label{lem3}

\begin{itemize}

\item[(1)]  If   $\theta^{**}\in ]s_0,s_1^+[$,
   then $\zeta \theta^{**}$ is a pole of $\phi_2(\theta_1(s))$ on $]s_1^+, s'_0[$.

\item[(2)] 
If   $\theta^{*}\in ]s_2^+, s_0[$,
 then $\eta \theta^{*}$ is a pole of $\phi_1(\theta_2(s))$ on $]s_0'',s_2^+[$.
\end{itemize}
\end{lem}

\begin{proof}

     By meromorphic continuation procedure
\begin{equation}
\label{mpo}  
 \phi_2(\zeta\theta^{**})=\frac{\gamma_2(\eta \theta^{**})\gamma_1(\theta^{**})\phi_2(\eta\theta^{**})}{\gamma_2(\theta^{**})
  \gamma_1(\eta\theta^{**})}.
\end{equation}

  Let us check that the numerator in (\ref{mpo}) is non zero, this will prove the statement (1)  the lemma. 

It is clear that $\gamma_1(\theta^{**})\ne 0$ due to stability conditions (\ref{u}) and (\ref{v}).

  Suppose that $\gamma_2(\eta\theta^{**})=0$. This could be only if $\eta \theta^{**}=\theta^{**} \in  ]s_0, s_1^+[$, thus
      $\theta^{**}=s_2^-$ where $\theta_2(s_2^-) <0$ and consequently  $\phi_1( \eta\theta^{**})<\infty$.
   But by meromorphic continuation of $\phi_2(\theta_1(s))$
  to the arc  $\{s \in {\cal E} : \theta_2(s)<0\}$ we have:
 $\phi_2(\eta \theta^{**}_1)=-\frac{\gamma_1(\eta \theta^{**})\phi_1(\eta \theta^{**}_2)}{\gamma_2( \eta\theta^{**})}$, from where
by (\ref{mpo})
    $$\phi_2(\zeta \theta^{**})=-\frac{\gamma_1(\theta^{**})} { \gamma_2(\theta^{**})}\phi_1(\eta \theta_2^{**}).$$
  Then $\zeta \theta^{**}$  is clearly a pole of $\phi_2$, this finishes the proof of statement (1) of the lemma in this particular case.
   Otherwise $\gamma_2(\eta \theta^{**}) \ne 0$.

 Let us finally check that $\phi_2(\eta \theta^{**})\ne 0$.
 Let us first observe  that $\phi_2(\theta_1(s))\ne 0$ for any $s \in {\cal E}$ with one of two coordinates non-positive.
     In fact,  if the first coordinate $\theta_1(s)$  of $s$ is non-positive, then $\phi_2(\theta_1(s))\ne 0$ by its definition.
   If $s$ has the second coordinate $\theta_2(s)$ non-positive, then
  $\phi_2(\theta_1(s))=-\frac{\gamma_1(s)}{\gamma_2( s)}\phi_1(\theta_2(s))$  where
  $\gamma_1(s)$ can not have zeros with the second coordinate non-positive by stability conditions and
  neither $\phi_1(\theta_2(s))$ by its definition.
 Hence, $\phi_2(\theta_1(s))\ne 0$ on the  arc  $\{s \in {\cal E} : \theta_1(s)\leq 0 \hbox{ or } \theta_2(s) \leq 0\}$.

     It remains to consider the case
  where both coordinates of $\eta \theta^{**}$ are positive, i.e.
 $\theta^{**} \in ]\eta s'_0, s_1^+[$ where the parameters are such that  $s^{1,+}_2>\theta_2(\eta s'_0)>0$
  and to show that  $\phi_2(\eta \theta^{**}_1)\ne 0$.
     Suppose the contrary, that $\phi_2(\eta \theta^{**})=0$. Then there are zeros of $\phi_2$ on $] \eta s_1^+, s'_0 [$
  and among these zeros there exists $\theta^0$  the closest one to $s'_0$.
  By meromorphic continuation
\begin{equation}
\label{zez}   
 \phi_2(\theta^{0}_1)=\frac{\gamma_2( \eta \zeta \theta^{0})\gamma_1(\zeta \theta^{0})\phi_2(\eta\zeta \theta^{0})}{\gamma_2(\zeta \theta^{0}_1)
  \gamma_1(\eta\zeta \theta^{0})},
\end{equation}
  where $\eta\zeta \theta^0 \in ]\theta_0, s''_0]$.
      First of all, we note that $\phi_2(\eta\zeta \theta^{0})\ne 0$ if $\eta \zeta \theta^0 \in [ \theta_0, s'_0[$, since
$\theta^0$ is the closest zero to $s'_0$, and   $\phi_2(\eta\zeta \theta^{0})\ne 0$  if $\eta \zeta \theta^0 \in [s'_0, s''_0]$, because
  one of coordinates of $\eta \zeta \theta^0$ is non-positive within this segment.
   Hence, $\phi_2(\eta\zeta \theta^{0})\ne 0$  for any point  $\eta\zeta \theta^0 \in ]\theta_0, s''_0[$.

   Furthermore, since $\eta\zeta \theta^0 \in ]\theta_0, s''_0[$, then  $\eta \zeta \theta^0 \ne \theta^{**}$ and thus
     $\gamma_2(\eta \zeta \theta^0)\ne 0$ except for $\eta \zeta \theta^0= s_0$.
  As for this particular case $\eta \zeta \theta^0= s_0$, we would have
 $\phi_2(\theta^0)=-\gamma_1(s''_0)\phi_1(s_0)\gamma_2^{-1}(s''_0)\ne 0$, so that
   $\theta^0=\zeta \eta s_0$ can not be a zero of $\phi_2$.

     The point $\zeta \theta^0 \in \zeta [ \eta \theta^{**}, s'_0]=\zeta \eta [\eta s'_0,   \theta^{**}]$
   that is the segment $[\eta s'_0, \theta^{**}]$ rotated for the angle $-2\beta$.  Hence 
  $\zeta \theta^0$  is located on ${\cal E}$ below $\theta^{**}$.
    Then  $\gamma_1(\zeta \theta^0)=0$  combined with $\gamma_2(\theta^{**})=0$ is impossible by stability conditions (\ref{u}) and (\ref{v}).
   Thus $\gamma_1(\zeta \theta^0)\ne 0$.
It follows  from (\ref{zez}) that $\phi_2(\theta^0_1)\ne 0$.  Thus there exist no zeros of $\phi_2 $ on  
  $] \eta s^{1,+}, s'_0 [$
      and finally  $\phi_2(\eta \theta^{**}) \ne 0$.
  Therefore the numerator in (\ref{mpo}) is non zero,  hence  $\zeta \theta^{**}$ is a pole of $\phi_2$.

   The reasoning for $\theta^*$ is the same.

\end{proof}

\begin{lem}
\label{lem1}
\begin{itemize}

\item[(i)]  Assume that $\theta^p \in ]s_1^+, s'_0[$ is a pole of $\phi_2(\theta_1)$ and it is the closest pole to $s'_0$ . 

   If the parameters $(\Sigma, \mu)$ are such that $\theta_2(s_1^+)\leq 0$, 
  or the parameters  $(\Sigma, \mu, R)$ are such that $\theta_2(s_1^+) >0$  but $\eta \zeta \theta^p \not\in ]\eta s_1^+, s_0[$, then
      $\gamma_2( \zeta \theta^p)=0$ where $\zeta \theta^p \in ]s_1^+, s_0[$  and $\theta^p$ is a pole of the first oder.

   If the parameters  $(\Sigma, \mu, R)$ are such that $\theta_2(s_1^+)>0$  and $\eta \zeta \theta^p \in ]\eta s_1^+, s_0[$, then
      either $\gamma_2( \zeta \theta^p)=0$  where $\zeta \theta^p \in ]s_0, s_1^+[$  or  $\gamma_1 (\eta\zeta \theta^p)=0$.
       Furthermore,  in this case,    if  $\gamma_2( \zeta \theta^p)$   and   $\gamma_1 (\eta\zeta \theta^p)$ do not equal zero simultaneously, 
      then  $\theta^p$ is a pole of the first order.

\item[(ii)] 
Assume that $\theta^p \in ]s_0'',s_2^+[$ is a pole of $\phi_1(\theta_2)$ and it is the closest pole to $s_0''$.

    If the parameters $(\Sigma, \mu)$ are such that $\theta_1(s_2^+)\leq 0$, 
  or the parameters  $(\Sigma, \mu, R)$ are such that $\theta_1(s_2^+) >0$  but $\zeta \eta \theta^p \not\in ]s_0,\zeta s_2^+[$, then
      $\gamma_1( \eta \theta^p)=0$ where $\eta \theta^p \in ]s_0, s_2^+[$  and $\theta^p$ is a pole of the first oder.

   If the parameters  $(\Sigma, \mu, R)$ are such that $\theta_1(s_2^+)>0$  and $\zeta \eta \theta^p \in ]s_0, \eta s_1^+, [$, then
      either $\gamma_1( \eta \theta^p)=0$  where $\eta \theta^p \in ] s_2^+,s_0[$  or  $\gamma_2 (\zeta\eta \theta^p)=0$.
       Furthermore,  in this case,    if  $\gamma_1( \eta \theta^p)$   and   $\gamma_2 (\zeta\eta \theta^p)$ do not equal zero simultaneously, 
      then  $\theta^p$ is a pole of the first order. 

\end{itemize}
\end{lem}

\begin{proof}
 Due to meromorphic continuation procedure we have
\begin{equation}
\label{pri}
\phi_2(\theta^p_1)=\frac{\gamma_2(\eta\zeta \theta^p) \gamma_1(\zeta\theta^p) \phi_2(\eta\zeta\theta^p_1)}{\gamma_2(\zeta\theta^p)\gamma_1(\eta\zeta\theta^p)}
\end{equation}
   where $\eta\zeta \theta^p \in ]\theta^p, s''_0]$.

 Assume that $\eta\zeta\theta^p \in ]\theta^p, s_0[$.
   In this case point  $\eta\zeta \theta^p$ has the second coordinate positive and so does $\zeta\theta^p \in [s_0, s_1^+[$.
 It follows that $\theta_2(s_1^+)>0$ and $\eta\zeta \theta^p \in \eta ]s_0, s_1^+[ \cap \{\theta : \theta_2>0\}=]\eta s_1^+, s_0[$.

  Thus, if the parameters $(\Sigma, \mu)$ are such that $\theta_2(s_1^+)\leq 0$, 
  or the parameters  $(\Sigma, \mu, R)$ are such that $\theta_2(s_1^+)>0$  but $\eta \zeta \theta^p \not\in ]\eta s_1^+, s_0[$, then
    the second coordinate of $\eta\zeta \theta^p$ is non-positive, i.e. $\eta\zeta\theta^p \in [s_0, s''_0]$. In this case
\begin{equation}
 \phi_2(\eta\zeta\theta^p_1)=-\frac{ \gamma_1(\eta\zeta\theta^p)  }{ \gamma_2(\eta\zeta\theta^p)  }\phi_1(\eta\zeta\theta^p)
\end{equation}
    from where by (\ref{pri}) 
\begin{equation}
\label{dp}
\phi_2(\theta^p_1)=-\frac{ \gamma_1(\zeta\theta^p) \phi_1(\eta\zeta\theta^p_2)}{\gamma_2(\zeta\theta^p)}.
\end{equation}
        Since $\phi_1(\eta\zeta\theta^p_2)$ is finite for  
    for any  $\eta\zeta\theta^p \in [s_0, s''_0]$  by its initial definition, the formula (\ref{dp}) implies
that $\gamma_2(\zeta\theta^p)=0$ and the pole $\theta^p$  is of the first order.

   If   parameters  $(\Sigma, \mu, R)$ are such that $\theta_2(s_1^+)>0$  and  $\eta \zeta \theta^p \in ]\eta s_1^+, s_0[$, then 
either $\eta\zeta\theta^p \in [s''_0, s_0]$ or $\eta\zeta\theta^p \in ]s_0, \theta^p[$. In the first case we have 
 (\ref{dp}) as previously from where  $\gamma_2(\zeta\theta^p)=0$ and the pole $\zeta \theta^p$ is of the first order.
   Let us turn to the second case $\eta\zeta\theta^p \in ]\theta^p, s_0[$ for which we will use the formula (\ref{pri}).  
 The pole $\theta^p$ being  the closest to $s'_0$, then 
$\eta\zeta\theta^p$  can not be a pole of $\phi_2$ on $]\theta^p, s'_0[$.
   It can neither be a pole of $\phi_2$ on $[s'_0, s_0]$, since this function is initially well defined on this segment.
    Hence in  formula (\ref{pri}) 
  $\phi_2(\eta\zeta\theta^p_1)\ne \infty$ for $\eta\zeta \theta^p \in ]\theta^p, s'_0[$.
     It follows from (\ref{pri}) that either  
 $\gamma_2( \zeta \theta^p)=0$  or  $\gamma_1 (\eta\zeta \theta^p)=0$
   and  if these two equalities do not hold simultaneously, then pole  $\theta^p$  must be of the first order.

   The proof in the case (ii) is symmetric.

\end{proof}

           Figure~\ref{fi12} gives two illustrations of Lemmas \ref{lem3} and \ref{lem1}.
  
      On the left figure the parameters are such that $\theta_1(s_2^+)>0$ and $\theta_2(s_1^+)>0$.
  Let us look at  zeros $\theta^*$ of $\gamma_1$ and $\theta^ {**}$ of $\gamma_2$ different from $s_0$.
   We see $\theta^{*} \in ]s_2^+, s_0[$, then $\eta \theta^*$ is the first candidate 
  for the closest pole of $\phi_1$ to $s''_0$  on $]s''_0 ,s_2^+[$.
   We also see $\theta^{**} \not\in [s_0, \zeta s_2^+]$, then there are no other candidates.
  Hence the closest  pole  of $\phi_1$ to $s''_0$  on $]s''_0 ,s_2^+[$  is $\eta \theta^{*}$.
   Since $\theta^{**} \in ]s_0, s_1^+[$, then $\zeta \theta^{**}$ is the first candidate 
   for the closest pole of $\phi_2$ to $s'_0$  on $]s_1^+, s'_0[$.
   Furthermore,  $\theta^*  \in ] \eta s_1^+, s_0[$, so that $\zeta \eta \theta^*$
is  the second candidate  to  be the closest pole of $\phi_2$ to $s'_0$  on $]s_1^+, s_0[$.
  We see at  the picture that $\zeta \eta \theta^*$ is closer to $s'_0$ than $\zeta \eta^{**}$.

    On the right figure  the parameters are such that $\theta_1(s_2^+)<0$ and $\theta_2(s_1^+)<0$.
 We see $\theta^{*} \in ]s_2^+, s_0[$, then $\eta \theta^*$ is immediately 
   the closest pole of $\phi_1$ to $s''_0$  on $]s''_0 ,s_2^+[$.
 Since $\theta^{**} \not \in ]s_0, s_1^+[$, then there are 
  no poles  of $\phi_2$  on $]s_1^+, s'_0[$.

\begin{figure}[hbtp]
  \centering
  \includegraphics[scale=0.9]{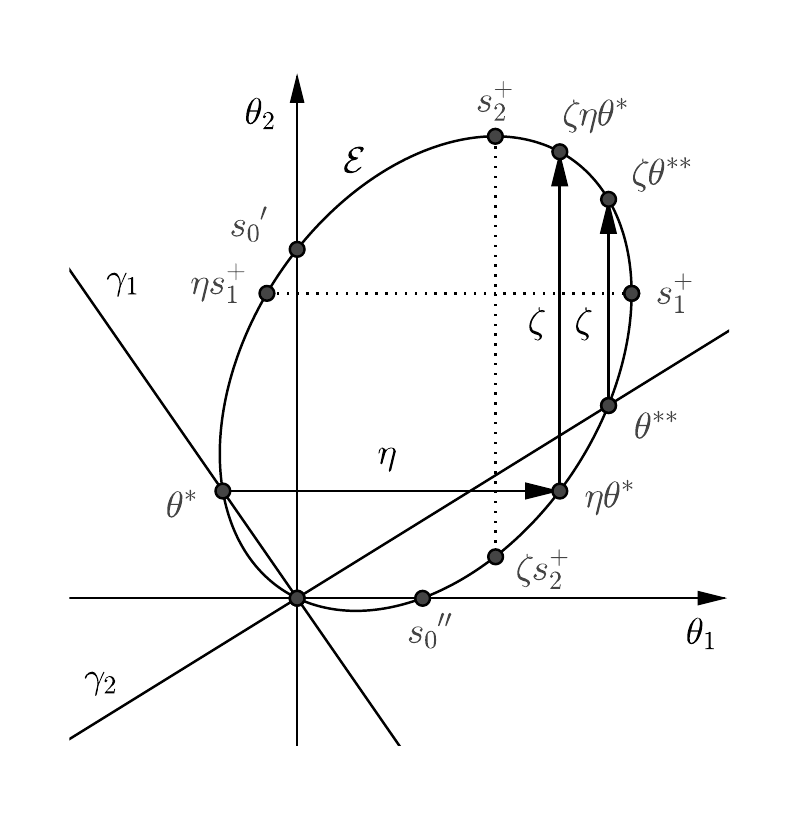}
   \includegraphics[scale=0.9]{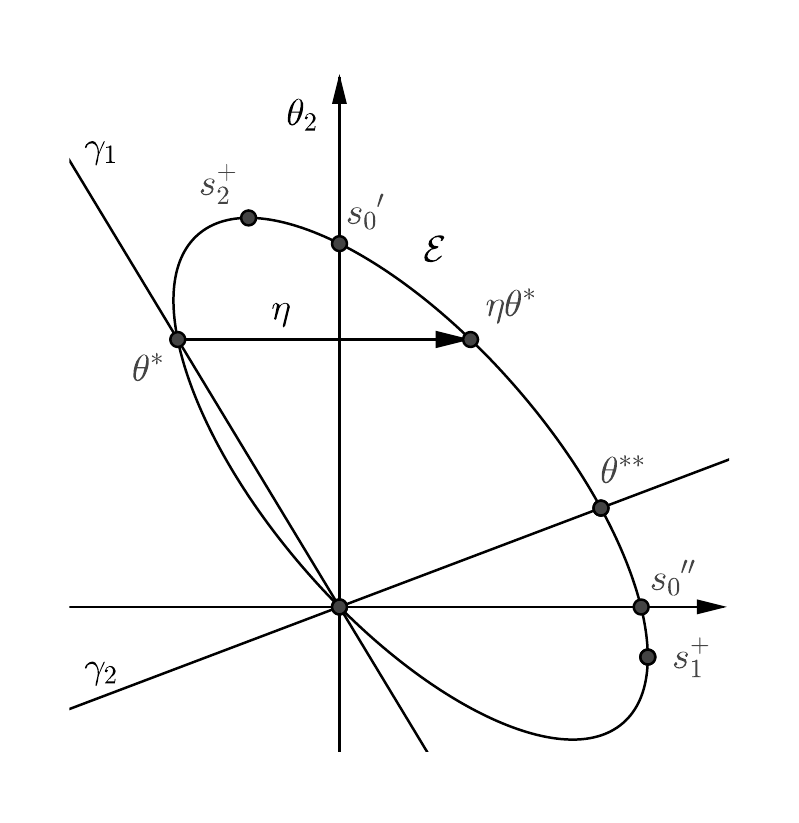}
  \caption{On the left  figure :   $\eta \theta^*$ is the  
   closest pole of $\phi_1$ to $s''_0$  on $]s''_0 ,s_2^+[$, 
  $\zeta \eta \theta^*$
is  the closest pole of $\phi_2$ to $s'_0$  on $]s_1^+, s_0[$. On the right figure:
 $\eta \theta^*$ is the  
   closest pole of $\phi_1$ to $s''_0$  on $]s''_0 ,s_2^+[$, 
  there are no poles of $\phi_2$  on $]s_1^+, s_0[$    }
  \label{fi12}
  \end{figure}

     We will also need the following two lemmas.

 \begin{lem}
\label{lem2}
  \begin{itemize}

\item[(1)]
     Assume that $\theta_2(s_1^+)>0$. Then for any $s \in ]\eta s_1^+, s_0[$ we have
  $\theta_2(\zeta \eta s)> \theta_2(\eta s)$.

\item[(2)]    Assume that $\theta_1(s_2^+)>0$. Then for any $s \in ]s_0,\zeta s_2^+ [$ we have
  $\theta_1(\eta \zeta s)> \theta_1(\zeta s)$.

\end{itemize}
 \end{lem}

\begin{proof}
 Since $\theta_2(s_1^+)>0$, then $\theta_2(\zeta \eta s_0) - \theta_2(\eta s_0)>0$.
   Consider the function  $f(s)=\theta_2(\zeta s)-\theta_2(s)$  for $s \in [s''_0, s_1^+]$.
   It depends continuously on $s$ on this arc.
    We note that  $f(s''_0)= \theta_2(\zeta \eta s_0)-\theta_2(\eta s_0)>0$,
   $f(s_1^+)=0$.
     Furthermore, since $s_1^-\not\in [s''_0,  s_1^+]$,
     then $f(s)\ne 0$ for all  $s \in ]s''_0, s^{1,+}]$.   Hence $f(s)>0$ for all  $s \in ]s''_0, s_1^+]$,
  from where  $\theta_2(\zeta \eta s)- \theta_2(\eta s) =f(\eta s)>0$ for any $s \in \eta ]s''_0, s_1^+]= [\eta s_1^+, s_0[$.
   The proof in the other case is symmetric.

\end{proof} 

\begin{lem}
\label{lem4}

 Assume that $\gamma_2(s)$  has a zero $\theta^{**} \in ]s_0, s_1^+[$ and 
$\gamma_1(s)$ has a  zero $\theta^* \in]s_2^+, s_0[$.
    Then one of the following three assertions holds true :
\begin{itemize}

\item[(i)]   The closest pole of $\phi_2(\theta_1(s))$ to $s'_0$ on  $]s_1^+, s'_0[$ is $\zeta \theta^{**}$,
    the closest pole of $\phi_1(\theta_2(s))$ to $s''_0$ on  $]s''_0, s_2^+[$  is $\eta \theta^*$, both of them are of the first order.

\item[(ii)]  The closest pole of $\phi_2(\theta_1(s))$ to $s'_0$ on  $]s_1^+, s'_0[$  is $\zeta \theta^{**}$, it is of the first order.
    The closest pole of $\phi_1(\theta_2(s))$ to $s''_0$ on  $]s''_0, s_2^+[$ is $\eta \zeta \theta^{**}$ where $ \theta_1(\eta \zeta \theta^{**})>
  \theta_1(\zeta \theta^{**})$.

\item[(iii)]  The closest pole of $\phi_2(\theta_1(s))$ to $s'_0$ on  $]s_1^+, s'_0[$  is $\zeta \eta \theta^{*}$ where $
\theta_2(\zeta \eta \theta^{*})> \theta_2(\eta \theta^{*})$.
    The closest pole of $\phi_1(\theta_2(s))$ to $s''_0$ on  $]s''_0, s_2^+[$ is $\eta \theta^{*}$, it is of the first order.

\end{itemize}

\end{lem}

The case (ii) is illustrated on Figure \ref{fi12}. 

\begin{proof}
By Lemma  \ref{lem3} there exist poles of the function $\phi_1(\theta_2(s))$ on $]s''_0, s_2^+[$.  By Lemma \ref{lem1} under parameters 
  such that $\theta_1(s_2^+)\leq 0$ or   $\theta_1(s_2^+)> 0$  and $\theta^{**} \not \in ]s_0, \zeta s_2^+[$, 
   $\eta \theta^*$ is the closest pole to $s''_0$ and it is of the first order. 
   By the same lemma under parameters 
   such that  $\theta_1(s_2^+)> 0$  and $\theta^{**} \in ]s_0, \zeta s_2^+[$,
    either $\eta \theta^{*}$  or 
   $\eta \zeta \theta^{**}$ is the closest pole to $s'_0$.   By Lemma \ref{lem1},
if $\gamma_1(\zeta \theta^{*})\ne 0$,
 pole $\eta \theta^*$ is of the first order. 
  Condition  $\gamma_1(\zeta\theta^*) \ne 0$ is equivalent to $\zeta \theta^* \ne \theta^{**}$.  This  means   
  just  that  pole   $\eta \theta^{*}$ is different from  
   $\eta \zeta \theta^{**}$  which is another candidate for the closest pole to $s''_0$. 
     By Lemma \ref{lem2}   $\theta_1(\eta\zeta \theta^{**})>\theta_1(\zeta \theta^{**})$  
     To summarize, one of two following statements holds true: 

\begin{itemize}

\item[(a1)]  Point $\eta \theta^*$  is the closest pole of $\phi_1(\theta_2(s))$  to $s''_0$ on $]s''_0, s_2^+[$ and it is of the first order;   

\item[(b1)]  The parameters  are  such that  $\theta_1(s_2^+)> 0$  and $\theta^{**} \in ]s_0, \zeta s_2^+[$.
 Point $\eta \zeta \theta^{**}$  is the closest pole of  $\phi_1(\theta_2(s))$  to $s''_0$ on $]s''_0, s_2^+[$  and 
$ \theta_1(\eta\zeta \theta^{**})> \theta_1(\zeta \theta^{**})$.
 
\end{itemize}

 By Lemmas \ref{lem3},  \ref{lem1}, \ref{lem2}  and the same considerations,  
one of the following statements about  $\phi_2(\theta_1)$   holds true: 

\begin{itemize}

\item[(a2)]
Pole $\zeta \theta^{**}$ is the closest pole  of $\phi_2(\theta_1(s))$  to $s'_0$ on $]s_1^+, s'_0[$ and it is of the first order. 

\item[(b2)]  
 The parameters are such that 
   $\theta_2(s_1^+)>0$, $\theta^* \in ]\eta s_1^+, s_0[$, 
point  $\zeta \eta \theta^{*}$ is the closest pole   of $\phi_2(\theta_1(s))$  to $s'_0$  on $]s_1^+, s'_0[$    and
   $\theta_2(\zeta\eta \theta^{*})>\theta_2(\eta \theta^{*})$
\end{itemize}

     Let us finally prove that (b1) and (b2) can not hold true simultaneously.
 Assume that $\theta_2(s_1^+)>0$, $\theta^{*} \in ]\eta s_1^+, s_0[$,
   $\theta_1(s_2^+)>0$, $\theta^{**} \in ]s_0, \zeta s_2^+[$  and e.g. (b2),
   that is  $\zeta \eta \theta^{*}$ is the closest pole  to $s_0$.
   Note that in this case $\zeta \theta^{**} \in ]s_2^+, s'_0[$. Then
   $\zeta \eta \theta^{*}$ is closer to $s'_0$ than the pole  $\zeta \theta^{**}$  on this segment or coincides with it.
           Hence  $\theta_1(\zeta \eta \theta^{*}) \leq \theta_1(\zeta \theta^{**})$ and
   $\theta_2(\zeta \eta \theta^{*}) \leq \theta_2 (\zeta \theta^{**})$. 
     By Lemma~\ref{lem2}
$\theta_1(\eta \theta^*) =\theta_1 (\zeta \eta \theta^{*})$ and  $\theta_1(\zeta \theta^{**})<\theta_1(\eta \zeta \theta^{**})$,
$\theta_2(\eta \theta^*) < \theta_2(\zeta \eta \theta^{*})$ and  $\theta_2(\zeta \theta^{**})=\theta_2(\eta \zeta \theta^{**})$.
   Then
$\theta_1(\eta \theta^*)<\theta_1(\eta \zeta \theta^{**})$,
$\theta_2(\eta \theta^*) < \theta_2(\eta \zeta \theta^{**})$.
   This means that that  $\eta \theta^*$ is the closest pole of $\phi_1(\theta_2(s))$ to $s''_0$, 
   $\eta\theta^* \ne \eta \zeta \theta^{**}$ , so that  (b1) is impossible for $\phi_1(\theta_2(s))$, then we have (a1).

  In the same way assumption (b1) leads to (a2).
     Thus (b1), (b2) can not hold true simultaneously, the lemma is proved.
     \end{proof}

\section{Contribution of the saddle-point and of the poles to the asymptotic expansion} 

\subsection{Stationary distribution density as a sum of integrals on \texorpdfstring{${\bf S}$}{S} }
\label{invlaplacetransf}

By the  functional equation (\ref{maineq})  and the  inversion formula of  Laplace transform
 (we refer to \cite{doetsch_introduction_1974} and
\cite{brychkov_multidimensional_1992}), 
       the density $\pi(x_1,x_2)$ can be represented as a double integral  
\begin{align}
\pi(x_1,x_2)
&=
\frac{1}{(2\pi i)^2}
\int_{-i\infty}^{i\infty} \int_{-i\infty}^{i\infty}
e^{-x_1\theta_1-x_2\theta_2} \varphi(\theta_1,\theta_2) 
\mathrm{d} \theta_1 \mathrm{d} \theta_2\nonumber
\\
&=\frac{-1}{(2\pi i)^2} \int_{-i\infty}^{i\infty} \int_{-i\infty}^{i\infty}
e^{-x_1\theta_1-x_2\theta_2}
\frac{\gamma_1(\theta) \varphi_1(\theta_2)+\gamma_2(\theta)\varphi_2(\theta_1)}{\gamma(\theta)}
\mathrm{d} \theta_1 \mathrm{d} \theta_2.\label{vbn}
\end{align}
 We now reduce it to a sum of single integrals.

\begin{lem}
 For any $(x_1, x_2) \in {\bf R}^2_+$  
$$
\pi(x_1,x_2)= I_1(x_1,x_2)+I_2(x_1,x_2)
$$
where
\begin{align}
I_1 (x_1, x_2) &=
\frac{1}{2\pi i} \int_{-i\infty}^{i\infty}
\varphi_2(\theta_1) \gamma_2(\theta_1, \Theta_2^+(\theta_1)) e^{-x_1\theta_1-x_2\Theta_2^+(\theta_1)}
\frac{\mathrm{d} \theta_1}{\sqrt{  d(\theta_1) }} \label{I1}
\end{align}
and
\begin{align}
I_2 (x_1, x_2) &=
\frac{1}{2\pi i} \int_{-i\infty}^{i\infty}
\varphi_1(\theta_2) \gamma_1(\Theta_1^+(\theta_2), \theta_2) e^{-x_1\Theta_1^+(\theta_2)-x_2\theta_2}
\frac{\mathrm{d} \theta_2}{\sqrt{  \tilde d (\theta_2) }}. \label{I2}
\end{align} \label{propI}
\end{lem}

\begin{proof}
By inversion formula  (\ref{vbn})  
   \begin{align*}
\pi(x_1,x_2)
&=
\frac{-1}{2\pi i} \int_{-i\infty}^{i\infty}
\varphi_2(\theta_1) e^{-x_1\theta_1} 
\Big(\frac{1}{2\pi i} \int_{-i\infty}^{i\infty} \frac{\gamma_2(\theta)}{\gamma(\theta)} e^{-x_2\theta_2} \mathrm{d} \theta_2\Big) \mathrm{d} \theta_1
\\ &+
\frac{-1}{2\pi i} \int_{-i\infty}^{i\infty}
\varphi_1(\theta_2) e^{-x_2\theta_2} 
\Big(\frac{1}{2\pi i} \int_{-i\infty}^{i\infty} \frac{\gamma_1(\theta)}{\gamma(\theta)} e^{-x_1\theta_1} \mathrm{d} \theta_1\Big) \mathrm{d} \theta_2.\\
\end{align*}
      Now it suffices to show the following formulas
\begin{equation}
\label{xxxp}
\frac{-1}{2\pi i} \int_{-i\infty}^{i\infty} \frac{\gamma_2(\theta)e^{-x_2\theta_2}}{\gamma(\theta)} \mathrm{d} \theta_2
=
 \frac{\gamma_2(\theta_1, \Theta_2^+(\theta_1))}{\sqrt{ d(\theta_1) }} e^{-x_2\Theta_2^+(\theta_1)}, 
\end{equation}
\begin{equation}
\label{yyyp} 
\frac{-1}{2\pi i} \int_{-i\infty}^{i\infty} \frac{\gamma_1(\theta)e^{-x_1\theta_1}}{\gamma(\theta)} \mathrm{d} \theta_1
=
 \frac{\gamma_1(\Theta_1^+(\theta_2), \theta_2)}{\sqrt{\tilde d(\theta_2) }} e^{-x_1\Theta_1^+(\theta_2)}.  
\end{equation}
   Let us prove (\ref{xxxp}).
  For any $\theta_1 \in i{\bf R} \setminus \{0\}$,  
    the function  $\gamma(\theta)=\frac{\sigma_{22}}{2}(\theta_2-\Theta_2^{+}(\theta_1))(\theta_2-\Theta_2^{-}(\theta_1))$
   has two zeros   $\Theta_2^{+}(\theta_1)$ and   $\Theta_2^{-}(\theta_1)$. Their real parts are of opposite signs:
 $\Re(\Theta_2^-(\theta_1))< 0$ and $\Re(\Theta_2^+(\theta_1))>0$.
  Thus  for any  fixed $\theta_1 \in i{\bf R}\setminus\{0\}$,   function 
$\frac{\gamma_2(\theta) e^{-x_2\theta_2}}{\gamma(\theta)}$ of the argument $\theta_2$ has two poles 
  on the complex plane ${\bf C}_{\theta_2}$, one at $\Theta^{-}_2(\theta_1)$ with  negative real part and another one 
     at $\Theta^{+}_2(\theta_1)$   with positive real part.
  Let us construct a contour $\mathcal{C}_R =[-iR,iR] \cup \{R e^{it} \mid t \in ]-\pi/2, \pi/2[\}$    composed of the purely imaginary segment 
$[-iR,iR]$ and   the half of the circle with radius 
$R$ and center $0$ on  ${\bf C}_{\theta_2}$, see Figure \ref{contourintresidu}. 
For $R$ large enough $\Theta^{+}_2(\theta_1)$ is inside the contour.  
The integral of $\frac{\gamma_2(\theta) e^{-x_2\theta_2}}{\gamma(\theta)}$ 
   over this contour taken in the counter-clockwise direction equals the residue at the unique pole of the integrand:   
\begin{equation*}
 \frac{1}{2\pi i}
\int_{\mathcal{C}_R} \frac{\gamma_2(\theta) e^{-x_2\theta_2}}{\gamma(\theta)} d\theta_2 =   {\rm res}_{\theta_2=\Theta^{+}_2(\theta_1)} 
    \frac{\gamma_2(\theta) e^{-x_2\theta_2}}{\gamma(\theta)}=   \frac{\gamma_2(\theta_1 ,\Theta^{+}_2(\theta_1))}{(\sigma_{22}/2) (  
   \Theta^{+}_2(\theta_1)-\Theta^{-}_2(\theta_1) )}e^{-x_2  \Theta^{+}_2(\theta_1)}
\end{equation*}
\begin{equation}
\label{jmmcb}
=  
 \frac{\gamma_2(\theta_1, \Theta^{+}_2(\theta_1))}{
   \sqrt{d(\theta_1)}} e^{-x_2  \Theta^{+}_2(\theta_1)} \  \   \  \hbox{ for all  large enough }R\geq 0.
\end{equation}

\begin{figure}[hbtp]
\centering
\includegraphics[scale=0.7]{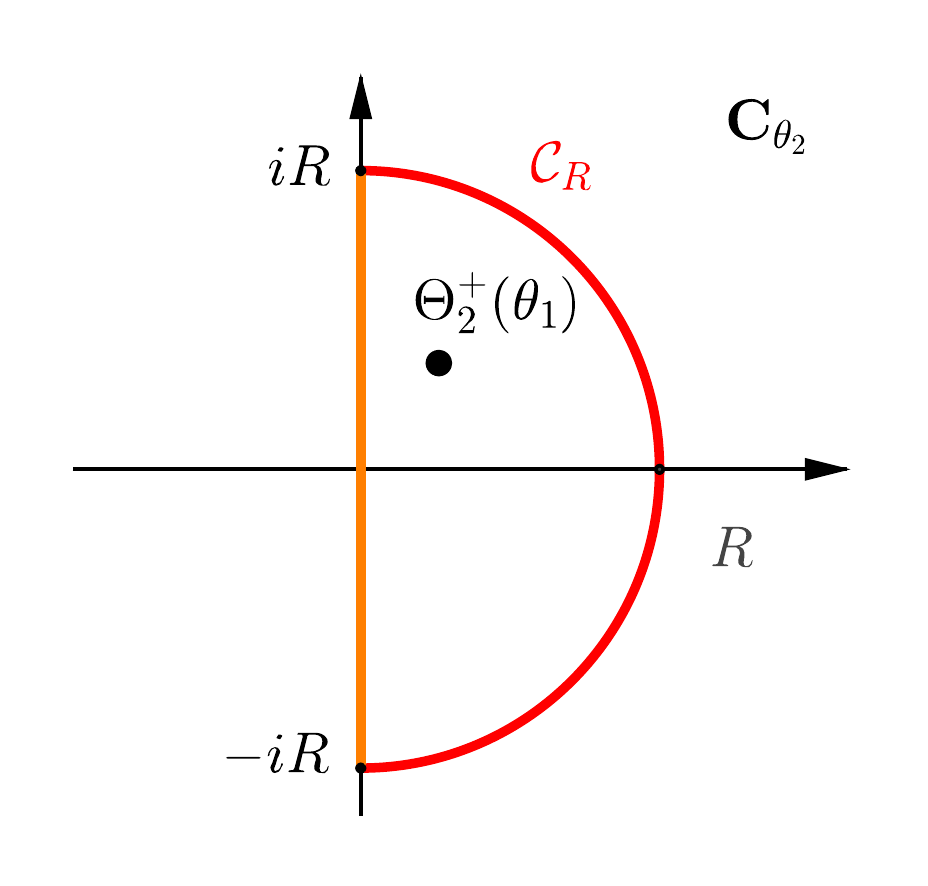} 
\caption{Contour $\mathcal{C}_R$ on  ${\bf C}_{\theta_2}$.}
\label{contourintresidu}
\end{figure}
Let us take the limit of this integral as $R \to \infty$: 
\begin{equation}
\lim_{R \to \infty} \int_{\mathcal{C}_R} \frac{\gamma_2(\theta) e^{-x_2\theta_2}}{\gamma(\theta)} d\theta_2 = 
- \lim_{R \to \infty} \int_{-i R}^{i R} \frac{\gamma_2(\theta) e^{-x_2\theta_2}}{\gamma(\theta)} d\theta_2 +
\lim_{R \to \infty} \int_{\{Re^{it}|t\in]-\frac{\pi}{2},\frac{\pi}{2}[\}} \frac{\gamma_2(\theta) e^{-x_2\theta_2}}{\gamma(\theta)} d\theta_2.
\label{eqax}
\end{equation}
The last term equals
\begin{equation}
\label{56}
\lim_{R \to \infty} \int_{\{Re^{it}|t\in]-\frac{\pi}{2},\frac{\pi}{2}[\}} \frac{\gamma_2(\theta) e^{-x_2\theta_2}}{\gamma(\theta)} d\theta_2=\lim_{ R \to  \infty } 
\int_{-\frac{\pi}{2}}^{\frac{\pi}{2}} \frac{\gamma_2(\theta_1,Re^{it}) }{\gamma(\theta_1,Re^{it})} e^{-x_2 Re^{it}} iR e^{it}d t.
\end{equation}
     We note that 
$\sup_{R>0}\sup_{t\in]-\frac{\pi}{2},\frac{\pi}{2}[} |iRe^{it} \frac{\gamma_2(\theta_1,Re^{it}) }{\gamma(\theta_1,Re^{it})}|<\infty $
  and 
$\sup_{R>0}\sup_{t\in]-\frac{\pi}{2},\frac{\pi}{2}[}|e^{-x_2 R e^{it}}|\leq 1$. Furthermore 
  $|e^{-x_2 Re^{it}}|=e^{-x_2 R\cos{t}}\to 0$  as $R\to \infty$  for all $t\in]-\frac{\pi}{2},\frac{\pi}{2}[$.
   Then by dominated convergence theorem  the limit  (\ref{56}) equals 0 as $R \to \infty$. 
Hence,  due to (\ref{jmmcb}) and (\ref{eqax}) 
$$     \frac{\gamma_2(\theta_1 ,\Theta^{+}_2(\theta_1))}{(\sigma_{22}/2) (  
   \Theta^{+}_2(\theta_1)-\Theta^{-}_2(\theta_1) )}e^{-x_2  \Theta^{+}_2(\theta_1)} =
\lim_{R \to \infty} \int_{\mathcal{C}_R} \frac{\gamma_2(\theta) e^{-x_2\theta_2}}{\gamma(\theta)} d\theta_2 = 
 \int_{-i \infty}^{i \infty} \frac{\gamma_2(\theta) e^{-x_2\theta_2}}{\gamma(\theta)} d\theta_2,$$
  that proves  (\ref{xxxp})  for any $\theta_1 \in i{\bf  R} \setminus \{0\}$. The proof of \eqref{yyyp} is analogous.

 Note also that the integral    
$$ \frac{1}{2\pi i} \int_{-i\infty}^{i\infty}
\varphi_2(\theta_1) \gamma_2(\theta_1, \Theta_2^+(\theta_1)) e^{-x_1\theta_1-x_2\Theta_2^+(\theta_1)}
\frac{\mathrm{d} \theta_1}{\sqrt{  d(\theta_1) }}  $$ 
is absolutely convergent. In fact $\sup_{\theta_1 \in i{\bf R}} |\phi_2(\theta_1)| \leq \nu_2({\bf R}^2_+)$ by definition of $\phi_2$. 
   It is elementary to see that  $\sup_{\theta_1 \in i{\bf R}} |\gamma_2(\theta_1, \Theta^+_2(\theta_1)) d^{-1/2} (\theta_1)|<\infty$.
   Furthermore,  for any $\theta_1 \in i{\bf R}$, 
    $\Re  \Theta_2^+(\theta_1) =\sigma_{22}^{-1} (-\mu_2 +\Re \sqrt{d(\theta_1)})$, thus for some constant $c>0$ we have
     $\Re  \Theta_2^+(\theta_1) > c |\Im \theta_1|$. Then the integral is absolutely convergent.  This concludes the proof of formula (\ref{I1}).    
   The proof of (\ref{I2}) is completely analogous. 
      \end{proof}

\paragraph{\bf Remark}
These integrals are equal to those on the Riemann surface ${\bf S}$  along properly oriented contours 
   ${\cal I}_{\theta_1}^+$ and ${\cal I}_{\theta_2}^+$ respectively. Thanks to the parametrization of Section \ref{param} we have
\begin{equation}
\label{sss}
\frac{\mathrm{d}\theta_1}{\sqrt{d(\theta_1)}}=\frac{\mathrm{d}\theta_2}{\sqrt{\tilde d(\theta_2)}}=\frac{i \mathrm{d}s}{s\sqrt{\det \Sigma }}.
\end{equation}
 Then we can write for $x=(x_1, x_2) \in{\bf R}_+^2$  the density $\pi(x_1, x_2)$ as a sum of  two integrals on ${\bf S}$ :
$$
I_1+I_2=
\frac{1}{2\pi \sqrt{\det \Sigma}} \int_{\mathcal{I}_{\theta_1}^+}
\frac{\varphi_2(s) \gamma_2(\theta(s))}{s}  e^{-\langle \theta(s) \mid x \rangle }
{\mathrm{d}s} +
\frac{1}{2\pi \sqrt{\det \Sigma}} \int_{\mathcal{I}_{\theta_2}^+}
\frac{\varphi_1(s) \gamma_1(\theta(s))}{s} e^{-\langle \theta(s) \mid x \rangle }
\mathrm{d}s.
$$

\subsection{Saddle-point}
\label{subsec:saddle}

Let us  fix $\alpha \in ]0, \pi/2[$ and  put $(x_1,x_2)=r e_\alpha= r(\cos (\alpha), \sin (\alpha))$ where $\alpha \in ]0, \pi/2[$, 
  Our aim now is to find the asymptotic  expansion of  $\pi(r\cos (\alpha), r\sin (\alpha))$, that is the one of  the sum
\begin{equation}
\label{iwi}
I_1+I_2=
\frac{1}{2\pi \sqrt{\det \Sigma}} \int_{\mathcal{I}_{\theta_1}^+}
\frac{\varphi_2(s) \gamma_2(\theta(s))}{s}  e^{- r \langle \theta(s) \mid e_\alpha \rangle }
{\mathrm{d}s}
+
\frac{1}{2\pi \sqrt{\det \Sigma}} \int_{\mathcal{I}_{\theta_2}^+}
\frac{\varphi_1(s) \gamma_1(\theta(s))}{s} e^{- r \langle \theta(s) \mid e_\alpha \rangle }
\mathrm{d}s 
\end{equation}  
as  $r \to \infty$ and to prove that  for any $\alpha_0 \in ]0, \pi/2[$ this asymptotic expansion is uniform in a small neighborhood 
  ${\cal O}(\alpha_0)\in ]0, \pi/2[$.   

      These integrals are typical to apply the saddle-point method, see \cite{fedoryuk_asymptotic_1989} or \cite{pemantle_analytic_2013}.    
  Let us study the function $\langle \theta(s) \mid e_\alpha \rangle $ on ${\bf S}$ and its critical points.
\begin{lem} \label{sp}

\begin{itemize}
\item[(i)]  For any $\alpha \in ]0, \pi/2[$  function  $\langle \theta(s) \mid e_\alpha \rangle $ 
  has two critical points  on ${\bf S}$ denoted by  $\theta^{+}(\alpha)$ and $\theta^{-}(\alpha)$. Both of them 
are on  ellipse $\mathcal{E}$, 
  $\theta^{+}(\alpha) \in  ]s_1^+, s_2^+[$,  $\theta^{-}(\alpha) \in ]s_1^-, s_2^-[$.
Both of them  are non-degenerate.

\item[(ii)]  The coordinates of $\theta^{+}(\alpha)=(\theta_1^+(\alpha), \theta_2^+(\alpha))$ are given by formulas : 
\begin{eqnarray}
\theta_1^\pm(\alpha)&=& \frac{\mu_2\sigma_{12}-\mu_1\sigma_{22}}{\det \Sigma}
\pm\frac{1}{\det \Sigma}
\sqrt{\frac{D_1}{1+\frac{\tan(\alpha)^2}{(\sigma_{22}-\tan(\alpha)\sigma_{12})^2} \det\Sigma}
}   \nonumber\\
\theta_2^\pm(\alpha)&=& \frac{\mu_1\sigma_{12}-\mu_2\sigma_{11}}{\det \Sigma}
\pm\frac{1}{\det \Sigma}
\sqrt{\frac{D_2}{1+\frac{\tan(\alpha)^2}{(\sigma_{11}-\tan(\alpha)\sigma_{12})^2} \det\Sigma}
} \label{cdpt}
\end{eqnarray}
  where notations
 $D_1 =(\mu_2\sigma_{12}-\mu_1\sigma_{22})^2 +\mu_2^2 \det{\Sigma}$
  and
 $D_2 =(\mu_1\sigma_{12}-\mu_2\sigma_{11})^2 +\mu_1^2 \det{\Sigma}$ are used.
With the parametrization of Section \ref{param} the corresponding points on $\bf S$ are such that: 
$$s_\pm(\alpha)^2=\frac{\cos \alpha (\theta_1^+ -\theta_1^-)+\sin \alpha (\theta_2^+ -\theta_2^-) e^{i\beta}}{\cos \alpha (\theta_1^+ -\theta_1^-)+\sin \alpha (\theta_2^+ -\theta_2^-) e^{-i\beta}}.$$

\item[(iii)]  Function $\theta^{+}(\alpha)$ is an isomorphism between $[0, \pi/2]$ and ${\cal A}=  [s^{1,+}, s^{2,+}]$, 
$\lim_{\alpha \to 0} \theta^{+}(\alpha)=s_1^+$, 
     $\lim_{\alpha \to \pi/2} \theta^{+}(\alpha)=s_2^+$.
   Function $\theta^{-}(\alpha)$   is an isomorphism between $[0, \pi/2]$ and $ [s^{1,-}, s^{2,-}]$, 
$\lim_{\alpha \to 0} \theta^{-}(\alpha)=s_1^-$, 
     $\lim_{\alpha \to \pi/2} \theta^{-}(\alpha)=s_2^-$.

\item[(iv)]  Function  $\langle  \theta(s) \mid e_\alpha \rangle$
  is  strictly {\it  increasing } on  the arc  $[\theta^{-}(\alpha), \theta^{+}(\alpha)]$ of $\mathcal{E}$  and strictly decreasing on
     the arc   $[\theta^{+}(\alpha), \theta^{-}(\alpha)]$. 
    Namely, $\theta^{+}(\alpha)$ is its maximum on $\cal{ E}$ and $\theta^{-}(\alpha)$ is its minimum:
 $$\theta^{+}(\alpha)=  {\rm argmax}_{s \in {\cal E}} \langle  \theta(s) \mid e_\alpha \rangle \   \   \  
     \theta^{-} (\alpha)=  {\rm argmin}_{s \in {\cal E}} \langle  \theta(s)  \mid e_\alpha \rangle.$$
 
\end{itemize}
\end{lem}

\begin{proof}
      Let us look for critical points with coordinates $(\theta_1, \theta_2)$ 
 of $\langle \theta(s) \mid e_\alpha \rangle $  on ${\bf S}$.
   Equation $ (\theta_1 \cos (\alpha) + \theta_2(\theta_1) \sin (\alpha) )'_{\theta_1}=0$ implies
   $\tan(\alpha) \frac{d \theta_2}{d \theta_1}=-1$.
   Substituting it into equation $\gamma(\theta_1, \theta_2(\theta_1))'_{\theta_1} \equiv 0$ and writing also $\gamma(\theta_1,\theta_2) \equiv 0$ 
 we get the system of two equations 
$$
\left\{
\begin{array}{l} 
 -\sigma_{11}\theta_1 \tan(\alpha) +\sigma_{22}\theta_2  +\sigma_{12}\theta_1 -\sigma_{12}\theta_2\tan(\alpha) 
   -\mu_1\tan(\alpha) +\mu_2 =0\\
\sigma_{11}\theta_1^2+\sigma_{22} \theta_2^2 +2\sigma_{12}\theta_1 \theta_2 +\mu_1 \theta_1 +\mu_2 \theta_2=0
\end{array}
\right.
$$
   from where we  compute  $\theta^{-}(\alpha)=(\theta^{-}_1(\alpha), \theta^{-}_2(\alpha))$ and 
 $\theta^{+}(\alpha)=(\theta^{+}_1(\alpha), \theta^{+}_2(\alpha))$ explicitly  as announced in (\ref{cdpt}).
  We  check directly that $\frac{d^2 \theta_2}{d \theta_1}\ne 0$ at these points, so they are non-degenerate critical points.
        It is also easy to see from (\ref{cdpt}) 
  that $\theta^{-}_1(\alpha)$ is strictly increasing from branch  point  $\theta_1^-$ to $\theta_1(\theta_2^-)$ and 
 that  $\theta^{+}_1(\alpha)$ is strictly decreasing  from  branch point $\theta_1^+$ to $\theta_1(\theta_2^+)$  when $\alpha$ runs 
  the segment $[0, \pi/2]$. In the same way $\theta^{-}_2(\alpha)$ is strictly decreasing from $\theta_2(\theta_1^-)$ to $\theta_2^-$ and 
     $\theta^{+}_2(\alpha)$ is strictly increasing  from $\theta_2(\theta_1^+)$ to $\theta_2^+$  when $\alpha$ runs 
  the segment $[0, \pi/2]$. This proves  assertions (i)--(iii).

       Finally, since there are no critical points on
    $\cal{ E}$ except for  $\theta^+(\alpha)$ and $\theta^-(\alpha)$, function 
   $\langle  \theta(s) \mid e_\alpha \rangle$ is monotonous on the arcs $[\theta^{-}(\alpha), \theta^{+}(\alpha)]$  and 
        $[\theta^{+}(\alpha), \theta^{-}(\alpha)]$.  In view of the inequality $\langle  \theta^+(\alpha) \mid e_\alpha \rangle >
  \langle  \theta^-(\alpha) \mid e_\alpha \rangle$, assertion (iv) follows.
 \end{proof}

\noindent{\bf Notation of the saddle-point.} 
From now one we are interested in point $\theta^{+}(\alpha)$ that we denote by $\theta(\alpha)$ for shortness. 

\noindent{\bf The steepest-descent contour  $\gamma_\alpha$.}
  The level curves $\{s : \Re \langle \theta(s) \mid e_\alpha\rangle=   \langle \theta(\alpha) \mid e_\alpha\rangle  \}$
  are orthogonal at $\theta(\alpha)$ and subdivide its neighborhood into four sections. 
  The curves of steepest descent   $\{s : \Im \langle \theta(s) \mid e_\alpha\rangle=0 \}$ on ${\bf S}$  are orthogonal at $\theta(\alpha)$ as well, see Lemma 1.3, Chapter IV in \cite{fedoryuk_saddle-point_1977}.
   One of them coincides with  $\cal{E}$.  
      We denote the other one by $\gamma_\alpha$.
 The real part $\Re  \langle \theta(s) \mid \alpha\rangle$ is strictly increasing on $\gamma_\alpha$ as $s$ goes far away
from $\theta(\alpha)$, see \cite[Section 4.2]{fedoryuk_asymptotic_1989}.
      The level curves of functions  $\Re \langle \theta(s) \mid e_\alpha\rangle$
  and  
$\Im \langle \theta(s) \mid e_\alpha\rangle$ are pictured in Figure \ref{ligneniveau}.

\begin{figure}[hbtp]
\centering
\includegraphics[scale=0.37]{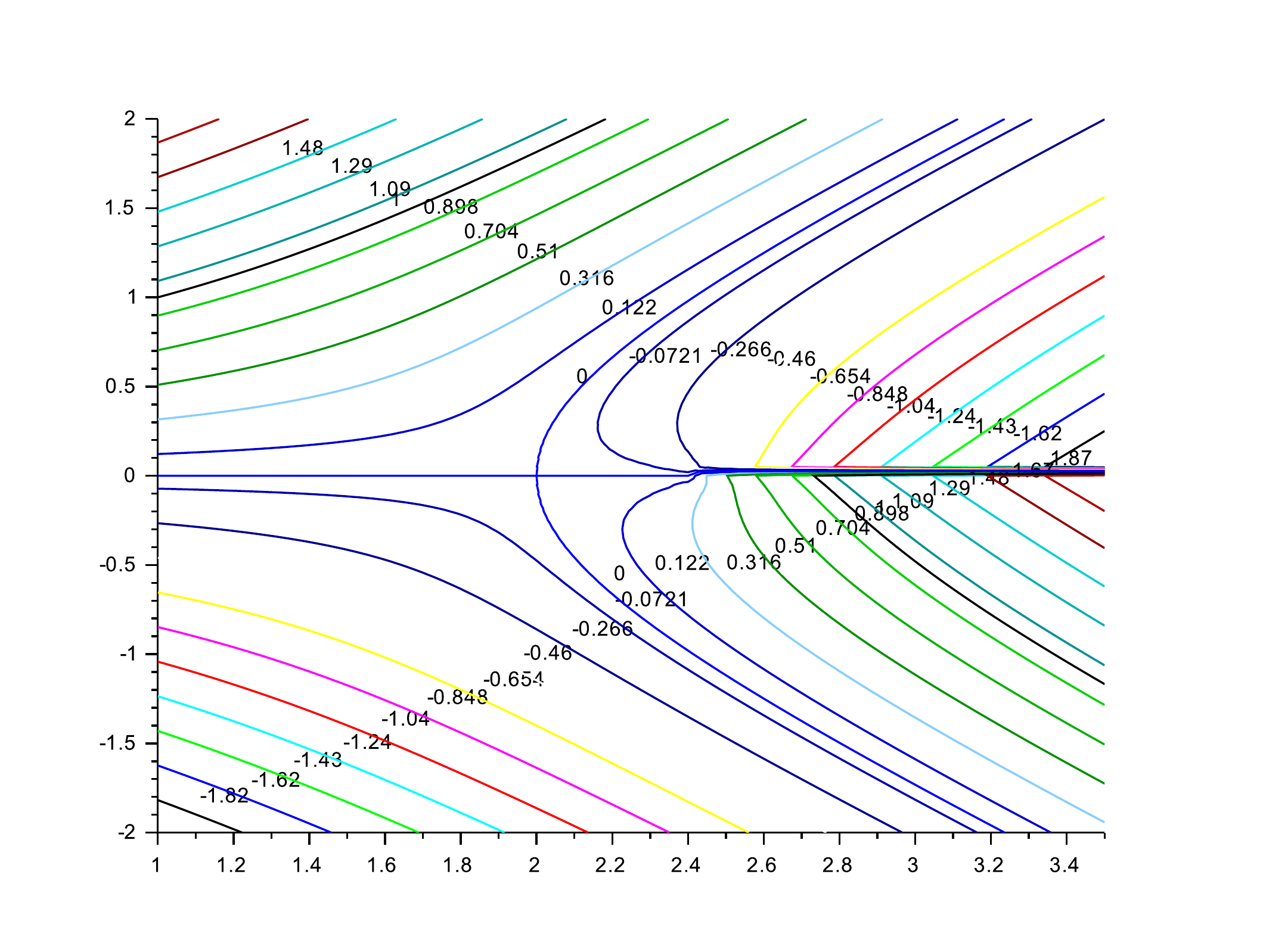}\hfill
\includegraphics[scale=0.37]{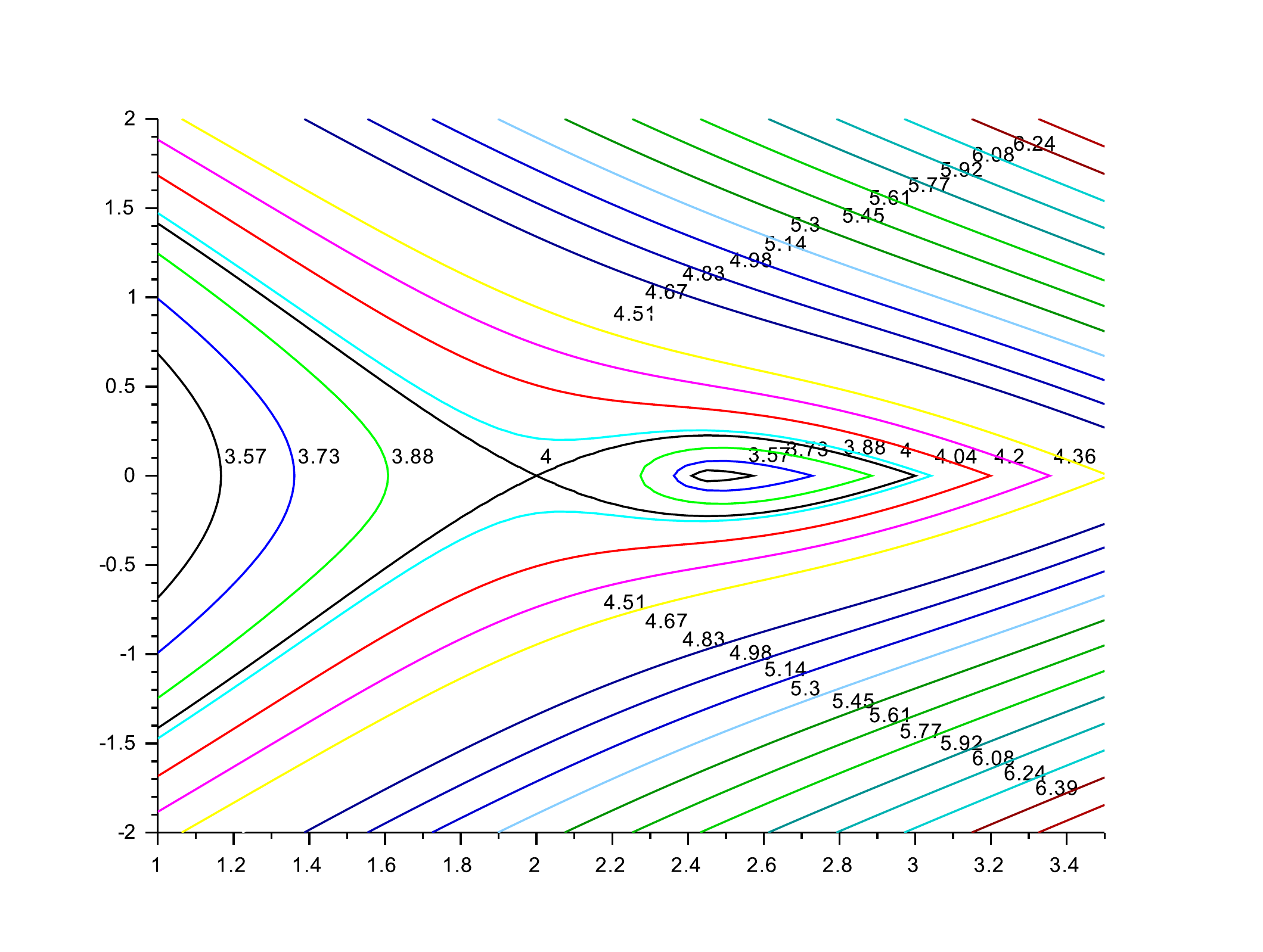}
\caption{Level sets of $\Im  \langle \theta(s) \mid e_\alpha \rangle $ and $\Re  \langle \theta(s) \mid e_\alpha \rangle$}
\label{ligneniveau}
\end{figure}

   Let $z_{\alpha,+}=  (\theta_1( z_{\alpha,+} ),  \theta_2(z_{\alpha,+}  )) $ and $z_{\alpha,-}=(\theta_1(  
  z_{\alpha,-} ), \theta_2  (z_{\alpha,-}  )) $ be the end points of $\gamma_\alpha$ where  $\Im \theta_1(  
  z_{\alpha,-} )>0$ and $\Im \theta_1(  
  z_{\alpha,-} )<0$.
      We can fix end points   
 $z_{\alpha , -}$   and $z_{\alpha , +}$ in such a way that  $\forall \alpha \in \mathcal{O}(\alpha_0)$ and  some small $\epsilon>0$ 
 $$\Re \langle z_{\alpha,\pm} \mid e_\alpha  \rangle =  \langle \theta(\alpha)  \mid e_\alpha \rangle +\epsilon.$$
   For technical reasons we  choose $\epsilon$ small enough such that $\Re \theta_1(z_{\alpha,\pm}) \in ]\theta_1^{-}, \theta_1^+[$ and 
    $\Re \theta_2(z_{\alpha,\pm}) \in ]\theta_2^{-}, \theta_2^+[$.

\subsection{Shifting the integration contours}
\label{subsec:shiftingcontours}

  Our aim now is to shift the integration contours
$\mathcal{I}_{\theta_1}^+$ and $\mathcal{I}_{\theta_2}^+$   in (\ref{iwi})  up to new contours 
$\Gamma_{\theta_1,\alpha}$ and $\Gamma_{\theta_2,\alpha}$  respectively which coincide with $\gamma_\alpha$ 
   in a neighborhood of $\theta(\alpha)$  on ${\bf S}$ and are  ``higher'' than $\theta(\alpha)$ in the sense 
  of level curves of the function $\Re \langle \theta(s) \mid e_\alpha \rangle$, that is  
  $\Re \langle \theta(s) \mid e_\alpha \rangle  > \Re \langle \theta(\alpha) \mid e_\alpha \rangle+\epsilon$
    for any $s \in \Gamma_{\theta_i,\alpha}\setminus \gamma_{\alpha}$ with $i=1,2$.
   When shifting  the contours we should of course take  into account the poles of the integrands and the residues at them.

  Let us construct $\Gamma_{\theta_1,\alpha}$ and $\Gamma_{\theta_2,\alpha}$.
    We set
$$\Gamma_{\theta_1, \alpha}^{1,+}=\{s : \Re \theta_1(s)= \Re \theta_1(z_{\alpha,+}),  \Im \theta_1(z_{\alpha,+}  )  \leq \Im \theta_1(s) \leq  V(\alpha) \}$$ 
  where $V(\alpha)>0$ will be defined later. 
       Then the end points of $\Gamma_{\theta_1, \alpha}^{1,+}$ are $z_{\alpha,+}$ and $Z_{\alpha,+}$ where 
     $\Re \theta_1(z_{\alpha,+})= \Re \theta_1(Z_{\alpha,+})$,  $\Im \theta_1(Z_{\alpha,+})= V(\alpha)$.
Next 
$$\Gamma_{\theta_1,\alpha}^{2,+}=\{s : \Im \theta_1(s)=V(\alpha), 0\leq  \Re \theta_1(s) \leq  \Re \theta_1(z_{\alpha,+})\}$$ 
  if $  \Re \theta_1(z_{\alpha,+}) >0$ and  
$$\Gamma_{\theta_1,\alpha}^{2,+}=\{s : \Im \theta_1(s)=V(\alpha), 0\geq  \Re \theta_1(s) \geq  \Re \theta_1(z_{\alpha,+})\}$$    
  if $  \Re \theta_1(z_{\alpha,+}) <0$.
   This contour
 goes from 
$Z_{\alpha,+}$ up to $Z^0_{\alpha, +}$  on ${\mathcal I}_{\theta_1}$ 
     with $\Re (\theta_1(s))= 0$,  $\Im (\theta_1(s))=V(\alpha)$.
    Finally $\Gamma_{\theta_1, \alpha}^{3,+}$ coincides with $\mathcal{I}_{\theta_1}^+$  from $Z_{\alpha, +}^0$ up to infinity :
$$\Gamma_{\theta_1, \alpha}^{3,+} =  \{s :  \Re \theta_1(s)=0,  \Im \theta_1(s)\geq  V(\alpha)\}.$$
    We define in the same way 
  $\Gamma_{\theta_1, \alpha}^{1,-}=\{s : \Re \theta_1(s)= \Re \theta_1(z_{\alpha,-}),   -V(\alpha) \leq \Im s \leq  \Im \theta_1(z_{\alpha,-} ) \}$. 
  The end points of $\Gamma_{\theta_1, \alpha}^{1,-}$ are $z_{\alpha,-}$ and $Z_{\alpha,-}$ where 
   $\Re \theta_1(z_{\alpha,-})= \Re \theta_1(Z_{\alpha,-})$,  $\Im \theta_1(Z_{\alpha,-})= -V(\alpha)$.
Next $\Gamma_{\theta_2,\alpha}^{2,-}=\{s : \Im \theta_1(s)=-V(\alpha), 0\leq  \Re \theta_1(s) \leq  \Re \theta_1(z_{\alpha,-})\}$
   or  $\Gamma_{\theta_2,\alpha}^{2-}=\{s : \Im \theta_1(s)=-V(\alpha), 0\geq  \Re \theta_1(s) \geq  \Re \theta_1(z_{\alpha,-})\}$
   according to the sign of $\Re \theta_1(z_{\alpha, -})$. It
 goes from 
$Z_{\alpha,-}$  to $Z^0_{\alpha, -}$  on ${\bf S}_{\theta_1}$
     with $\Re (\theta_1( Z^0_{\alpha, -} ))= 0$,  $\Im (\theta_1(Z^0_{\alpha, -}  ))= - V(\alpha)$.
    Finally $\Gamma_{\theta_1, \alpha}^{3,+}$ coincides with ${\mathcal I}_{\theta_1}^+$  from $Z_{\alpha, -}^0$ up to infinity.
  Then  contour $\Gamma_{\theta_1,\alpha}= \Gamma_{\theta_1,\alpha}^{3,-}\cup \Gamma_{\theta_1,\alpha}^{2,-} \cup \Gamma_{\theta_1,\alpha}^{1,-}\cup \gamma_\alpha \cup \Gamma_{\theta_1,\alpha}^{1,+} \cup \Gamma_{\theta_1,\alpha}^{2,+} \cup 
  \Gamma_{\theta_1,\alpha}^{3,+} \subset {\bf S}_{\theta_1}^1$. 
     One can  visualize this contour on  Figure \ref{contint} : in the left picture  it is drawn on parametrized ${\bf S}$, 
    in the right picture it is  projected on the complex plane ${\bf C}_{\theta_1}$ .

\begin{figure}[hbtp]
\centering
\includegraphics[scale=0.6]{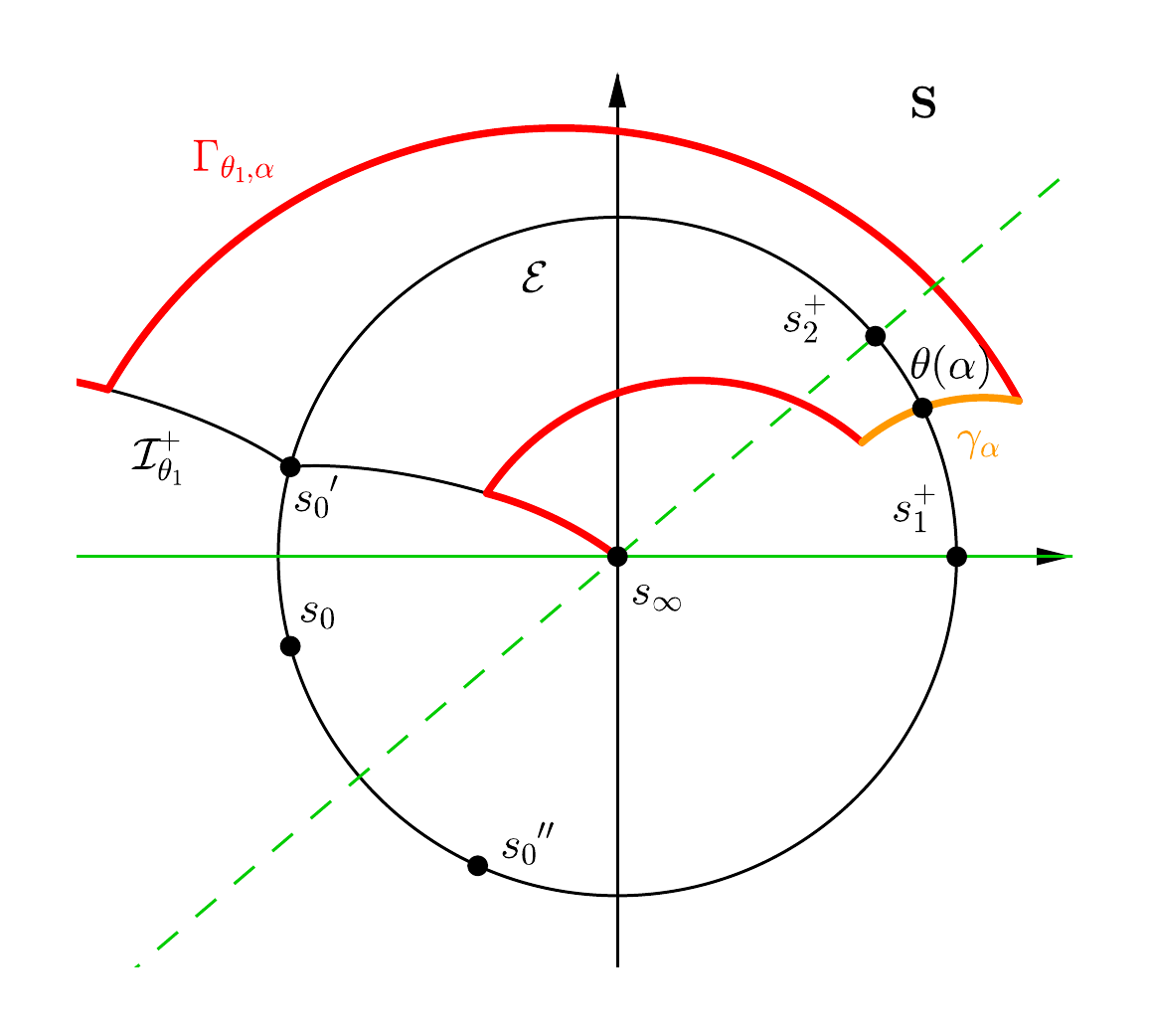}
\includegraphics[scale=0.6]{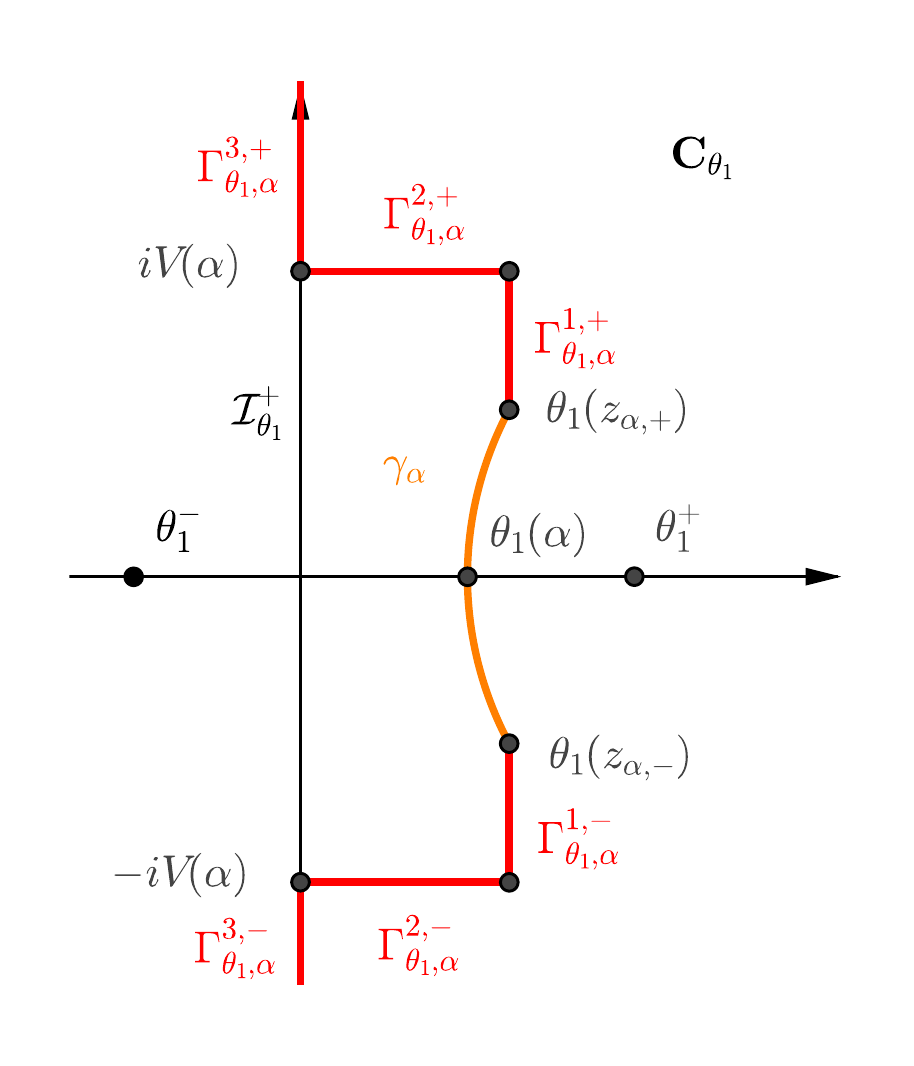} 
\caption{Contour $\Gamma_{\theta_1, \alpha}$  on parametrized ${\bf S}$  and projected on  ${\bf C}_{\theta_1}$.}
\label{contint}
\end{figure}

          The contour $\Gamma_{\theta_2, \alpha}$ is constructed analogously with respect to $\theta_2$-coordinate, 
   $\Gamma_{\theta_2,\alpha}= \Gamma_{\theta_2,\alpha}^{3,-}\cup \Gamma_{\theta_2,\alpha}^{2,-} \cup \Gamma_{\theta_2,\alpha}^{1,-}\cup \gamma_\alpha \cup \Gamma_{\theta_2,\alpha}^{1,+} \cup \Gamma_{\theta_2,\alpha}^{2,+} \cup 
  \Gamma_{\theta_2,\alpha}^{3,+} \subset {\bf S}_{\theta_2}^1$. The curve  of steepest descent $\gamma_\alpha$ is common for $\Gamma_{\theta_1, \alpha}$ 
   and $\Gamma_{\theta_2, \alpha}$. 

   Let us recall that poles of $\phi_1(s)$ and $\phi_2(s)$ on ${\bf S}$ may occur only at $\cal{E}$.
   Let us also recall the convention  that 
   an arc $\}a, b\{$ on $\cal{E}$ is the one with ends $a$ and $b$ which does not include $s_0=(0,0)$. 

\medskip 

\noindent{\bf Notation of the sets of poles $\mathcal{P}'_\alpha$ and $\mathcal{P}''_\alpha$.}
     Let ${\cal P'}_\alpha$ be the set of poles of the first order of the function $\phi_2(\theta_1(s))$ on the arc $\}\theta(\alpha), s'_0\{$.
  Let ${\cal P''}_\alpha$ be the set of poles of  the first order of the function
 $\phi_1(\theta_2(s))$ on  the arc $\}\theta(\alpha), s''_0\{$.
   
\medskip

  Then the following lemma holds true.

\begin{lem}
\label{ppmm}
Let $\alpha_0 \in ]0, \pi/2[$ be such that $\theta(\alpha_0)$  is not a pole of $\phi_1(\theta_2(s))$ neither of $\phi_2(\theta_1(s))$.
  If   ${\cal P'}_\alpha  \cup {\cal P''}_\alpha $ is not empty,    
   then for any $\alpha \in \mathcal{O}(\alpha_0)$ 
\begin{eqnarray}
\lefteqn{ \pi (r e_\alpha)  =
 \sum_{p \in {\cal P'}_\alpha} {\rm res}_p \phi_2(\theta_1(s))
\frac{\gamma_2(p)}{\sqrt{d(\theta_1(p))}} e^{-r \langle \theta(p) \mid e_\alpha \rangle } +
   \sum_{p \in {\cal P''}_\alpha} {\rm res}_p \phi_1(\theta_2(s))
\frac{\gamma_1(p)}{\sqrt{\tilde d(\theta_2(p))}} e^{-r \langle \theta(p) \mid e_\alpha \rangle } }\nonumber \\  
&&{} + \frac{1}{2\pi \sqrt{\det \Sigma}}
\big( \int_{\Gamma_{\theta_1, \alpha} }
\frac{\varphi_2(s) \gamma_2(\theta(s))}{s}  e^{- r \langle \theta(s) \mid e_\alpha \rangle }
{\mathrm{d}s}
+
\int_{\Gamma_{\theta_2, \alpha} }
\frac{\varphi_1(s) \gamma_1(\theta(s))}{s} e^{- r \langle \theta(s) \mid e_\alpha \rangle }
\big)
\mathrm{d}s.\label{prolongam}
\end{eqnarray}
  If  $\mathcal{ P'}_\alpha \cup
  \mathcal{ P''}_\alpha$  is empty, representation  (\ref{prolongam})  stays valid where the corresponding sums over $ p \in {\cal P'}_\alpha$ and 
   $p \in {\cal P''}_\alpha$
 are omitted. 
\end{lem}

\begin{proof}
   It follows from the assumption of the lemma that $\theta(\alpha)$ is 
not a pole  of $\phi_1(\theta_2(s))$ neither of $\phi_2(\theta_1(s))$ for any 
   $\alpha$ in  a small enough neighborhood  $\mathcal{O}(\alpha_0)$. 
  Then  we use  the representation of the density  (\ref{iwi}) and apply Cauchy theorem shifting the contours 
  to $\Gamma_{\theta_1, \alpha}$ and $\Gamma_{\theta_2, \alpha}$.
\end{proof}

In order to find  the asymptotic expansion of the density $\pi(r\cos(\alpha), r \sin(\alpha)) $,
    we have to evaluate now the contribution of the residues at poles in (\ref{prolongam}) 
  and the one of integrals along shifted contours $\Gamma_{\theta_1,\alpha}$ and $\Gamma_{\theta_2, \alpha}$.
  This is the subject of the next two sections. 
 
\subsection{Asymptotics of integrals along shifted contours  \texorpdfstring{$\Gamma_{\theta_1,\alpha}$}{Gtheta1} and \texorpdfstring{$\Gamma_{\theta_2, \alpha}$}{Gtheta2}}
\label{subsec:asymptintalongshift}
 
  To finish the construction of  $\Gamma_{\theta_1,\alpha}$ and $\Gamma_{\theta_2, \alpha}$, it remains to specify $V(\alpha)$.
   For that purpose we consider closer the function 
$$f_\alpha(s)=\langle \theta(s) \mid e_\alpha \rangle = \theta_1(s) \cos \alpha + \theta_2(s) \sin \alpha.$$
   Let us define the projection of this function on ${\bf C}_{\theta_1}$ :  
$$f_\alpha(\theta_1)= \theta_1 \cos \alpha +\Theta_2^{+}(\theta_1) \sin \alpha ,   \  \   \theta_1 \in {\bf C}_{\theta_1}.$$
   Clearly $f_\alpha (s)= f_\alpha (\theta_1(s))= \langle \theta(s) \mid e_\alpha \rangle $ on ${\bf S}_{\theta_1}^1$.

\begin{lem}
\label{uv}
\begin{itemize}
\item[(i)]
For any fixed $u \in  [\theta_1^-,\theta_1^+]$  the function $ v \to \Re (f_\alpha (u+iv))$ is increasing 
  on $[0, \infty[$ and decreasing on  $]-\infty, 0]$.
\item[(ii)]  There exist  constants $d_1\leq 0$, $d_2>0$ and $V>0$ such that:
\begin{equation} \label{tra}
\inf_{u\in [\theta_1^-,\theta_1^+]} \Re (f_\alpha (u+iv)) \geq d_1+d_2 \sin (\alpha) |v|   \  \  \forall v\geq V  \ \hbox{and } \forall v\leq -V, \  \  \forall \alpha \in ]0, \pi/2[.
\end{equation} 
\end{itemize}
\end{lem}

\begin{proof}  We compute :
 $$\Re (f_\alpha (u+iv))=\cos(\alpha) u+\frac{\sin(\alpha)}{\sigma_{22}} (-\sigma_{12} u-\mu_2 +\Re \sqrt{d(u+iv)})$$
      with the discriminant $d(u+iv)=(\det \Sigma) (u+iv-\theta_1^-)(\theta_1^+ -u-iv)$. Then  
$$
\Re \sqrt{d(u+iv)}=\sqrt{\det\Sigma}\sqrt{|(u+iv-\theta_1^-)(\theta_1^+ -u-iv)|} \cos (\frac{\omega_-(u+iv) +\omega_+(u+iv)}{2})
$$
where $\omega_-(u+iv)$ et $\omega_+(u+iv)$  are defined as
 $\omega_-(u+iv)=\arg (\theta_1^+ -u-iv)$  and $\omega_+(u+iv)=\arg (u+iv-\theta_1^-)$, see Figure \ref{angleomega}. 
    We have 
$$
 \cos (\frac{\omega_-(u+iv) +\omega_+(u+iv)}{2})=
\sqrt{  \frac{1}{2} \cos   ( \omega_-(u+iv) +\omega_+(u+iv) ) +\frac{1}{2}  }$$
\begin{align*}
&=\sqrt{   \frac{1}{2} \cos   ( \omega_-(u+iv) ) \cos( \omega_+(u+iv) )  -     \frac{1}{2} \sin   ( \omega_-(u+iv) ) \sin ( \omega_+(u+iv) ) +
   \frac{1}{2}  }
\\   
  &=\sqrt{ \frac{(u-\theta_1^-)(\theta_1^+-u)- v (-v)}{ 2    |(u+iv-\theta_1^-)(\theta_1^+ -u-iv)|}  +\frac{1}{2} }.
\end{align*}
  Thus  
\begin{multline*} \Re (f_\alpha (u+iv))=\cos(\alpha) u
\\ +\frac{\sin(\alpha)}{\sigma_{22}} \Big(-\sigma_{12} u-\mu_2
   +\sqrt{ \frac{1}{2} } \sqrt{ (u-\theta_1^-)(\theta_1^+-u)+v^2 +   |(u+iv-\theta_1^-)(\theta_1^+ -u-iv)|}\Big)
 \end{multline*}
   Both statements of the lemma follow directly from this representation.  \end{proof}

 
  
\begin{figure}[hbtp]
\centering
\includegraphics[scale=1.6]{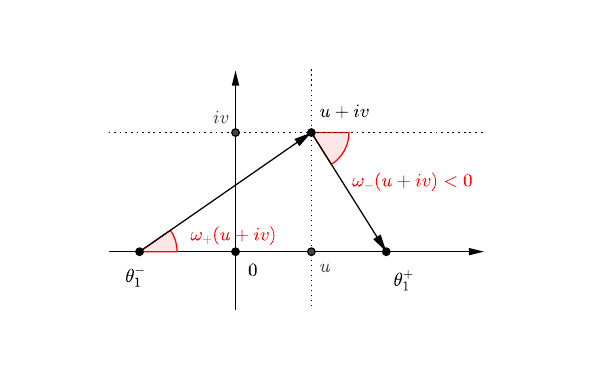}
\caption{$\omega_- (u+iv)$ et $\omega_+ (u+iv)$}
\label{angleomega}
\end{figure}

   We may now choose $V(\alpha)$ and such that 
\begin{equation}
\label{ddddd}
V(\alpha) = \max \Big( V, \ \frac{\langle \theta(\alpha) \mid e_\alpha \rangle +\epsilon-d_1}{d_2 \sin (\alpha)}\Big)
\end{equation}
 in accordance with notations of Lemma \ref{uv}.  This concludes the construction of $\Gamma_{\theta_1,\alpha}$ and $\Gamma_{\theta_2, \alpha}$.
    
The asymptotic expansion of integrals along  these contours is given in the following lemma. The main contribution comes from 
  the integrals along $\gamma_\alpha$, while all other parts of integrals are proved to be exponentially negligible by construction.

\begin{lem}
\label{saddleas}
   Let $\alpha_0 \in ]0, \pi/2[$ and $\mathcal{O}(\alpha_0)$ a small enough neighborhood of $\alpha_0$.  Then
   when
   $r \to \infty$ uniformly  for $\alpha \in \mathcal{O}(\alpha_0)$ we have
\begin{equation}
\label{assh1}
\frac{1}{2\pi \sqrt{\det \Sigma}} \int_{\Gamma_{\theta_1, \alpha} }
\frac{\varphi_2(s) \gamma_2(\theta(s))}{s}  e^{- r \langle \theta(s) \mid e_\alpha \rangle }
{\mathrm{d}s}   \sim \sum_{l=0}^{k} \frac{c^l_{\theta_1} (\alpha)}{r^l\sqrt{ r}}  
  e^{-r \langle \theta(\alpha) \mid e_\alpha \rangle },  
\end{equation}
\begin{equation}
\label{assh2}
\frac{1}{2\pi \sqrt{\det \Sigma}} \int_{\Gamma_{\theta_2, \alpha} }
\frac{\varphi_1(s) \gamma_1(\theta(s))}{s} e^{- r \langle \theta(s) \mid e_\alpha \rangle }
\mathrm{d}s \sim
 \sum_{l=0}^{k} \frac{c^l_{\theta_2} (\alpha)}{r^l\sqrt{ r}}  
  e^{-r \langle \theta(\alpha) \mid e_\alpha \rangle }.
\end{equation} 
     The constants $c^l_{\theta_1}(\alpha)$, $c^l_{\theta_2}(\alpha)$, $l=0,1,2, \ldots$  depend continuously of $\alpha$ and can be made 
  explicit in terms of functions $\phi_1$ and $\phi_2$ and their derivatives at $\theta(\alpha)$. 
   Namely
\begin{eqnarray}
c^0_{\theta_1}(\alpha) &=&  \frac{1}{\sqrt{2\pi \det\Sigma}} \frac{\varphi_2(s(\alpha))\gamma_2(\theta(\alpha))}{s(\alpha)\sqrt{ f_\alpha ''(s(\alpha))}}, \nonumber\\
c^0_{\theta_2}(\alpha)&=& 
 \frac{1}{\sqrt{2\pi \det\Sigma}} \frac{\varphi_1(s(\alpha))\gamma_1(\theta(\alpha))}{s(\alpha)\sqrt{ f_\alpha ''(s(\alpha))}}. \nonumber 
\end{eqnarray}
\end{lem}

\begin{proof}
    By Lemma~\ref{uv}  (i)  and  by (\ref{sss})  for any $r>0$.
\begin{equation}
\label{z}
 \Big|\int_{\Gamma_{\theta_1,\alpha}^{1,\pm} }
         \frac{ \phi_2(\theta_1(s)) \gamma_2(s)}{ s\sqrt{\det \Sigma}} \exp^{-r\langle \theta(s) \mid e_\alpha  \rangle }{\rm d}\,s \Big| \leq 2 V(\alpha)  \sup_{s \in  \Gamma_{\theta_1,\alpha}^{1,\pm}    } \Big|\frac{    \phi_2(\theta_1(s)) \gamma_2(s)  }{\sqrt{d(\theta_1(s))} }\Big| e^{-r \langle \theta(\alpha) \mid e_\alpha \rangle -r\epsilon  }.
\end{equation}

    The length of $\Gamma_{\theta_2,\alpha}^{\pm}$ being smaller than $(\theta_1^+-\theta_1^-)$, 
        by Lemma~\ref{uv} (ii)   and  by (\ref{sss})  for any $r>0$ 
\begin{equation}
\label{zz}
\Big|\int_{\Gamma_{\theta_1,\alpha}^{2,\pm} }
         \frac{ \phi_2(\theta_1(s)) \gamma_2(s) }{ s  \sqrt{\det\Sigma}  } \exp^{-r\langle \theta(s) \mid e_\alpha  \rangle }{\rm d}\, s \Big|  \leq (\theta_1^+ -\theta_1^-) \sup_{s \in  \Gamma_{\theta_2,\alpha}^{2,\pm}    } \Big|\frac{    \phi_2(\theta_1(s)) \gamma_2(s)  }{\sqrt{d(\theta_1(s))} } \Big| e^{-r(d_1 +d_2 \sin (\alpha) V(\alpha))
 }
\end{equation}
  where due to  the choice \eqref{ddddd} of $V(\alpha)$
\begin{equation}
\label{zzz}
 e^{-r (d_1+d_2 \sin(\alpha) V(\alpha)) }  \leq  e^{-r \langle \theta(\alpha) \mid e_\alpha \rangle -r\epsilon  }. 
\end{equation}
    Finally note that for any $s \in \Gamma_{\theta_1,\alpha}^{3,\pm}$ 
  $$\frac{\gamma_2(s)}{\sqrt{d(\theta_1(s) )} } =
 r_{1,2} \frac{ \theta_1(s)}{ \sqrt{d(\theta_1(s) )}   }  + r_{22} \frac{-b(\theta_1(s)) +  \sqrt{d(\theta_1(s))}  }{  2 a(\theta_1(s))  \sqrt{d(\theta_1(s) )}    }
     $$ where $\Re \theta_1(s)=0$,  $\Im \theta_1(s) \geq V$. 
   Then there exists a constant $D>0$ such that $| \gamma_2(s) d^{-1/2}(\theta_1(s))|\leq D$ for any $s \in  \Gamma_{\theta_1,\alpha}^{3,\pm}$ 
  and any $\alpha \in ]0, \pi/2[$. Moreover $|\phi_2(\theta_1(s))|\leq \nu_1({\bf R_+}) $ for any $s \in {\mathcal I}_{\theta_1}$.
   Thus by Lemma~\ref{uv} (ii)  and by (\ref{sss})
  \begin{equation}
\label{zzzz}
\Big|\int_{\Gamma_{\theta_1,\alpha}^{3,\pm} }
        \frac{  \phi_2(\theta_1(s)) \gamma_2(s) }{s \sqrt{ \det \Sigma } }  \exp^{-r\langle \theta(s) \mid e_\alpha  \rangle } {\rm d}\,s \Big| \leq
  2  D \nu_1({\bf R_+})  \int\limits_{V(\alpha)}^\infty e^{- r (d_1+d_2 \sin (\alpha) v) }dv $$
$$    \leq  2 D \nu_1({\bf R_+})  \frac{1}{ c \sin (\alpha) V(\alpha)} 
  e^{-r(d_1+d_2  \sin (\alpha) V(\alpha) ) }
   \leq  2   D \nu_1({\bf R_+})  \frac{1}{ c \sin (\alpha) V(\alpha)}   e^{-r \langle \theta(\alpha) \mid e_\alpha \rangle -r\epsilon  }.
\end{equation}

   The contours $\Gamma_{\theta_1, \alpha}^{i,\pm}$
   for $i=1,2$ 
   being far away from poles of $\phi_2$ and  zeros of $d(\theta_1(s))$ for all  $\alpha \in {\cal O}(\alpha_0)$,
   $\sup_{\alpha \in {\cal O}(\alpha_0)  } \sup_{s \in  \Gamma_{\theta_1,\alpha}^{i,\pm}    } \Big|\frac{    \phi_2(\theta_1(s)) \gamma_2(s)  }{\sqrt{d(\theta_1(s))} }\Big|<\infty$   for $i=1,2$, and   of course $ \sup_{\alpha \in {\cal O}(\alpha_0) } (\sin(\alpha) V (\alpha))^{-1}$ and  
    $\sup _{\alpha \in {\cal O}(\alpha_0)} V(\alpha)$ are finite as well.
   It follows that for some constant $C>0$ , any $r>0$ and any $\alpha \in {\cal O}(\alpha_0)$ 
\begin{equation}
\label{bbbbb}
 \Big|\int_{\Gamma_{\theta_1,\alpha}^{1,\pm} \cup  \Gamma_{\theta_1,\alpha}^{2,\pm} \cup \Gamma_{\theta_1,\alpha}^{3,\pm}    }
         \frac{ \phi_2(\theta_1(s)) \gamma_2(s)}{s \sqrt{\det \Sigma } } e^{-r\langle \theta(s) \mid e_\alpha  \rangle }{\rm d}\, s \Big|  
     \leq C  e^{-r \langle \theta(\alpha) \mid e_\alpha \rangle -r\epsilon  }.
\end{equation}
    As for the contour  $\gamma_\alpha$ of the steepest descent of the function $\langle \theta (s) \mid e_\alpha \rangle$, 
  we apply the standard saddle-point method, see  e.g. Theorem 1.7,  Chapter IV in \cite{fedoryuk_saddle-point_1977}:  for any $k>0$ when $r \to \infty$, uniformly  $\forall \alpha \in {\mathcal O}(\alpha_0)$,
\begin{equation}
\label{sspp}
\frac{1}{2\pi\sqrt{\det \Sigma}}\int_{\gamma_\alpha}  \frac{\phi_2(s) \gamma_2(\theta(s))}{s} e^{-r\langle \theta(s) \mid e_\alpha  \rangle }{\rm d}\,s
   \sim \sum_{l=0}^{k} \frac{c^l_{\theta_1}(\alpha) }{r^l\sqrt{ r}}      e^{-r \langle \theta(\alpha) \mid e_\alpha \rangle }, 
\end{equation}
   where 
$c^0_{\theta_1}(\alpha)$ is given explicitly in the statement  of the  lemma 
  and all other constants $c^l_{\theta_1}(\alpha)$  can be written in terms of the same functions and their derivatives at $\theta(\alpha)$.
    Thus (\ref{assh1}) is proved  and the proof of (\ref{assh2})  for the integral over $\Gamma_{\theta_2, \alpha}$ is absolutely analogous.
\end{proof}

\subsection{Contribution of poles into the  asymptotics of \texorpdfstring{$\pi(r \cos(\alpha), r \sin(\alpha))$}{pi r alpha}}
\label{subsec:contributionpole}

     Once Lemma \ref{saddleas} established the asymptotics of integrals along shifted contours $\Gamma_{\theta_1, \alpha}$ 
  and $\Gamma_{\theta_2, \alpha}$, let us come back to Lemma \ref{ppmm} and evaluate the contribution to the density of residues at  poles 
  over $\mathcal{P}'_\alpha \cup \mathcal{P}''_\alpha$.
   There are two possibilities:
\begin{itemize}
 
\item [(i)]  $\mathcal{P}'_\alpha \cup \mathcal{P}''_\alpha$
 is empty, then the asymptotics of the density is determined by the saddle-point via  Lemma~\ref{saddleas}.  

\item[(ii)]    $\mathcal{P}'_\alpha \cup \mathcal{P}''_\alpha$
 is not empty. Then due to monotonicity of the function $\langle \theta(s) \mid e_\alpha \rangle $ on $\mathcal{E}$, 
  see Lemma \ref{sp} (iv), for any $p \in \mathcal{P}'_\alpha \cup \mathcal{P}''_\alpha$  we have 
  $\langle \theta(p) \mid e_\alpha \rangle < \langle \theta(\alpha) \mid e_\alpha \rangle$.  Hence  {\it all} residues 
     at poles   $p \in \mathcal{P}'_\alpha \cup \mathcal{P}''_\alpha$  bring more important  contribution  to the asymptotic expansion  
as $r \to \infty$ than integrals over  $\Gamma_{\theta_1, \alpha}$ 
  and $\Gamma_{\theta_2, \alpha}$. 

\end{itemize}
      First of all, we would like to distinguish the set of parameters $(\Sigma, \mu , R)$ under  which
  (i) or (ii)  hold true.  Secondly, under (ii), we would like to find the most important pole from the asymptotic point of view. 
               Let us  look  closer at the arc $\{s'_0, \theta(\alpha) \}$.
        Under parameters such that $\theta_1(s_2^+)<0$     
 we have   $s'_0 \in  ]s_1^+, s_2^+[$, see Figure \ref{fipoles}, the left picture.
     Then for some $\alpha' \in ]0, \pi/2[$     $\theta(\alpha')=s'_0$.
     This arc written in square brackets  in the anticlockwise direction is $]s'_0, \theta(\alpha)[$ for any $\alpha \in ]\alpha', \pi/2[$ 
     and the function $\langle  \theta(s) \mid e_\alpha \rangle$ is increasing when $s$ runs from $s'_0$ to  $\theta(\alpha)$.
     For any $\alpha \in [0, \alpha'[$  this arc is written 
    $]\theta(\alpha), s'_0[$   
   and the function $\langle  \theta(s) \mid e_\alpha \rangle$  is decreasing when $s$ runs from $\theta(\alpha)$ so $s'_0$.
      Under parameters such that $\theta_1(s_2^+)\geq 0$,        
 we have   $s'_0 \not\in  ]s_1^+, s_2^+[$,  see Figure \ref{fipoles} the right picture,  from where this arc is   written    
    $]\theta(\alpha), s'_0[ $   for any $\alpha \in ]0, \pi/2[$. The function 
     $\langle  \theta(s) \mid e_\alpha \rangle$  is decreasing when $s$ runs from $\theta(\alpha)$ to $s'_0$.

    \begin{figure}[hbtp]
  
    \centering
    \includegraphics[scale=0.5]{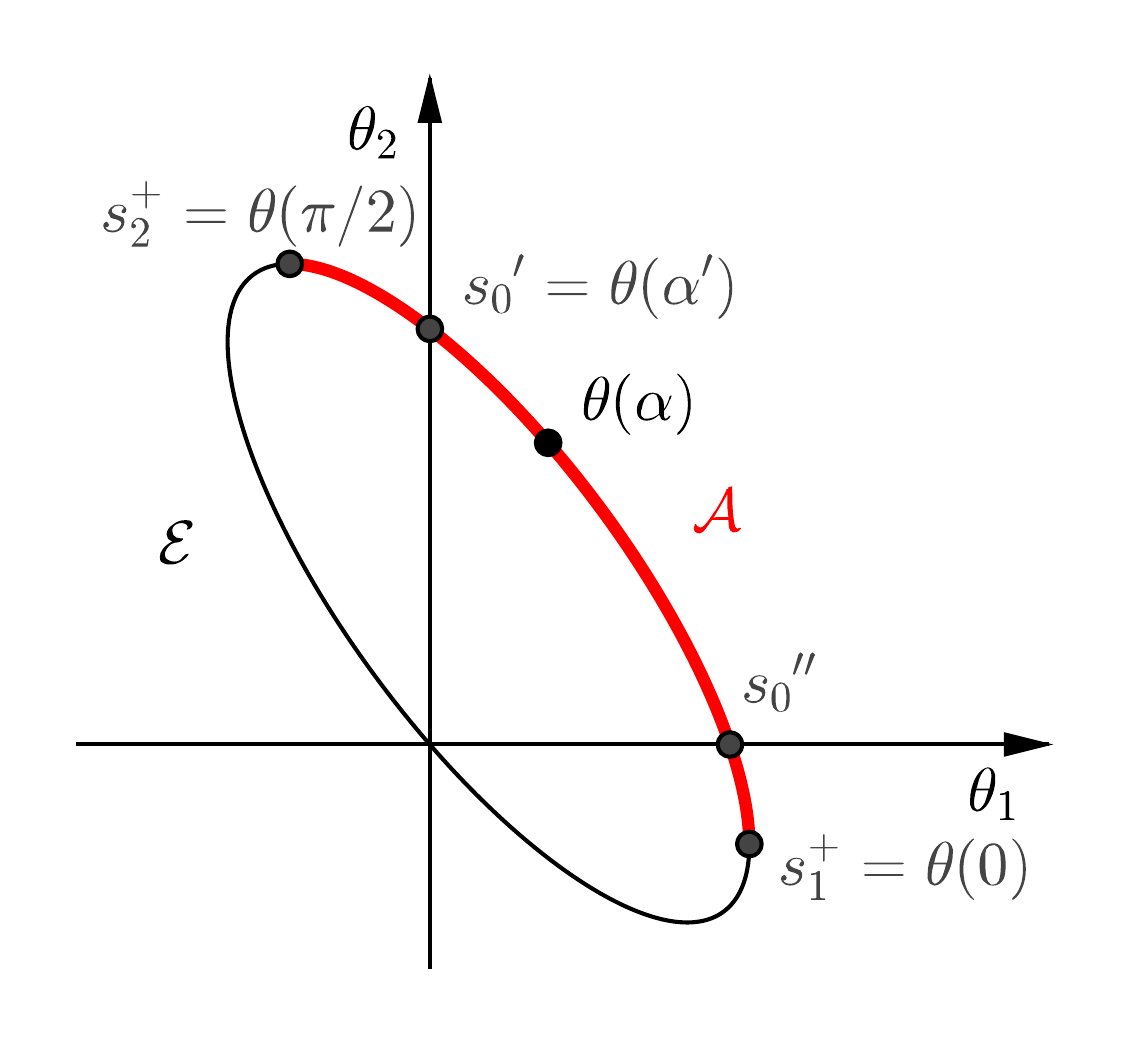}
     \includegraphics[scale=0.5]{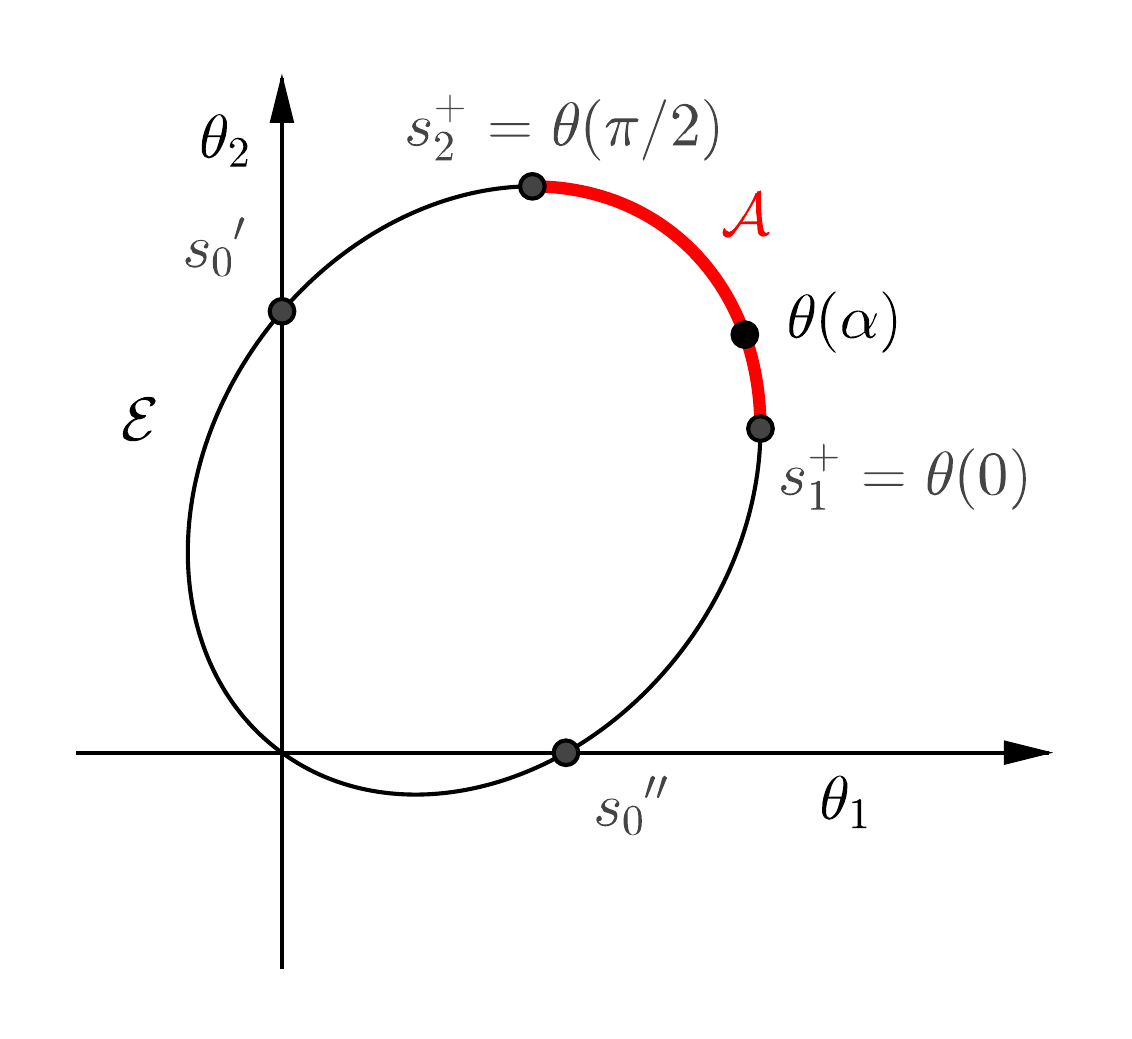}
    \caption{The arc $\mathcal{A}= [s_1^+,s_2+]$ if $\theta_1(s_2^+)<0$, $\theta_2(s_1^+)<0$ on the left picture, if $\theta_1(s_2^+)>0$, $\theta_2(s_1^+)>0$ on the right picture}
     \label{fipoles}
    \end{figure}
    
  The important conclusion  is that  in all cases, the pole $p$  of $\phi_2$  on the arc 
$\{ s'_0, \theta(\alpha)\}$   with the smallest  $\langle  \theta(p) \mid e_\alpha \rangle$
  is the  \underline{closest} to  $s'_0$. 
        In the same way we can consider the arc  $\{s''_0, \theta(\alpha)\}$  and find out, due to monotonicity of the function 
 $\langle  \theta(s) \mid e_\alpha \rangle$, 
 that the pole of $\phi_1$ 
   with the smallest  $\langle  \theta(p) \mid e_\alpha \rangle$
  is the  \underline{closest} to  $s''_0$.    
          We know from Lemmas \ref{lem3}--\ref{lem4} the way that these poles are related to zeros of $\gamma_1$ and $\gamma_2$.
    Now we summarize this information in the following theorem.
  
\begin{thm} 
\label{thmpoles}

\begin{itemize} 

\item[(a)]   Let $\zeta \theta^{**} \not \in \{\theta(\alpha), s'_0\}$,  $\eta \theta^* \not\in \{\theta(\alpha), s''_0)\}$. 
   Then ${\cal P'}_\alpha$ and ${\cal P''}_\alpha$ are both empty, 
       $\theta(\alpha)$ is not a pole of $\phi_1$ and neither of $\phi_2$.

 \item[(b)] Let $\zeta \theta^{**}  \in \}\theta(\alpha), s'_0\}$  and $\eta \theta^* \not\in \} \theta(\alpha), s''_0\}$.
   Then 
\begin{equation}
\label{ki}
 \min_{p \in {\cal P'}_\alpha \cup {\cal P''}_\alpha }  \langle \theta(p) \mid e_\alpha \rangle  = \langle \zeta \theta^{**} \mid e_\alpha \rangle 
\end{equation}
   and this minimum
  over  $ {\cal P'}_\alpha \cup {\cal P''}_\alpha   $  is achieved  at  the unique element $p=\zeta \theta^{**}$ which is a pole of the first order of $\phi_2$.

\item[(c)]  Let $\zeta \theta^{**}  \not\in \}\theta(\alpha), s'_0 \}$  and $\eta \theta^* \in \{s''_ 0, \theta(\alpha) \{$.
   Then 
\begin{equation}
\label{kii}
 \min_{ p \in {\cal P'}_\alpha \cup {\cal P''}_\alpha}  \langle \theta(p) \mid e_\alpha \rangle  = \langle \eta \theta^{*} \mid e_\alpha \rangle 
\end{equation}
     and this minimum 
  over  $  {\cal P'}_\alpha \cup {\cal P''}_\alpha    $ is achieved at the unique element $p=\eta \theta^{*}$ which is a pole of the first order of $\phi_1$.

\item[(d)]    Let $\zeta \theta^{**}  \in  \}\theta(\alpha), s'_0 \}$  and $\eta \theta^* \in  \{s''_0, \theta(\alpha) \{$.
       
 If  $\langle \zeta \theta^{**} \mid e_\alpha \rangle < \langle \eta \theta^{*} \mid e_\alpha \rangle$, then 
(\ref{ki}) is valid.   If  $\langle \zeta \theta^{*} \mid e_\alpha \rangle > \langle \eta \theta^{**} \mid e_\alpha \rangle$, then 
(\ref{kii}) is valid.  In both cases the minimum over  ${\cal P'}_\alpha \cup {\cal P''}_\alpha$ is achieved at the unique element 
  which is the pole of the first order  $p=\zeta \theta^{**}$  of $\phi_2$  or the  pole of the first order $p=\eta \theta^*$ of $\phi_1$  respectively. 

   If      $\langle \zeta \theta^{**} \mid e_\alpha \rangle = \langle \eta \theta^{*} \mid e_\alpha \rangle$, then 
 \begin{equation}
\label{kit}
 \min_{p \in {\cal P'}_\alpha \cup {\cal P''}_\alpha }  \langle \theta(p) \mid e_\alpha \rangle  = \langle \zeta \theta^{**} \mid e_\alpha \rangle=
    \langle \eta \theta^{*} \mid e_\alpha \rangle.
\end{equation}   
     This minimum over  ${\cal P'}_\alpha \cup {\cal P''}_\alpha $ is achieved 
   at exactly  two elements $p=\zeta \theta^{**}$
  and $p=\eta \theta^*$ which are poles of the first order of $\phi_1$ and $\phi_2$ respectively.

\end{itemize}

\end{thm}

\begin{proof} (a)
   Let $\theta_1(s_2^+)<0$ and let $\alpha >\alpha'$ defined above. Then $\theta_1(\alpha)<0$ and
       all points of the arc $\{\theta(\alpha), s'_0\{$ 
  have the first coordinate negative, so that function $\phi_2(\theta_1(s))$ is initially well defined at them and holomorphic.
     Let now $\theta_1(s_2^+)<0$ and $\alpha \in ]0, \alpha'[$ or $\theta_1(s_2^+)\geq 0$. Then $\theta_1(\alpha)>0$ and the arc 
  $\{\theta(\alpha), s'_0\}$  written in the anticlockwise direction is $[\theta(\alpha), s'_0]$. 
   Assume that  $\phi_2(\theta_1(s))$ has poles on   
$[\theta(\alpha), s'_0[$ and $\theta^p$ is the closest to $s'_0$.  Then by Lemma \ref{lem1} 
    either $\gamma_2(\zeta \theta^p)=0$  or parameters are such that $\theta_2(s_1^+)>0$, $\eta \zeta \theta^p \in ]\eta s_1^+, s_0[$ and 
   $\gamma_1(\eta\zeta \theta^p)=0$.   In the first case $\zeta \theta^p=\theta^{**}$  is a  zero of $\gamma_2$  different from $s_0$.
   This implies $\theta^p =\zeta \theta^{**} \in  [\theta(\alpha) s'_0[$ which is impossible by assumptions.
   In the second case $\eta \zeta \theta^p=\theta^*$  is a zero of $\gamma_1$ different from $s_0$. This implies 
    $\zeta \theta^p=\eta \theta^* \in \eta ]\eta s_1^+, s_0[=]s''_0, s_1^+[ \subset ]s''_0, \theta(\alpha)[ =\}\theta(\alpha), s''_0\{$ 
   that  contradicts the assumptions as well.
    Hence $\phi_2(\theta_1(s))$ has no poles on the open arc $\}\theta(\alpha), s'_0\{$ and neither at $\theta(\alpha)$, 
        ${\cal P'}_\alpha$ is empty, 
     The reasoning for ${\cal P''}_\alpha$ is the same.

(b) By stability conditions (\ref{u}) and (\ref{v})  $\theta^{**}_1>0$, then $\zeta\theta^{**}_1>0$. Thus  $\theta_1(\alpha)>0$, 
     in the case $\theta_1(s_2^+)<0$ the angle $\alpha$ 
  must be smaller than $\alpha'$ and the arc $\}\theta(\alpha), s'_0\{$ should be  written $]\theta(\alpha), s'_0[$. 
    By Lemma~\ref{lem3} there exist poles of  function $\phi_2(\theta_1(s))$ on this arc and $\zeta \theta^{**}$ is one among them.
   By Lemma~\ref{lem1} $\zeta \theta^{**}$ can not be the closest pole to $s'_0$  only if  the parameters are such that $\theta_2(s_1^+)>0$
   and for some $\theta^p \in ]\theta(\alpha), s'_0[$  such that  $\eta \zeta \theta^p \in ]\eta s_1^+, s_0[$  
   $\gamma_1(\eta\zeta \theta^p)=0$. But then $\eta \zeta \theta^p=\theta^*$  is  a zero of $\gamma_1$ different from $s_0$. It follows
    $\zeta \theta^p=\eta \theta^* \in \eta ]\eta s_1^+, s_0[=]s''_0, s_1^+[ \subset ]s''_0, \theta(\alpha)[= \}\theta(\alpha), s''_0\{$ 
    that is impossible by  assumptions.
   Hence by Lemma \ref{lem1} $\zeta \theta^{**}$ is the closest pole to $s'_0$ of $\phi_2(\theta_1(s))$ and it is of the first order.
    The function $\langle \theta(s) \mid  e_\alpha \rangle $ being decreasing on  $]\theta(\alpha), s'_0[$ when $s$ runs 
  the arc in the anticlockwise direction, 
 thus
\begin{equation}
\label{pp}
       \min_{p  \in {\cal P'}_\alpha  }  \langle \theta(p) \mid e_\alpha \rangle  = \langle \zeta \theta^{**} \mid e_\alpha \rangle,
\end{equation}
   and the minimum is achieved on the unique element $\zeta \theta^{**}$.

   If ${\cal P}''_\alpha$ is empty then the statement (b)  is proved. 

Assume that ${\cal P}''_\alpha$ is not empty.
     Then there exist poles of $\phi_1(\theta_2(s))$ on the arc  $\}\theta(\alpha), s''_0\{$. Since function $\phi_1(\theta_2(s))$ is initially well defined 
  and holomorphic at all points  with the second coordinate negative, then $\theta_2(\alpha)>0$ and the arc is 
     $]s''_0, \theta(\alpha)[$ when written in the anticlockwise direction. Let $\theta^p$ be a  pole of $\phi_1(\theta_2(s))$ which is the closest to 
      $s''_0$.     
Then by Lemma \ref{lem1} 
    either $\gamma_1(\eta \theta^p)=0$  or parameters are such that $\theta_1(s_2^+)>0$, $\zeta \eta \theta^p \in ]s_0, \zeta s_2^+[$ and 
   $\gamma_2(\zeta\eta \theta^p)=0$.   In the first case $\eta \theta^p=\theta^{*}$  is a  zero of $\gamma_1$  different from $s_0$.
   This implies $\theta^p =\eta \theta^{*} \in  ]s''_0, \theta(\alpha)[$ which is impossible by assumptions.
   In the second case $\zeta \eta \theta^p=\theta^{**}$ where 
    $\eta \theta^p=\zeta \theta^{**} \in \zeta ]s_0, \zeta s_2^+[=]s_2^+, s'_0[ \subset ]\theta(\alpha), s'_0[$.
   Thus   $\theta^p=\eta \zeta \theta^{**}$ is the closest pole to $s''_0$.
   Hence, the closest pole of the first order coincides with it or is further away from $s''_0$. 
         Since the function $<\theta(s) \mid e_\alpha>$ is increasing
   on $]s''_0, \theta(\alpha)[$ when $s$ is running from $s''_0$ to $\theta(\alpha)$, 
     we derive
       $$\min_{p \in {\cal P''}_\alpha  } \langle \theta(p) \mid e_\alpha \rangle \geq  \langle \eta \zeta \theta^{**} \mid e_\alpha \rangle .$$
  But by Lemma \ref{lem2}
  $$ \theta_1(\eta \zeta \theta^{**})> \theta_1(\zeta \theta^{**}),  \  \  \theta_2(\eta \zeta \theta^{**})= \theta_2(\zeta \theta^{**})$$
  from where 
   $$ \langle \eta \zeta \theta^{**} \mid e_\alpha \rangle   >   \langle \zeta \theta^{**} \mid e_\alpha \rangle .$$
        Thus, whenever ${\cal P''}_\alpha$ is non empty, 
 $$\min_{p \in {\cal P''}_\alpha  }  \langle \theta(p) \mid e_\alpha \rangle  >    \langle  \zeta \theta^{**} \mid e_\alpha \rangle .$$
    This inequality combined with  (\ref{pp}) finishes the proof of (b).

  The proof of (c) is symmetric.

 (d)  Since $\theta^{*}_2=\eta \theta^{*}_2>0$ and $\theta^{**}_1=\zeta \theta^{**}_1>0$ by stability conditions (\ref{u}) and  (\ref{v}), 
    then $\theta(\alpha)$ has both coordinates positive.
    The corresponding arcs written in the anticlockwise direction are $]\theta(\alpha), s'_0[\subset ]s_1^+ , s'_0[$ and 
   $]s''_0, \theta(\alpha)[ \subset ]s''_0, s_2^+[$. By Lemma \ref{lem1} 
   $\zeta \theta^{**}$ is a pole of $\phi_2(\theta_1(s))$  on the first of these arcs  while $\eta \theta^*$ is a pole of $\phi_1(\theta_2(s))$ on the second one.
     Then one of the statements of  Lemma  \ref{lem4}  (i), (ii) or (iii) holds true.

   Under the statement (i), taking into account the monotonicity of the function $\langle \theta(s) \mid e_\alpha \rangle$ on the arcs, 
    we derive immediately that   
      $ \min_{p \in {\cal P'}_\alpha  }  \langle \theta(p) \mid e_\alpha \rangle  = \langle \zeta \theta^{**} \mid e_\alpha \rangle$, 
  and this minimum is achieved on the unique element $p=\zeta \theta^{**}$ . We derive also  that 
    $\min_{p \in {\cal P''}_\alpha  }  \langle \theta(p) \mid e_\alpha \rangle  = \langle \eta \theta^{*} \mid e_\alpha \rangle $
   and  this minimum is achieved on the unique element $p=\eta \theta^*$.  
     Thus, under the statement (i) of Lemma \ref{lem4}, the theorem is immediate.    

   Assume now (ii) of Lemma \ref{lem4}.   Again by monotonicity of  $\langle \theta(s) \mid e_\alpha \rangle$   we deduce 
        $ \min_{p \in {\cal P'}_\alpha  }  \langle \theta(p) \mid e_\alpha \rangle  = \langle \zeta \theta^{**} \mid e_\alpha \rangle$
  where the minimum is achieved at the unique element $\zeta \theta^{**}$.
       Under (ii)  all poles of $\phi_1(\theta_2(s))$ on  $]s''_0, \theta(\alpha)[$
    are not closer to $s''_0$ than $ \eta\zeta \theta^{**}$ , so that 
   either ${\cal P''}_\alpha$ is empty or       
  $$\min_{p \in {\cal P''}_\alpha  }  \langle \theta(p) \mid e_\alpha \rangle  \geq  \langle \eta\zeta  \theta^{**} \mid e_\alpha \rangle.$$
    By Lemma \ref{lem2}
  $\theta_1( \eta\zeta  \theta^{**})>\theta_1(\zeta \theta^{**}),$  $  \theta_2 ( \eta\zeta  \theta^{**})= \theta_2(\zeta \theta^{**})$ 
   from where   $\langle \eta\zeta  \theta^{**} \mid e_\alpha \rangle  >   \langle \zeta  \theta^{**} \mid e_\alpha \rangle $.
    Hence 
    $$\min_{p \in {\cal P''}_\alpha  }  \langle \theta(p) \mid e_\alpha \rangle  >  \langle  \zeta  \theta^{**} \mid e_\alpha \rangle,$$
  and finally
\begin{equation}
\label{ll} 
    \min_{p \in {\cal P'}_\alpha \cup {\cal P''}_\alpha  }  \langle   \theta(p) \mid e_\alpha \rangle  =
 \langle  \zeta  \theta^{**} \mid e_\alpha \rangle
\end{equation}
   where the minimum is achieved on the unique element $\zeta \theta^{**}$.
         From the other hand, the pole $\eta \theta^*  \in ]s''_0, \theta(\alpha)[$ of $\phi_1(\theta_2(s))$ in this case is not closer  to $s''_0$ 
    than  $ \eta\zeta \theta^{**}$.  Then the inequality 
\begin{equation}
\label{mm} 
   \langle \eta \theta^* \mid e_\alpha \rangle \geq \langle \eta \zeta \theta^{**} \mid e_\alpha \rangle > 
    \langle \zeta \theta^{**}  \mid e_\alpha \rangle 
\end{equation}
 is valid. 
     
Under the statement (iii) of  Lemma \ref{lem4}, by symmetric arguments, 
$\min_{p \in {\cal P'}_\alpha \cup {\cal P''}_\alpha  }  \langle   \theta(p) \mid e_\alpha \rangle  =
 \langle  \eta  \theta^{*} \mid e_\alpha \rangle$ where the minimum is achieved on the unique element $\eta \theta^{*}$, while 
   $\langle \eta \theta^* \mid e_\alpha \rangle<  \langle \zeta \theta^{**}  \mid e_\alpha \rangle$.
   The concludes the proof of the lemma. 

\end{proof}

\section{Asymptotic expansion of the density \texorpdfstring{$\pi(r\cos (\alpha), r \sin (\alpha))$}{pi r alpha}, \texorpdfstring{$r \to \infty$, $\alpha \in \mathcal{O}(\alpha_0)$}{}}
\label{sec:asymptexpansion}

\subsection{Given angle \texorpdfstring{$\alpha_0$}{alpha0}, asymptotic expansion  of the density  as a function of  parameters \texorpdfstring{$(\Sigma, \mu, R)$}{Sigma mu R}}
\label{subsec:givenangle}

We are now ready to formulate and prove the results. In this section we fix an angle $\alpha_0 \in ]0, \pi/2[$ and give the asymptotic 
  expansion of the density of stationary distribution depending on parameters  $(\Sigma, \mu, R)$, and more precisely on the position of zeros 
  of $\gamma_1$ and $\gamma_2$ on ellipse $\mathcal{E}$.
 
   In the first theorem parameters  $(\Sigma, \mu, R)$ are such that  the asymptotic expansion is determined by the saddle-point.

\begin{thm}
\label{thmresults1}
     Let $\alpha_0 \in ]0, \pi/2[$,  $\mathcal{O}(\alpha_0)$ is a small enough neighborhood of $\alpha_0$.
 Assume that  $\zeta \theta^{**} \not \in \{\theta(\alpha_0), s'_0\}$,  $\eta \theta^* \not\in \{\theta(\alpha_0), s''_0)\}$. 
   Then there exist constants $c^l(\alpha)$,  $l=0,1,2,\ldots $, such that for any $k>0$:
\begin{equation}
\label{as1}
\pi(r \cos (\alpha), r \sin (\alpha) ) \sim \sum_{l=0}^k \frac{c^l(\alpha)}{ r^l \sqrt{r}} e^{-r 
 \langle \theta (\alpha) \mid e_\alpha \rangle },\  \hbox{ as }r \to \infty, \  \hbox{ uniformly for }\alpha \in \mathcal{O}(\alpha_0). 
\end{equation}
   Constants  $c^l(\alpha)$  $l=0,1,2,\ldots $ depend continuously on $\alpha$ and can be expressed
 in terms of functions $\phi_1$ and $\phi_2$ and their derivatives at $\theta(\alpha)$. 
  Namely 
\begin{equation}
c^0(\alpha)= c_{\theta_1}^0(\alpha) + c_{\theta_2}^0(\alpha)
\end{equation}
where $c_{\theta_1}^0(\alpha)$ and $c_{\theta_2}^0(\alpha)$
are defined in Lemma~\ref{saddleas}.
\end{thm} 

\begin{proof}
      By Lemma \ref{sp} (iii) $\theta(\alpha)$ depends continuously on $\alpha$, then 
    $\zeta \theta^{**} \not \in \{\theta(\alpha), s'_0\}$,  $\eta \theta^* \not\in \{\theta(\alpha), s''_0\}$
  for all  $\alpha \in \mathcal{O}(\alpha_0)$.  By Theorem \ref{thmpoles} (a) 
  the sets  ${\cal P'}_\alpha$ and ${\cal P''}_\alpha$ are both empty, furthermore,
       $\theta(\alpha)$ is not a pole of $\phi_1$ and neither of $\phi_2$. 
  Then by Lemma \ref{ppmm} the density equals the sum of integrals along shifted 
  contours $\Gamma_{\theta_1, \alpha}$ and $\Gamma_{\theta_2, \alpha}$  the asymptotics of which is found in Lemma \ref{saddleas}, 
  $c^l(\alpha)=c^l_{\theta_1}(\alpha)+ c^l_{\theta_2}(\alpha)$, $l=0,1,2,\ldots$.
\end{proof}

In the second theorem parameters   $(\Sigma, \mu, R)$ are such that
 the  most important terms of the asymptotic expansion come from the poles of $\phi_1$ or $\phi_2$ and the smaller ones
  come  from the saddle-point.

\begin{thm}
\label{thmresults2}
   Let $\alpha_0 \in ]0, \pi/2[$,  $\mathcal{O}(\alpha_0)$ is a small enough neighborhood of $\alpha_0$. 
Assume that  $\zeta \theta^{**}  \in \}\theta(\alpha_0), s'_0\}$ or  $\eta \theta^* \in \}\theta(\alpha_0), s''_0\}$.
Assume also that  $\theta(\alpha_0)$ is \underline{not}  a pole of $\phi_1(\theta_2(s))$ neither of $\phi_2(\theta_1(s))$.
 Then for any $k>0$ when $r \to \infty$, uniformly for $\alpha \in \mathcal{O}(\alpha_0)$ we have
\begin{eqnarray}
\pi(r\cos(\alpha),  r \sin(\alpha)) & \sim &
 \sum_{p \in {\cal P'}_{\alpha_0}} {\rm res}_p \phi_2(\theta_1(s))
\frac{\gamma_2(p)}{\sqrt{d(\theta_1(p))}} e^{-r \langle \theta(p) \mid e_\alpha \rangle } \nonumber \\
&& { }+
   \sum_{p \in {\cal P''}_{\alpha_0} } {\rm res}_p \phi_1(\theta_2(s))
\frac{\gamma_1(p)}{\sqrt{\tilde d(\theta_2(p))}} e^{-r \langle \theta(p) \mid e_\alpha \rangle }  \nonumber  \\
&& {} +
\sum_{l=0}^k \frac{c^l(\alpha)}{ r^l \sqrt{r}} e^{-r 
 \langle \theta (\alpha) \mid e_\alpha \rangle }. \label{th23} 
\end{eqnarray}
     Constants  $c^l(\alpha)$  $l=0,1,2,\ldots$ are the same as  in Theorem \ref{thmresults1}.
  Furthermore
\begin{itemize}
\item[(i)]  If  $\zeta \theta^{**}  \in \}\theta(\alpha_0), s'_0\}$ and  $\eta \theta^* \not\in \}\theta(\alpha_0), s''_0\}$, 
   then the main  term in the expansion  (\ref{th23}) is at $p=\zeta \theta^{**}$.
\item[(ii)]  If   If  $\zeta \theta^{**}  \not\in \}\theta(\alpha_0), s'_0\}$ and  $\eta \theta^* \in \}\theta(\alpha_0), s''_0\}$, 
   then the main term in (\ref{th23}) is at $p=\eta \theta^{*}$.
\item[(iii)]    Let  $\zeta \theta^{**}  \in \}\theta(\alpha_0), s'_0\}$ and  $\eta \theta^* \in \}\theta(\alpha_0), s''_0\}$.
   If   $\langle \zeta \theta^{**} \mid e_{\alpha_0 }\rangle  <   \langle \eta \theta^{*} \mid e_{\alpha_0 }\rangle $,       
   then the main term in (\ref{th23}) is at $p=\zeta \theta^{**}$.
 
   If  $\langle \zeta \theta^{**} \mid e_{\alpha_0 }\rangle  > \langle \eta \theta^{*} \mid e_{\alpha_0 }\rangle $,  
    then main term in (\ref{th23}) is  at $p=\eta \theta^{*}$.
               
     If  $\langle \zeta \theta^{**} \mid e_{\alpha_0 }\rangle  = \langle \eta \theta^{*} \mid e_{\alpha_0 }\rangle $, then 
  two the most important terms in  the expansion (\ref{th23}) are at $p=\zeta \theta^{**}$ and at $p=\eta \theta^*$.

\end{itemize}

\end{thm} 

\begin{proof}
   Point $\theta(\alpha_0)$ being not a pole of  $\phi_1$ neither of $\phi_2$, one can choose  $\mathcal{O}(\alpha_0)$ small enough 
  such that $\theta(\alpha)$ is not a pole of no one of these functions
 and    $\mathcal{P}'_{\alpha} \cup \mathcal{P}''_\alpha    =  \mathcal{P}'_{\alpha_0} \cup \mathcal{P}''_{\alpha_0} $   
for all $ \alpha \in \mathcal{O}(\alpha_0)$.  By assumptions 
  $\zeta \theta^{**}  \in \}\theta(\alpha_0), s'_0\}$ or  $\eta \theta^* \in \}\theta(\alpha_0), s''_0\}$, then 
  by Theorem \ref{thmpoles} (b), (c) or (d) 
   $\mathcal{P}'_{\alpha_0} \cup \mathcal{P}''_{\alpha_0} $ is not empty.
      Finally  by virtue of Lemma \ref{ppmm}  and Lemma \ref{saddleas} the representation (\ref{th23}) holds true. 
      
 Let us study the main asymptotic term. 
     Statements (i), (ii) and (iii) for $\alpha=\alpha_0$ follow directly from 
  Theorem \ref{thmpoles} (b), (c) and  (d). They remain valid for any $\alpha \in \mathcal{O}(\alpha_0)$ 
due to  the continuity of the functions $\alpha \to \langle \theta(p) \mid e_\alpha \rangle $  for any $p \in  \mathcal{P}'_{\alpha_0} 
\cup \mathcal{P}''_{\alpha_0} $. 
 \end{proof}

\noindent{\bf Remark.}
     Under parameters such that 
 $\zeta \theta^{**}  \in \}\theta(\alpha_0), s'_0\}$,  $\eta \theta^* \in \}\theta(\alpha_0), s''_0\}$ and 
$\langle \zeta \theta^{**} \mid e_{\alpha_0 }\rangle  = \langle \eta \theta^{*} \mid e_{\alpha_0 }\rangle $  (case (iii)), 
  for any fixed angle $\alpha<\alpha_0$, the main asymptotic term is at $\eta \theta^{*}$ and the second one is at $\zeta \theta^{**}$; 
  for any fixed  angle $\alpha>\alpha_0$  the pole $\zeta\theta^{**}$ provides the main asymptotic term and $\eta \theta^*$ gives the second one.
   If $r\to \infty$ and $\alpha \to \alpha_0$, both of these terms should be taken into account.   

\medskip 

In Theorem \ref{thmresults2} $\theta(\alpha_0)$ is assumed not to be a pole of $\phi_1$ and neither of $\phi_2$, that is why Lemma \ref{ppmm} applies.
   Nevertheless, it may happen (for a very few angles and under some sets of parameters)
   that $\theta(\alpha_0)$ is a pole of one of these functions. In this case the following theorem holds true.

\begin{thm}
\label{thmresults3}
  Let $\alpha_0 \in ]0, \pi/2[$. 
Assume that  $\zeta \theta^{**}  \in \}\theta(\alpha_0), s'_0\}$ or  $\eta \theta^* \in \}\theta(\alpha_0), s''_0)\}$.

Assume also that  $\theta(\alpha_0)$ is  a pole of $\phi_1(\theta_2(s))$ or of $\phi_2(\theta_1(s))$.

 Then for any $\delta>0$ there exists a small enough neighborhood $\mathcal{O}(\alpha_0)$ such that 
\begin{eqnarray}
\pi(r\cos(\alpha),  r \sin(\alpha)) & \sim &
 \sum_{p \in {\cal P'}_{\alpha_0}} {\rm res}_p \phi_2(\theta_1(s))
\frac{\gamma_2(p)}{\sqrt{d(\theta_1(p))}} e^{-r \langle \theta(p) \mid e_\alpha \rangle } \nonumber \\
&& { }+
   \sum_{p \in {\cal P''}_{\alpha_0} } {\rm res}_p \phi_1(\theta_2(s))
\frac{\gamma_1(p)}{\sqrt{\tilde d(\theta_2(p))}} e^{-r \langle \theta(p) \mid e_\alpha \rangle }  \nonumber  \\
&&{}  + o(e^{-r (\langle \theta(\alpha) \mid e_\alpha \rangle -\delta ) } ) \  \  \  r \to \infty, \hbox{ uniformly } \forall \alpha \in \mathcal{O}(\alpha_0) \label{th233}   
\end{eqnarray}

   Furthermore, the main term in this expansion is the same as in Theorem \ref{thmresults2}, cases (i), (ii) and (iii). 
\end{thm}

\begin{proof}
   For any $\delta>0$ one can choose $\tau' \in \}s'_0, \theta(\alpha_0)\{$ and $\tau'' \in \}s''_0, \theta(\alpha_0)\{$
 close enough to $\theta(\alpha_0)$ so that $\mathcal{P}'_{\alpha_0} \subset \}s'_0, \tau'\{$ and 
   $\mathcal{P}''_{\alpha_0} \subset \}s''_0, \tau"\{$. 
   Furthermore $\tau'$  and $\tau''$ can be chosen close enough to $\alpha_0$ 
  so that $ \langle \theta(\alpha_0) \mid  e_{\alpha_0} \rangle -  \langle \tau' \mid e_{\alpha_0} \rangle <\delta/4$ 
 and    $ \langle \theta(\alpha_0) \mid  e_{\alpha_0} \rangle -  \langle \tau'' \mid e_{\alpha_0} \rangle <\delta/4$. Then by continuity 
  of the functions $\alpha \to \langle \theta(\alpha) \mid e_\alpha \rangle $, $\alpha \to  \langle \tau '\mid e_\alpha \rangle $, 
  $\alpha \to  \langle \tau'' \mid e_\alpha \rangle $ one can fix a small enough neighborhood  $\mathcal{O}(\alpha_0)$ 
such that 
\begin{equation}
\label{at}
\langle \theta(\alpha) \mid  e_{\alpha} \rangle -  \langle \tau' \mid e_{\alpha} \rangle <\delta/2,\  \  \  
  \langle \theta(\alpha) \mid  e_{\alpha} \rangle -  \langle \tau'' \mid e_{\alpha} \rangle <\delta/2 ,\  \  \forall \alpha \in \mathcal{O}(\alpha_0). 
\end{equation}
   Next, we shift the integration contours in (\ref{iwi})   $\mathcal{I}_{\theta_1}^+$ 
  and   $\mathcal{I}_{\theta_2}^+$ to the new  ones 
  $\Gamma'_{\theta_1, \alpha}$ and $\Gamma''_{\theta_2, \alpha}$ going through $\tau'$ and $\tau''$ respectively that 
  we construct as follows:
 $\Gamma'_{\theta_1, \alpha} =\Gamma'^{1}_{\theta_1, \alpha} \cup \Gamma'^{2, \pm }_{\theta_1, \alpha}\cup \Gamma'^{3, \pm}_{\theta_1, \alpha}$
  where 
  $\Gamma'^{1}_{\theta_1}=\{ s :  \Re \theta_1(s)= \Re \theta_1(\tau'),  -V(\alpha)\leq  \Im \theta_1(s)  \leq V(\alpha)  \}$, 
   $\Gamma'^{2,\pm }_{\theta_1, \alpha}= \{s :  \Im \theta_1(s)= \pm V(\alpha),        0\leq  \Re \theta_1(s) \leq   \Re \theta_1(\tau')\}
   $  if  $\Re \theta_1(\tau')>0$ and 
   $\Gamma'^{2, \pm }_{\theta_1, \alpha}= \{s :  \Im \theta_1(s)= \pm V(\alpha),        0\geq  \Re \theta_1(s) \geq  \Re \theta_1(\tau')\}$  if $\Re \theta_1(\tau')<0$ , 
   finally $\Gamma'^{3, +}_{\theta_1, \alpha}= \{s : \Re \theta_1(s)=0, \Im \theta_1(s) \geq V(\alpha)\}$, 
      $\Gamma'^{3, -}_{\theta_1, \alpha}= \{s : \Re \theta_1(s)=0, \Im \theta_1(s) \leq -V(\alpha)\}$.
   The construction of $\Gamma''_{\theta_2, \alpha}$ is analogous.
   The value $V(\alpha)$ is fixed as: 
\begin{equation}
\label{va}
V(\alpha)=\max\Big(V, \frac{\langle \tau' \mid e_\alpha \rangle -d_1}{d_2 \sin (\alpha)},   \frac{\langle \tau'' \mid e_\alpha \rangle -d_1}{d_2 \sin (\alpha)}
\Big)
\end{equation}
  with notations from Lemma \ref{uv}.
   Thanks to the  representation (\ref{iwi})   and  Cauchy theorem 
\begin{equation}
\label{lre}
\pi(re_\alpha) =
 \sum_{p \in {\cal P'}_{\alpha_0}} {\rm res}_p \phi_2(\theta_1(s))
\frac{\gamma_2(p)}{\sqrt{d(\theta_1(p))}} e^{-r \langle \theta(p) \mid e_\alpha \rangle } +
   \sum_{p \in {\cal P''}_{\alpha_0}} {\rm res}_p \phi_1(\theta_2(s))
\frac{\gamma_1(p)}{\sqrt{\tilde d(\theta_2(p))}} e^{-r \langle \theta(p) \mid e_\alpha \rangle }$$  
$${}+ \frac{1}{2\pi \sqrt{\det \Sigma}} \int_{\Gamma'_{\theta_1, \alpha} }
\frac{\varphi_2(s) \gamma_2(\theta(s))}{s}  e^{- r \langle \theta(s) \mid e_\alpha \rangle }
{\mathrm{d}s}
+
\frac{1}{2\pi \sqrt{\det \Sigma}} \int_{\Gamma''_{\theta_2, \alpha} }
\frac{\varphi_1(s) \gamma_1(\theta(s))}{s} e^{- r \langle \theta(s) \mid e_\alpha \rangle }
\mathrm{d}s. 
\end{equation}
     Applying Lemma \ref{uv} (i) for the estimation of integrals along $\Gamma'^{1}_{\theta_1,\alpha}$ and 
$\Gamma''^{1}_{\theta_1,\alpha}$, and the same lemma (ii) for the estimation of  those along
  $\Gamma'^{\pm 2 }_{\theta_1,\alpha}$  
$\Gamma''^{\pm 2}_{\theta_2,\alpha}$,  $\Gamma'^{\pm 3}_{\theta_1,\alpha}$ and 
$\Gamma''^{\pm 3}_{\theta_2,\alpha}$  exactly as in Lemma \ref{saddleas} and  in view of (\ref{va})  
   we can show that  with some constant $C>0$  
 \begin{eqnarray}
\label{in11}
\Big|  \int_{\Gamma'_{\theta_1, \alpha} }
\frac{\varphi_2(s) \gamma_2(\theta(s))}{s}  e^{- r \langle \theta(s) \mid e_\alpha \rangle }
{\mathrm{d}s}\Big| & \leq  & C e ^{- r \langle \tau' \mid e_\alpha \rangle},\nonumber \\
\Big|  \int_{\Gamma''_{\theta_1, \alpha} }
\frac{\varphi_2(s) \gamma_2(\theta(s))}{s}  e^{- r \langle \theta(s) \mid e_\alpha \rangle }
{\mathrm{d}s}\Big|  &\leq & C e ^{- r \langle \tau'' \mid e_\alpha \rangle} \  \  \forall r>0,  \forall \alpha \in \mathcal{O}(\alpha_0).
\nonumber 
\end{eqnarray}
   Hence,  by  (\ref{at}) 
\begin{equation}
 \int_{\Gamma'_{\theta_1, \alpha} }
\frac{\varphi_2(s) \gamma_2(\theta(s))}{s}  e^{- r \langle \theta(s) \mid e_\alpha \rangle }
{\mathrm{d}s}
+
 \int_{\Gamma''_{\theta_2, \alpha} }
\frac{\varphi_1(s) \gamma_1(\theta(s))}{s} e^{- r \langle \theta(s) \mid e_\alpha \rangle }
\mathrm{d}s = o(e^{- r(\langle \theta(\alpha) \mid e_\alpha \rangle -\delta ) })    
\end{equation} 
as $r \to \infty$  uniformly  $\forall \alpha \in \mathcal{O}(\alpha_0)$. 
    This finishes the proof of the representation (\ref{th233}).
  The analysis of the main term is the same as in Theorem \ref{thmresults2}. 
\end{proof}

       It remains to study the cases of parameters such that 
\begin{itemize}
\item[(O1)]   $\zeta \theta^{**}=\theta(\alpha_0) $ and $\eta \theta^{*} \not\in \{s''_0, \theta(\alpha_0)\{$ 

\item[(O2)]   $\eta \theta^*=\theta(\alpha_0)$ and $\zeta \theta^{**} \not\in \{s''_0, \theta(\alpha_0)\{$. 
\end{itemize}

By Lemma \ref{lem3}  this means that $\theta(\alpha_0)$ is a pole of one of  functions $\phi_1$ or $\phi_2$.
  Since in both cases $ \eta \theta^{*} \not\in \{s''_0, \theta(\alpha)\{$,  $\zeta \theta^{**} \not\in \{s''_0, \theta(\alpha)\{$, 
  we derive by the same reasoning as in Theorem \ref{thmpoles} (a) that $\mathcal{P}'_{\alpha_0} \cup \mathcal{P}''_{\alpha_0}$ 
   is empty.   The following theorem is valid.
 
\begin{thm}
\label{thmresults4}
  Assume that $\alpha_0$ is such that  the assumptions on parameters (O1) or (O2) are valid. 
 Then for any $\delta>0$ there exists a small enough neighborhood $\mathcal{O}(\alpha_0)$ such that 
\begin{equation}
\label{th2334}
\pi(r\cos(\alpha), r \sin(\alpha)) = o(e^{-r (\langle \theta(\alpha) \mid e_\alpha \rangle -\delta ) } ) \  \  \  r \to \infty, \hbox{ uniformly } \forall \alpha \in \mathcal{O}(\alpha_0).   
\end{equation}
\end{thm}  

\begin{proof}
  We choose $\tau'$ and $\tau''$ according to (\ref{at}) and proceed exactly as in the proof of Theorem~\ref{thmresults3}. 
\end{proof}

\noindent{\bf Remark.} In Theorems \ref{thmresults3} and \ref{thmresults4} 
$\theta(\alpha_0)$ is a pole of one of the functions $\phi_1$ or $\phi_2$, hence 
at least  one  of the integrals (\ref{iwi}) can not be shifted to $\Gamma_{\theta_1, \alpha_0}$ 
or  $\Gamma_{\theta_2, \alpha_0}$ going through $\theta(\alpha_0)$.
        Furthermore, although for any $\alpha \in \mathcal{O}(\alpha_0)$, $\alpha\ne \alpha_0$, 
 this shift is possible, the uniform asymptotic expansion by the saddle-point method as
 in  Lemma \ref{saddleas}  does not stay valid, that is why we are not able to  specify  small 
 asymptotic terms in Theorem \ref{thmresults3} neither to obtain a more precise result in Theorem \ref{thmresults4}.
   This should be possible if we consider 
  the double asymptotics $r\to \infty$ and $\alpha \to \alpha_0$  and apply the (more advanced)  saddle-point method 
  in the special case when the saddle-point is approaching a pole of the integrand. We do not do it in the present paper.

\medskip 

\noindent{\bf Remark.} Assumptions  of theorems \ref{thmresults1}  ---  \ref{thmresults4} are expressed in terms of positions on ellipse $\mathcal{E}$
 of points $\zeta \theta^{**}$ and $\eta \theta^*$ 
     that are images of zeros of $\gamma_1$ and $\gamma_2$ on $\mathcal{E}$
  by Galois automorphisms. 
     They can be also expressed in terms of the following simple inequalities.

Under parameters such that $\theta_1(\alpha_0)>0$,  we have 
 $\zeta\theta^{**} \ne \{s'_0, \theta(\alpha_0)\}$ iff $\theta^{**} \ne \{s_0, \zeta \theta(\alpha_0)\}$ that is equivalent to  
$\gamma_2(\zeta \theta(\alpha))<0$.
 Under parameters such that  $\theta_1(\alpha_0)\leq 0$, we have always $\zeta \theta^{**} \ne \{s'_0, \theta(\alpha_0)\}$  because $\theta_1(\zeta \theta^{**})>0$ by stability conditions,
    in this case we have also  $\gamma_2(\zeta \theta(\alpha_0))\geq 0$.
 We come to the following conclusions.  

\begin{itemize}

 \item[(i) ]  Assumption  $\zeta\theta^{**} \ne \{s'_0, \theta(\alpha_0)\}$   is equivalent to 
  the one that $\gamma_2(\zeta \theta(\alpha_0))<0$ or  $\theta_1(\alpha_0)\leq 0$.

    Assumption  $\zeta\theta^{**} \in \{s'_0, \theta(\alpha_0)\{$
  is equivalent to the one  that  $\gamma_2(\zeta \theta(\alpha_0))>0$ and  $\theta_1(\alpha_0)> 0$.

\item[(ii)]  Assumption  $\eta\theta^{*} \ne \{s''_0, \theta(\alpha_0)\}$   is equivalent to  the one that $\gamma_1(\eta \theta(\alpha_0))<0$ or  $\theta_2(\alpha_0)\leq 0$.

    Assumption  $\eta\theta^{**} \in \{s''_0, \theta(\alpha_0)\{$
  is equivalent to the one that  $\gamma_1(\eta \theta(\alpha_0))>0$ and  $\theta_2(\alpha_0)> 0$.

\end{itemize} 

\subsection{Given parameters \texorpdfstring{$(\Sigma, \mu, R)$}{Sigma mu R},  density asymptotics for all angles \texorpdfstring{$\alpha_0 \in ]0, \pi/2[$}{alpha0} } 
\label{subsec:givenparameters}

       In this section we state the asymptotics of the density for all angles $\alpha_0 \in ]0, \pi/2[$ once  
  parameters $(\Sigma, \mu, R)$ are fixed.
   Theorems \ref{thmresultsnew1} -- \ref{thmresultsnew3} below are direct corollaries of 
  Theorems \ref{thmresults1} -- \ref{thmresults4} and elementary geometric properties of ellipse $\mathcal{E}$
  and straight lines $\gamma_1(\theta)=0$ and $\gamma_2(\theta)=0$, therefore we do not give their proofs.  
  To shorten the presentation, 
 we restrict ourselves to the main term in the formulations of the results, although of course further terms 
 of the expansions could be written. The different cases of Theorem \ref{thmresultsnew1} are illustrated by Figures \ref{firstfig}--\ref{lastfig}.

\begin{thm}  Let $\theta_1(s_2^+)>0$, $\theta_2(s_1^+)>0$.
\label{thmresultsnew1}

\begin{itemize}

\item[(i)] Let $\gamma_2(s_1^+)\leq 0$ and $\gamma_1(s_2^+)\leq 0$
   Then for any $\alpha_0 \in ]0, \pi/2[$ we have :

\begin{equation}
\label{assadlep}
\pi(r \cos \alpha, r \sin \alpha) \sim \frac{c(\alpha_0)}{\sqrt{r}} \exp(-r\langle \theta(\alpha) \mid e_\alpha \rangle),  \  \ r \to \infty, 
\alpha \to \alpha_0,
\end{equation} 
  where the constant $c(\alpha_0)$ depends continuously on $\alpha_0 \in ]0, \pi/2[$ and 
$\lim_{\alpha_0 \to 0} c(\alpha_0)=\lim_{\alpha_0 \to \pi/2} c(\alpha_0)=0$. 

\item[(ii)] Let $\gamma_2(s_1^+)>0$ and $\gamma_1(s_2^+)\leq 0$.

\begin{itemize}

\item[(iia)]  Let $\gamma_2(\zeta s_2^+)\geq 0$ or equivalently $\frac{ d \Theta_2^{+}(\theta_1) }{ d \theta_1 } \Bigm|_{\theta_1^{**} }\geq 0$.
   Then for any $\alpha_0 \in ]0, \pi/2[$ we have
\begin{equation}
\label{p**}
\pi(r \cos \alpha, r \sin \alpha) \sim d_1 \exp(-r\langle \zeta \theta^{**} \mid e_\alpha \rangle),  \  \ r \to \infty, 
\alpha \to \alpha_0,
\end{equation} 
  with some constant $d_1>0$.

\item[(iib)]  Let  $\gamma_2(\zeta s_2^+)<0$ or equivalently $A^{**}\equiv  \frac{ d \Theta_2^{+}(\theta_1) }{ d \theta_1 } \Bigm|_{\theta_1^{**} }
 < 0$.
   Define $\alpha_1= \arctan ( -1/ A^{**}) \in ]0, \pi/2[$.
   Then for any $\alpha_0 \in ]0, \alpha_1[$ we have (\ref{p**}) and for any $\alpha \in ]\alpha_1, \pi/2[$ we have 
   (\ref{assadlep}).

\end{itemize}

\item[(iii)]  Let $\gamma_2(s_1^+)<0$ and $\gamma_1(s_2^+)\geq 0$.

\begin{itemize}

\item[(iiia)]  Let $\gamma_1(\eta s_1^+)\geq 0$ or equivalently $\frac{ d \Theta_1^{+}(\theta_2) }{ d \theta_2 } \Bigm|_{\theta_2^{*} }  \geq 0$.
   Then for any $\alpha_0 \in ]0, \pi/2[$ we have 
\begin{equation}
\label{p*}
\pi(r \cos \alpha, r \sin \alpha) \sim d_2 \exp(- r\langle \eta \theta^{*} \mid e_\alpha \rangle),  \  \ r \to \infty, 
\alpha \to \alpha_0,
\end{equation} 
  with some constant $d_2>0$.
 
\item[(iiib)]  Let  $\gamma_1(\eta s_1^+)<0$ or equivalently $ A^*\equiv \frac{ d \Theta_1^{+}(\theta_2) }{ d \theta_2 } \Bigm|_{\theta_2^{*} }  < 0$.
   Define $\alpha_2= \arctan (- A^*) \in ]0, \pi/2[$.
   Then for any $\alpha_0 \in ]0, \alpha_2[$ we have (\ref{assadlep}) and for any $\alpha \in ]\alpha_2, \pi/2[$ we have 
   (\ref{p*}).

\end{itemize}

\item[(iv)]   Let  $\gamma_2(s_1^+)>0$ and $\gamma_1(s_2^+)>0$.

\begin{itemize}

\item[(iva)]  Let $\theta_1(\zeta \theta^{**})\leq \theta_1(\eta \theta^*)$  and  $\theta_2(\zeta \theta^{**})\leq \theta_2(\eta \theta^*)$
  where at least one of inequalities is strict.
  Then for any  $\alpha_0 \in ]0, \pi/2[$ we have (\ref{p**}). 

\item[(ivb)]    Let $\theta_1(\zeta \theta^{**})\geq \theta_1(\eta \theta^*)$  and  $\theta_2(\zeta \theta^{**})\geq \theta_2(\eta \theta^*)$
  where at least one of inequalities is strict.
  Then for any  $\alpha \in ]0, \pi/2[$ we have (\ref{p*}). 

\item[(ivc)]   Let $\theta_1(\zeta \theta^{**})\leq  \theta_1(\eta \theta^*)$  and  $\theta_2(\zeta \theta^{**}) \geq \theta_2(\eta \theta^*)$
  where at least one of the inequalities is strict.
  Let us define $\beta_0 =\arctan \frac{ \theta_1(\zeta\theta^{**})-\theta_1(\eta \theta^*)  }{ \theta_2(\eta\theta^{*}) -\theta_2(\zeta \theta^{**})     }$
 Then for any  $\alpha_0 \in ]0, \beta_0[$ we have (\ref{p**}), 
  for any $\alpha_0 \in ]\beta_0, \pi/2[$ we have (\ref{p*}) and for $\alpha_0=\beta_0$ we have
\begin{equation}
\label{p***}
\pi(r \cos \alpha, r \sin \alpha) \sim
d_1 \exp(-r\langle \zeta \theta^{**} \mid e_\alpha \rangle)+ d_2 \exp(-r\langle \eta \theta^{*} \mid e_\alpha \rangle),  \  \ r \to \infty, 
\alpha \to \alpha_0.
\end{equation} 
  
\item[(ivd)]    Let $\theta_1(\zeta \theta^{**})\geq \theta_1(\eta \theta^*)$  and  $\theta_2(\zeta \theta^{**}) \leq \theta_2(\eta \theta^*)$.
   Let us define angles $\alpha_1$ and $\alpha_2$  as in (ii) and (iii). Then $0<\alpha_1 \leq \alpha_2 <\pi/2$;
            for any  $\alpha_0 \in ]0, \alpha_1[$ we have (\ref{p**}) , for any $\alpha_0 \in ]\alpha_1, \alpha_2[$ we have 
   (\ref{assadlep}) and for any $\alpha_0 \in ]\alpha_2, \pi/2[$ we have (\ref{p*}).
  
\end{itemize}

\end{itemize}

\end{thm}

\begin{thm}  Let $\theta_1(s_2^+)\leqslant 0$, $\theta_2(s_1^+) \leqslant 0$.
\label{thmresultsnew2}

\begin{itemize}

\item[(i)] Let $\gamma_2(s_1^+)\leq 0$ and $\gamma_1(s_2^+)\leq 0$
   Then for any $\alpha_0 \in ]0, \pi/2[$ the asymptotics (\ref{assadlep}) is valid.

\item[(ii)] Let $\gamma_2(s_1^+)>0$ and $\gamma_1(s_2^+)\leq 0$.
 Let  $A^{**}\equiv  \frac{ d \Theta_2^{+}(\theta_1) }{ d \theta_1 } \Bigm|_{\theta_1^{**} }$.
  Then $\alpha_1= \arctan ( -1/ A^{**}) \in ]0, \pi/2[$.
   For any $\alpha_0 \in ]0, \alpha_1[$ the asymptotics  (\ref{p**}) is valid and and for any $\alpha \in ]\alpha_1, \pi/2[$ the asymptotics 
   (\ref{assadlep}) holds true.

\item[(iii)]  Let $\gamma_2(s_1^+)<0$ and $\gamma_1(s_2^+)\geq 0$.
  Let  $ A^*\equiv \frac{ d \Theta_1^{+}(\theta_2) }{ d \theta_2 } \Bigm|_{\theta_2^{*} } $.
  Then $\alpha_2= \arctan (- A^*) \in ]0, \pi/2[$.
   For any $\alpha_0 \in ]0, \alpha_2[$ the asymptotics (\ref{assadlep}) is valid and for any $\alpha \in ]\alpha_2, \pi/2[$ the asymptotics 
   (\ref{p*}) holds true.

\item[(iv)]   Let  $\gamma_2(s_1^+)>0$ and $\gamma_1(s_2^+)>0$.
    Then either
  $\theta_1(\zeta \theta^{**})< \theta_1(\eta \theta^*)$  and  $\theta_2(\zeta \theta^{**}) > \theta_2(\eta \theta^*)$, or 
 $\theta_1(\zeta \theta^{**})> \theta_1(\eta \theta^*)$  and  $\theta_2(\zeta \theta^{**}) < \theta_2(\eta \theta^*)$, or 
 finally  $\theta_1(\zeta \theta^{**})=\theta_1(\eta \theta^*)$  and  $\theta_2(\zeta \theta^{**}) = \theta_2(\eta \theta^*)$.

\begin{itemize}

\item[(iva)]   Let $\theta_1(\zeta \theta^{**})< \theta_1(\eta \theta^*)$  and  $\theta_2(\zeta \theta^{**}) > \theta_2(\eta \theta^*)$.
  Let us define $\beta_0 =\arctan \frac{ \theta_1(\zeta\theta^{**})-\theta_1(\eta \theta^*)  }{ \theta_2(\eta\theta^{*}) -\theta_2(\zeta \theta^{**})     }$
 Then for any  $\alpha_0 \in ]0, \beta_0[$ we have (\ref{p**}), 
  for any $\alpha_0 \in ]\beta_0, \pi/2[$ we have (\ref{p*}) and for $\alpha_0=\beta_0$ we have
(\ref{p***}).
  
\item[(ivb)]    Let $\theta_1(\zeta \theta^{**})> \theta_1(\eta \theta^*)$  and  $\theta_2(\zeta \theta^{**}) < \theta_2(\eta \theta^*)$ or 
    $\theta_1(\zeta \theta^{**})= \theta_1(\eta \theta^*)$  and  $\theta_2(\zeta \theta^{**}) = \theta_2(\eta \theta^*)$ 
   Let us define angles $\alpha_1$ and $\alpha_2$  as in (ii) and (iii). Then $0<\alpha_1 \leq \alpha_2 <\pi/2$;
            for any  $\alpha_0 \in ]0, \alpha_1[$ we have (\ref{p**}) , for any $\alpha_0 \in ]\alpha_1, \alpha_2[$ we have 
   (\ref{assadlep}) and for any $\alpha_0 \in ]\alpha_2, \pi/2[$ we have (\ref{p*}).
  
\end{itemize}

\end{itemize}

\end{thm}

\begin{thm}  
\label{thmresultsnew3}
 Let $\theta_1(s_2^+)> 0$, $\theta_2(s_1^+) \leqslant 0$.

\begin{itemize}

\item[(i)] Let $\gamma_2(s_1^+)\leq 0$ and $\gamma_1(s_2^+)\leq 0$
   Then for any $\alpha_0 \in ]0, \pi/2[$ the asymptotics (\ref{assadlep}) is valid.

\item[(ii)] Let $\gamma_2(s_1^+)>0$ and $\gamma_1(s_2^+)\leq 0$.
\begin{itemize}

\item[(iia)]  Let $\gamma_2(\zeta s_2^+)\geq 0$ or equivalently $\frac{ d \Theta_2^{+}(\theta_1) }{ d \theta_1 } \Bigm|_{\theta_1^{**} }\geq 0$.
   Then for any $\alpha_0 \in ]0, \pi/2[$ the asymptotics \eqref{p**} is valid.
\item[(iib)]  Let  $\gamma_2(\zeta s_2^+)<0$ or equivalently $A^{**}\equiv  \frac{ d \Theta_2^{+}(\theta_1) }{ d \theta_1 } \Bigm|_{\theta_1^{**} }
 < 0$.
   Define $\alpha_1= \arctan ( -1/ A^{**}) \in ]0, \pi/2[$.
   Then for any $\alpha_0 \in ]0, \alpha_1[$ we have (\ref{p**}) and for any $\alpha \in ]\alpha_1, \pi/2[$ we have 
   (\ref{assadlep}).

\end{itemize}

\item[(iii)]  Let $\gamma_2(s_1^+)<0$ and $\gamma_1(s_2^+)\geq 0$.
  Let  $ A^*\equiv \frac{ d \Theta_1^{+}(\theta_2) }{ d \theta_2 } \Bigm|_{\theta_2^{*} } $.
  Then $\alpha_2= \arctan (- A^*) \in ]0, \pi/2[$.
   For any $\alpha_0 \in ]0, \alpha_2[$ the asymptotics (\ref{assadlep}) is valid and for any $\alpha \in ]\alpha_2, \pi/2[$ the asymptotics 
   (\ref{p*}) holds true.

\item[(iv)]   Let  $\gamma_2(s_1^+)>0$ and $\gamma_1(s_2^+)>0$.
    Then either $\theta_1(\zeta \theta^{**})\leq \theta_1(\eta \theta^*)$  and  $\theta_2(\zeta \theta^{**})\leq \theta_2(\eta \theta^*)$, or
  $\theta_1(\zeta \theta^{**})< \theta_1(\eta \theta^*)$  and  $\theta_2(\zeta \theta^{**}) > \theta_2(\eta \theta^*)$, or 
 $\theta_1(\zeta \theta^{**})> \theta_1(\eta \theta^*)$  and  $\theta_2(\zeta \theta^{**}) < \theta_2(\eta \theta^*)$, or 
 finally  $\theta_1(\zeta \theta^{**})=\theta_1(\eta \theta^*)$  and  $\theta_2(\zeta \theta^{**}) = \theta_2(\eta \theta^*)$.

\begin{itemize}
\item[(iva)]  Let $\theta_1(\zeta \theta^{**})\leq \theta_1(\eta \theta^*)$  and  $\theta_2(\zeta \theta^{**})\leq \theta_2(\eta \theta^*)$
  where at least one of inequalities is strict.
  Then for any  $\alpha_0 \in ]0, \pi/2[$ we have (\ref{p**}).

\item[(ivb)]   Let $\theta_1(\zeta \theta^{**})< \theta_1(\eta \theta^*)$  and  $\theta_2(\zeta \theta^{**}) > \theta_2(\eta \theta^*)$.
  Let us define $\beta_0 =\arctan \frac{ \theta_1(\zeta\theta^{**})-\theta_1(\eta \theta^*)  }{ \theta_2(\eta\theta^{*}) -\theta_2(\zeta \theta^{**})     }$
 Then for any  $\alpha_0 \in ]0, \beta_0[$ we have (\ref{p**}), 
  for any $\alpha_0 \in ]\beta_0, \pi/2[$ we have (\ref{p*}) and for $\alpha_0=\beta_0$ we have
(\ref{p***}).
  
\item[(ivc)]    Let $\theta_1(\zeta \theta^{**})> \theta_1(\eta \theta^*)$  and  $\theta_2(\zeta \theta^{**}) < \theta_2(\eta \theta^*)$ or 
    $\theta_1(\zeta \theta^{**})= \theta_1(\eta \theta^*)$  and  $\theta_2(\zeta \theta^{**}) = \theta_2(\eta \theta^*)$ 
   Let us define angles $\alpha_1$ and $\alpha_2$  as in (ii) and (iii). Then $0<\alpha_1 \leq \alpha_2 <\pi/2$;
            for any  $\alpha_0 \in ]0, \alpha_1[$ we have (\ref{p**}) , for any $\alpha_0 \in ]\alpha_1, \alpha_2[$ we have 
   (\ref{assadlep}) and for any $\alpha_0 \in ]\alpha_2, \pi/2[$ we have (\ref{p*}).
  
\end{itemize}

\end{itemize}

\end{thm}

The symmetric theorem for the case $\theta_1(s_2^+)\leqslant 0$, $\theta_2(s_1^+) > 0$ holds.

\subsection{Concluding remarks}
\label{subsec:concludingremarks}
     Let us remark that the approach of this article applies to the SRBM  in \underline{any cone of ${\bf R}^2$}.
Thanks to a  linear transformation $T\in {\bf R}^{2\times 2}$, it is easy to transform $Z(t)$, a reflected Brownian motion
  of parameters $(\Sigma,\mu,R)$ in a cone  into $T Z(t)$  a reflected Brownian motion 
   of parameters $(T\Sigma T^t,T\mu,TR)$   in the quarter plane.
For example if the initial cone is the set
$\{(x,y) | x\geqslant0 \text{ and } y\leqslant ax \}$ for some $a>0$, we may just take $T=
\begin{pmatrix}
1 & -\frac{1}{a} \\ 
0 & 1
\end{pmatrix}$.
The process $T Z(t)$ lives in a quarter plane. Then the approach of this article applies  and its results can be converted  to the initial cone
 by the inverse linear transformation.  
 The analytic approach  for discrete random walks is essentially  restricted to those with jumps to the nearest neighbors in the interior of the quarter plane.
 Since a linear transformation can not generally keep the length of jumps, this procedure  does not work in the discrete case. That is why the analytic
 approach in ${\bf R}^2$ has a more general scope of applications.  

\medskip 
To conclude this article, we sketch the way of recovering the asymptotic results of Dai and Miyazawa  \cite{dai_reflecting_2011}
   via the approach of this article.
  Given a directional vector $c=(c_1,c_2) \in {\bf R}_+^2$, thanks to the representation of Lemma \ref{propI} we obtain
 \begin{eqnarray}
\lefteqn{
\mathbf{P}( \langle c  \mid Z(\infty) \rangle \geq R) =\int_{x_1 \geqslant 0, \ x_2 \geqslant 0 \atop c_1x_1+c_2x_2 \geqslant R} \pi(x_1,x_2) \mathrm{d}x_1 \mathrm{d}x_2}\nonumber\\
&=&
\int_{x_1 \geqslant 0, \ x_2 \geqslant 0 \atop c_1x_1+c_2x_2 \geqslant R}
 I_1(x_1, x_2) \mathrm{d}x_1 \mathrm{d}x_2
 +
 \int_{x_1 \geqslant 0, \ x_2 \geqslant 0 \atop c_1x_1+c_2x_2 \geqslant R}
 I_2(x_1,x_2) \mathrm{d}x_1 \mathrm{d}x_2\nonumber\\
&=&\int\limits_{\mathcal{I}_{\theta_1}^{\epsilon,+}} g_1(\theta_1) \frac{1}{\theta_1}\frac{1}{\Theta_2^+(\theta_1)-\theta_1 \frac{c_2}{c_1}} e^{-\frac{ R}{c_1}\theta_1} \mathrm{d}\theta_1+
\int\limits_{\mathcal{I}_{\theta_1}^{\epsilon,+}} g_1(\theta_1) \frac{-c_2/c_1}{\Theta_2^+(\theta_1)(\Theta_2^+(\theta_1)-\theta_1 \frac{c_2}{c_1})} e^{-\frac{ R}{c_2}\Theta_2^+(\theta_1)} \mathrm{d}\theta_1\label{gg}\\
&&{}+\int\limits_{\mathcal{I}_{\theta_2}^{\epsilon,+}} g_2(\theta_2) \frac{1}{\theta_2}\frac{1}{\Theta_1^+(\theta_2)-\theta_2 \frac{c_1}{c_2}} e^{-\frac{ R}{c_2}\theta_2} \mathrm{d}\theta_2+
\int\limits_{\mathcal{I}_{\theta_2}^{\epsilon,+}} g_2(\theta_2) \frac{-c_1/c_2}{\Theta_1^+(\theta_2)(\Theta_1^+(\theta_2)-\theta_2 \frac{c_1}{c_2})} e^{-\frac{ R}{c_1}\Theta_1^+(\theta_2)} \mathrm{d}\theta_2,\nonumber\\
\label{ggg}
\end{eqnarray}
   where
$$g_1(\theta_1)= \frac{\varphi_2(\theta_1) \gamma_2(\theta_1, \Theta_2^+(\theta_1)) }{\sqrt{  d(\theta_1) }}, \  \  \  \
g_2(\theta_2)= \frac{\varphi_1(\theta_2) \gamma_1(\Theta_1^+(\theta_2), \theta_2) }{\sqrt{  d(\theta_2) }}.$$
    The first term in (\ref {gg}) is just the Laplace transform of  the function
$h_1(\theta_1)=g_1(\theta_1) \frac{1}{\theta_1}\frac{1}{\Theta_2^+(\theta_1)-\theta_1 \frac{c_2}{c_1}}$, its asymptotics is determined by
  the smallest real singularity of $h_1(\theta_1)$, see e.g. \cite{doetsch_introduction_1974}.
    This may be either the branch  point $\theta_1^+$ of $\phi_2(\theta_1)$, or the smallest
  pole of $h_1(\theta_1)$ on $]0, \theta_1^+[$ whenever it exists, the natural candidates are $\zeta \theta^{**}$,  $\zeta \eta  \theta^*$
  due to Lemmas \ref{lem3}-- \ref{lem4}
   or a point $\theta^c=(\theta_1^c, \theta_2^c)$ such that $\theta_2^c=\Theta_2^+(\theta_1^c)=
  \theta_1^c \frac{c_2}{c_1}$.  To determine the asymptotics of the second integral in (\ref{gg}), we shift the integration contour to the new one
  passing through the saddle-point
$\Theta_1(\theta_2^+)$ and take  into account the poles of the integrand we encounter, the most important of these poles are those listed above.
  The asymptotics of two terms in (\ref{ggg}) is determined in the same way. Combining all these results together we derive the main asymptotic term depending on the parameters that can be either
  $e^{-\frac{R}{c_1}\theta_1^{+}}$, $e^{-\frac{R}{c_2}\theta_2^+}$ preceding by $R^{-1/2}$ or $R^{-3/2}$ with some constant,
  or $e^{-\frac{R}{c_1}\theta_1^c}=e^{\frac{R}{c_2}\theta_2^c}$, $e^{-\frac{R}{c_i} (\zeta\theta^{**})_i}$, 
   $e^{-\frac{R}{c_i} (\eta\theta^{*})_i}$, $i=1,2$ 
       preceding by some constant and  
         the factor $R$ in some critical cases.
This analysis leads to the results of
\cite{dai_reflecting_2011}.

\bigskip 

\noindent{\bf Acknowledgement}  We are grateful to Kilian Raschel and Sandrine Péch\'e for helpful discussions and suggestions.

\newpage

\begin{figure}[hbtp]
\centering
\includegraphics[scale=0.75]{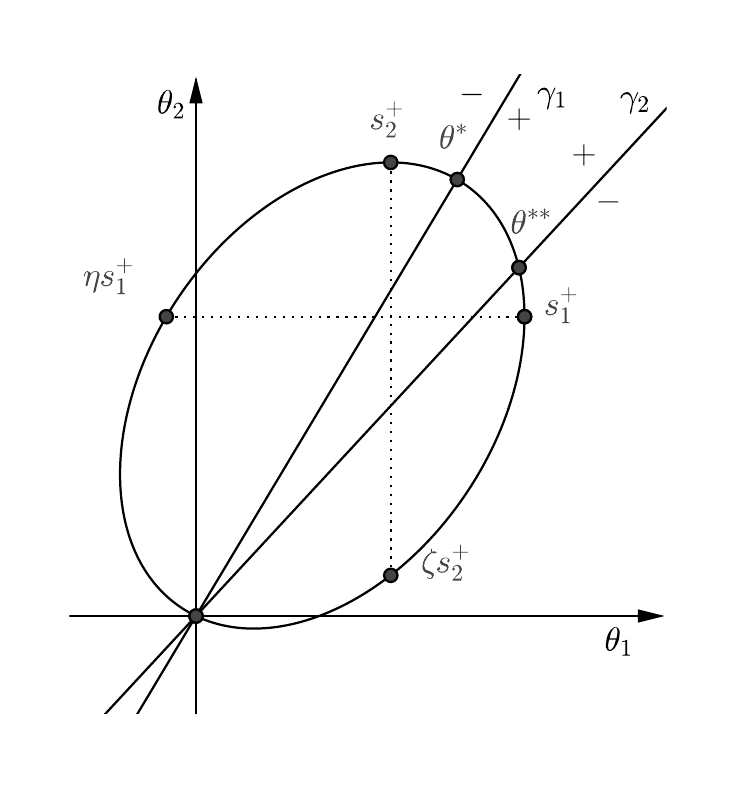}
\hfill
\includegraphics[scale=0.75]{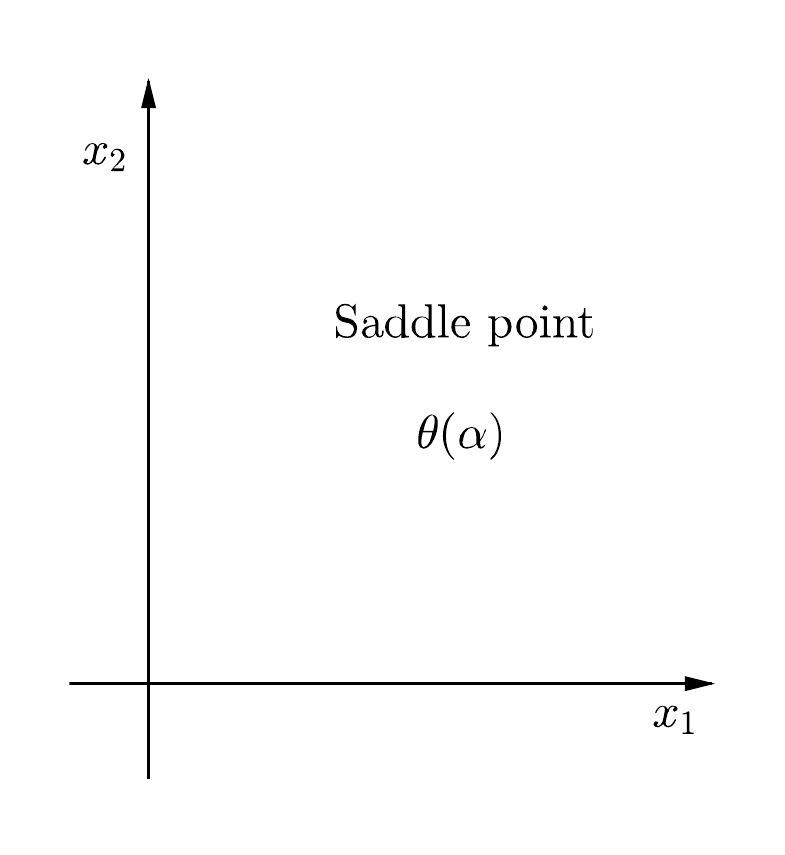}
\caption{Theorem \ref{thmresultsnew1} case (i)}
\label{firstfig}
\end{figure}

\begin{figure}[hbtp]
\centering
\includegraphics[scale=0.75]{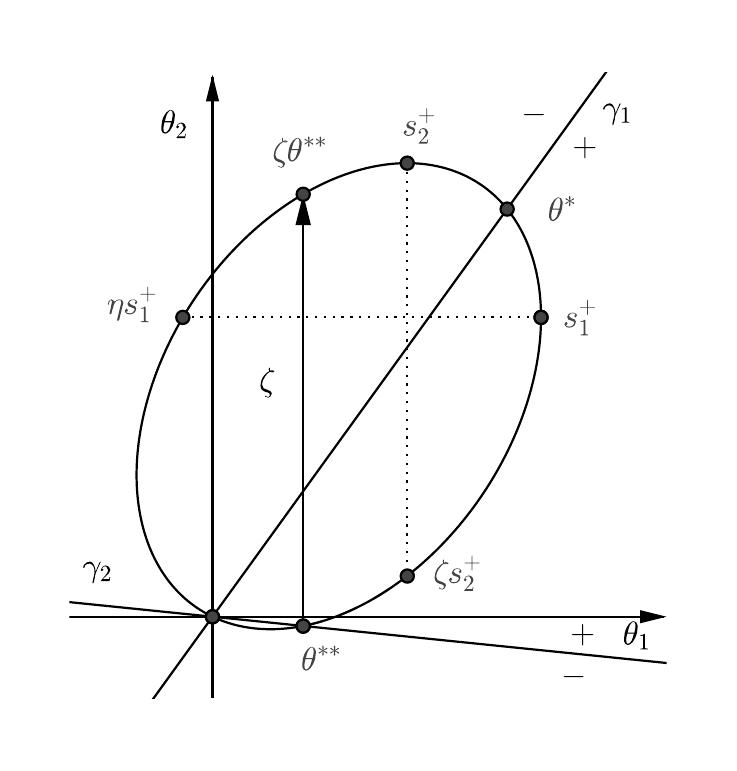}
\hfill
\includegraphics[scale=0.75]{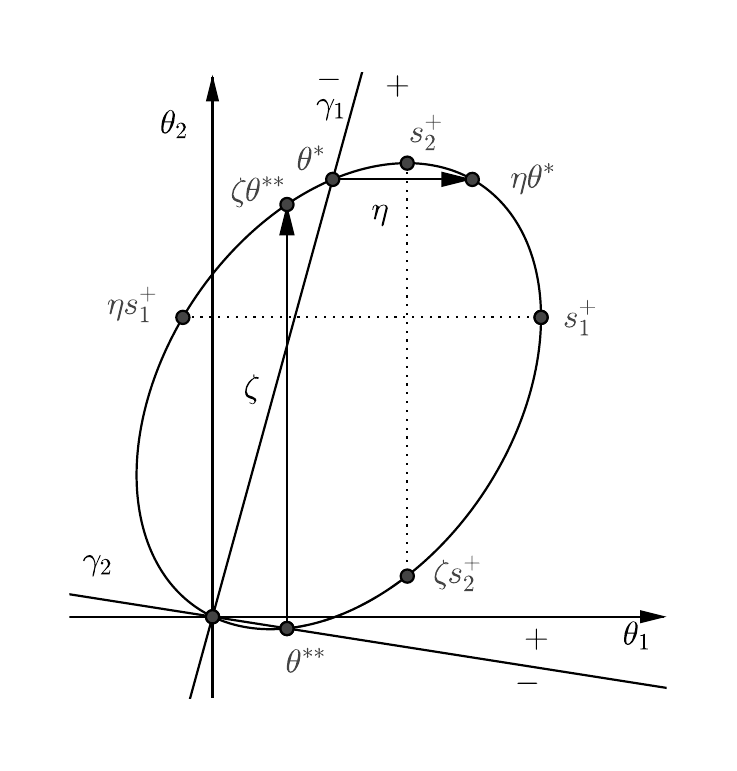}
\hfill
\includegraphics[scale=0.75]{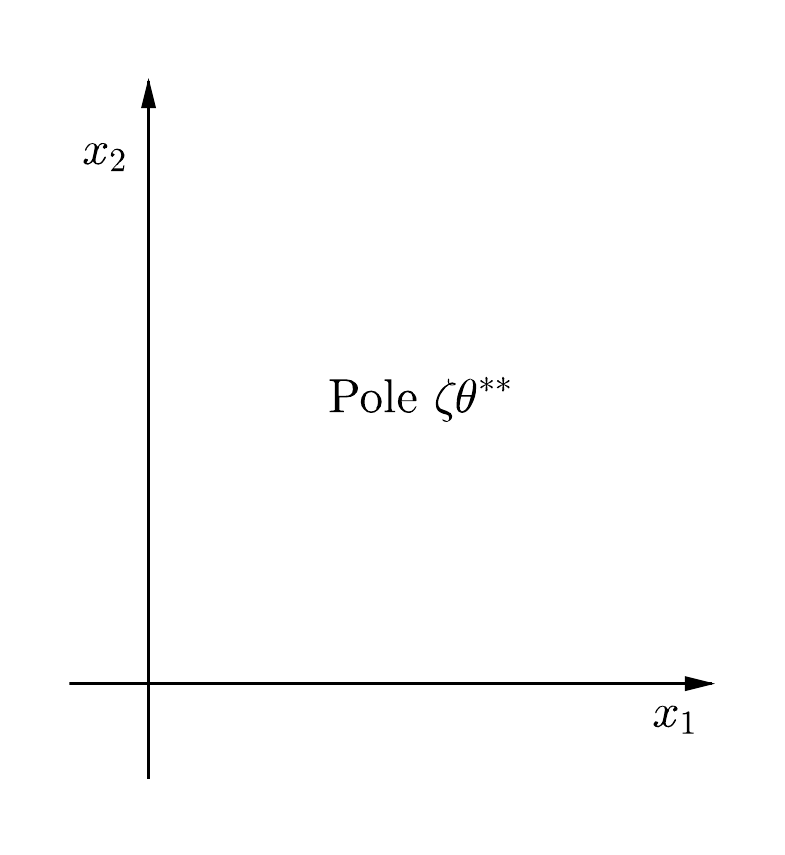}
\caption{Theorem \ref{thmresultsnew1} case (iia) and (iva)}
\end{figure}

\begin{figure}[hbtp]
\centering
\includegraphics[scale=0.75]{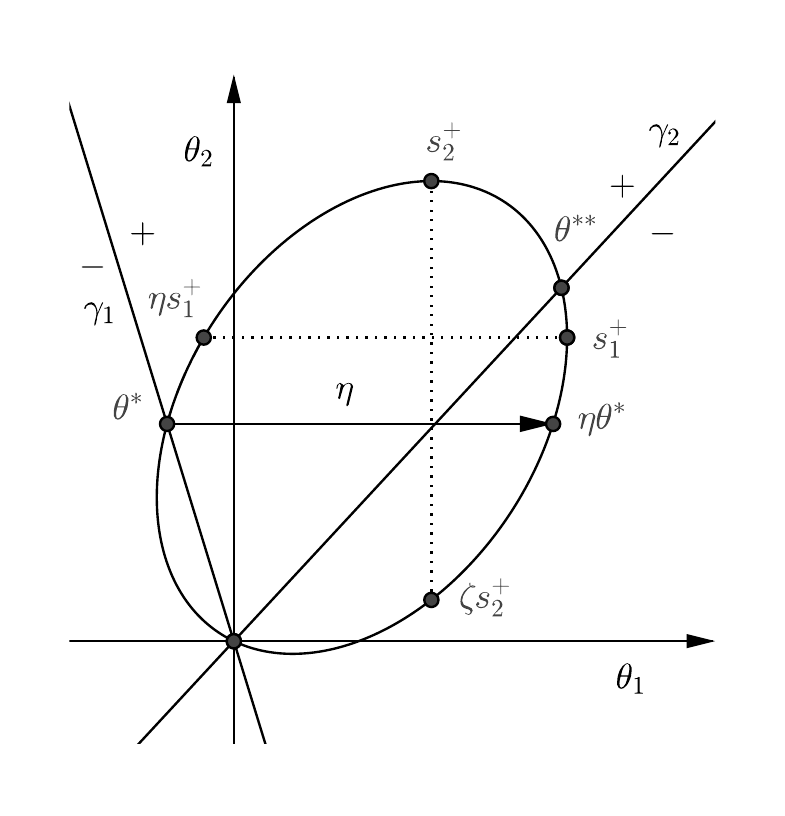}
\hfill
\includegraphics[scale=0.75]{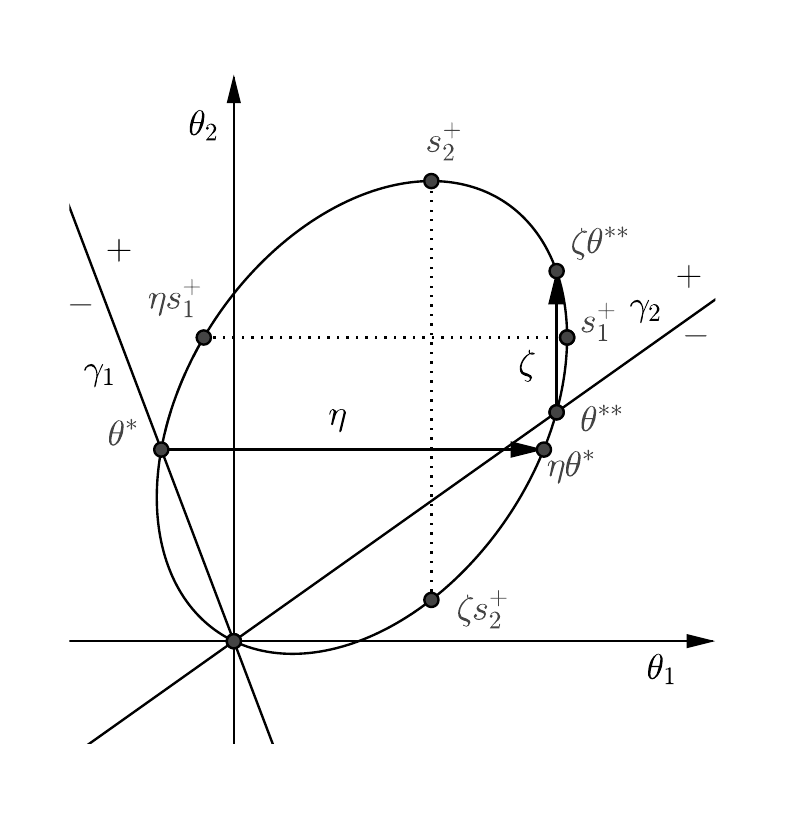}
\hfill
\includegraphics[scale=0.75]{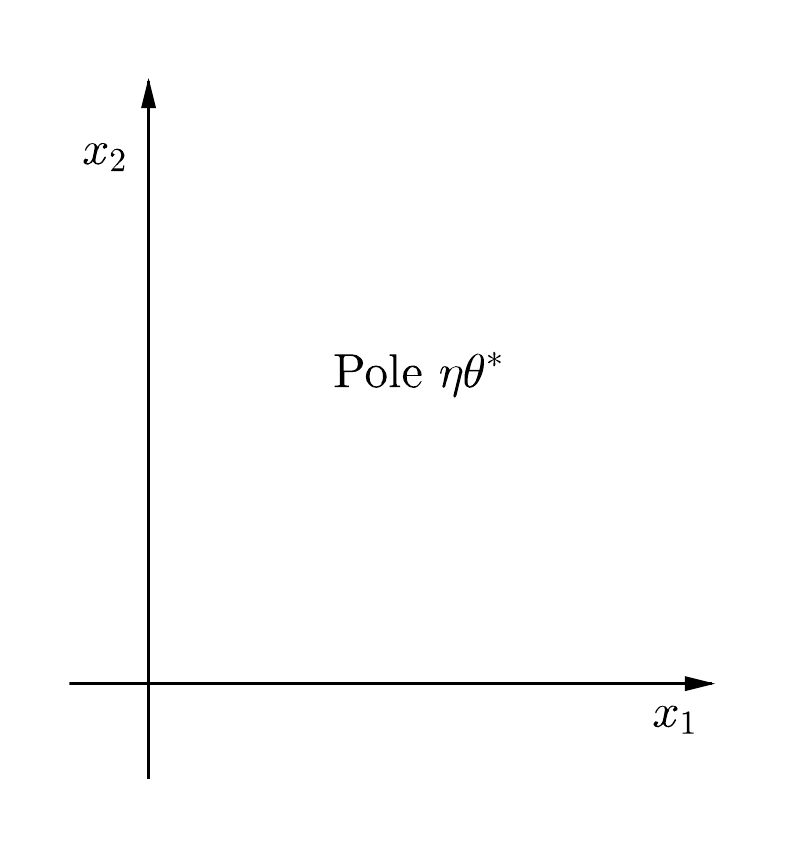}
\caption{Theorem \ref{thmresultsnew1} case (iiia) and (ivb)}
\end{figure}

\begin{figure}[hbtp]
\centering
\includegraphics[scale=0.75]{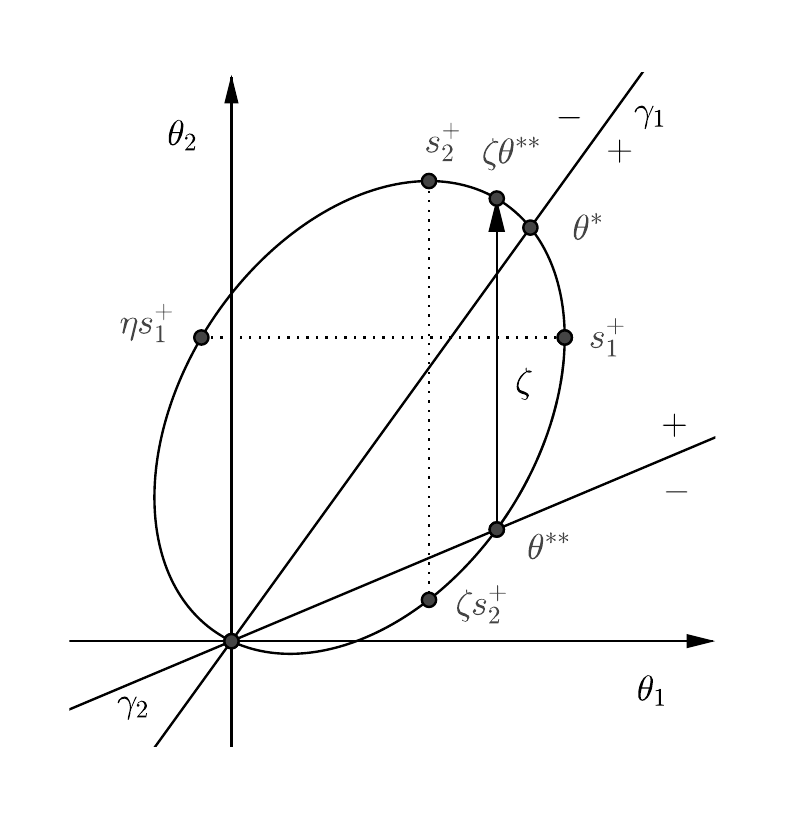}
\hfill
\includegraphics[scale=0.75]{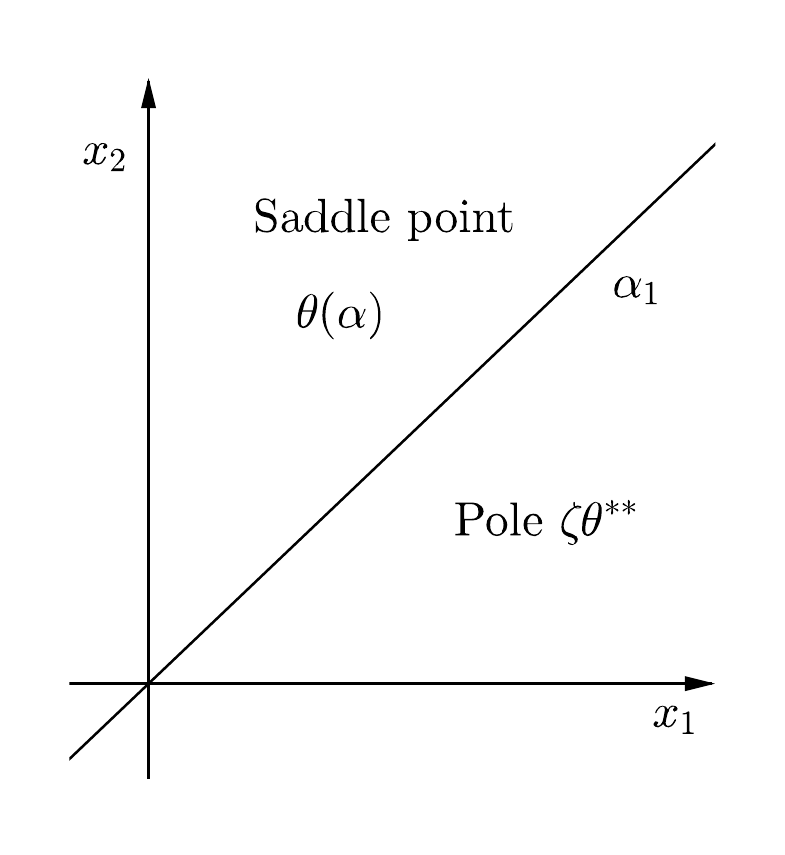}
\caption{Theorem \ref{thmresultsnew1} case (iib)}
\end{figure}

\begin{figure}[hbtp]
\centering
\includegraphics[scale=0.75]{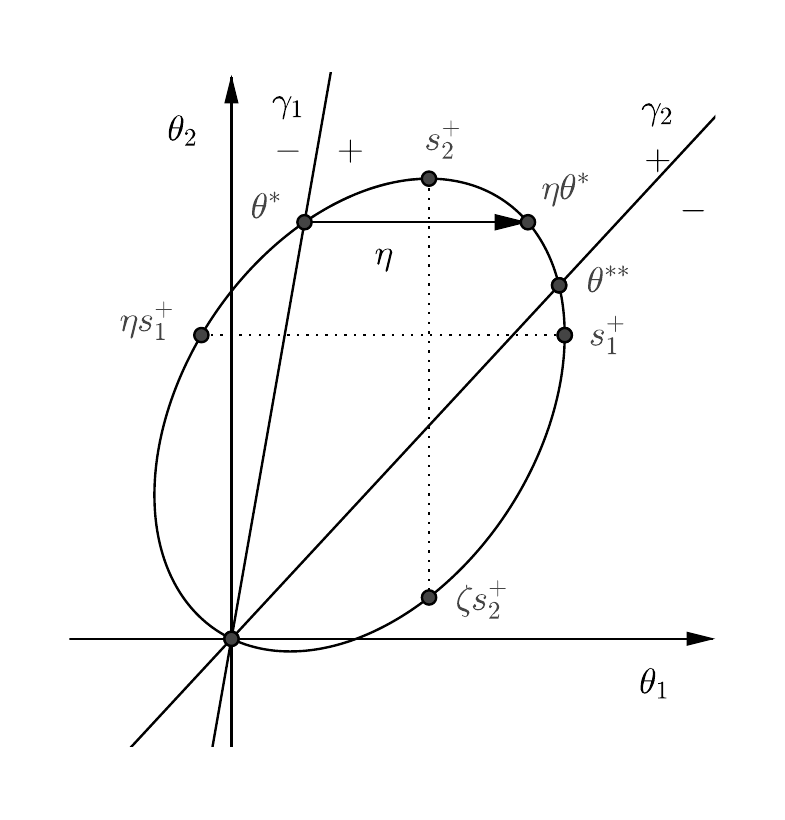}
\hfill
\includegraphics[scale=0.75]{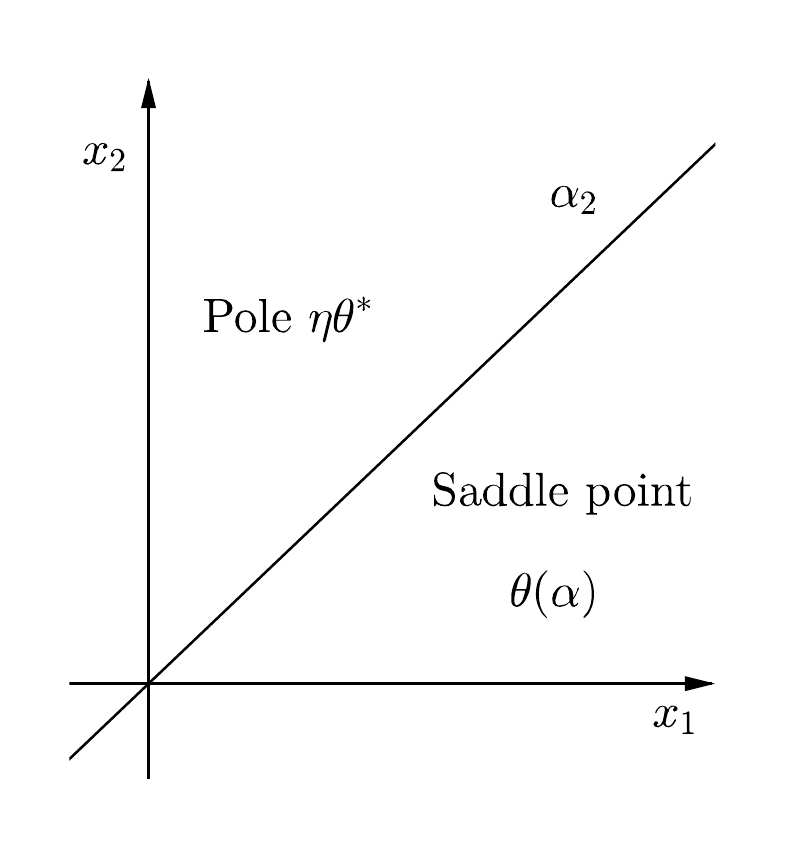}
\caption{Theorem \ref{thmresultsnew1} case (iiib)}
\end{figure}

\begin{figure}[hbtp]
\centering
\includegraphics[scale=0.75]{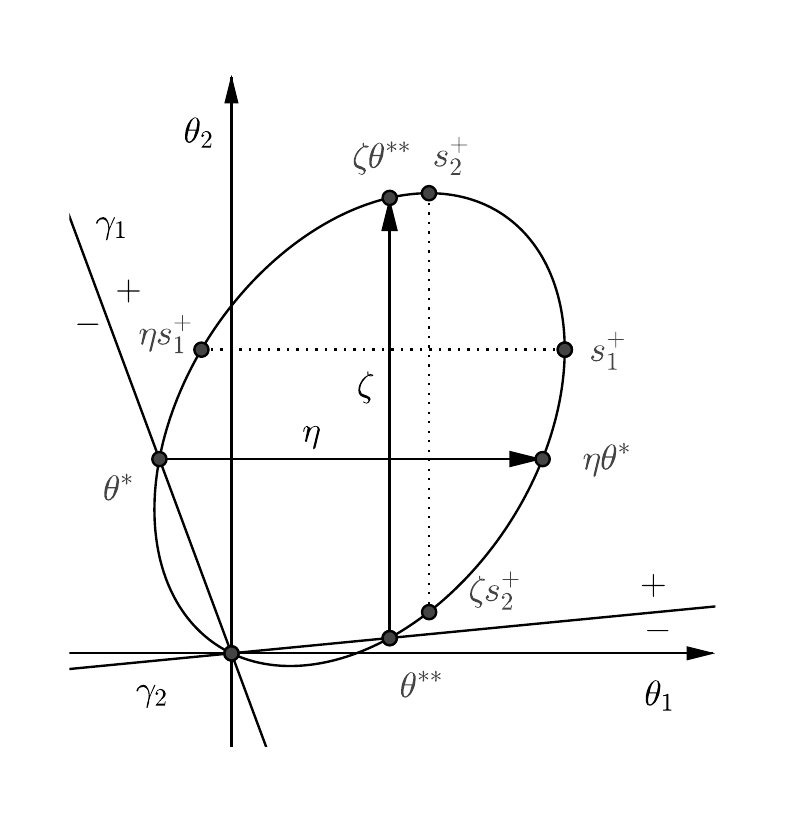}
\hfill
\includegraphics[scale=0.75]{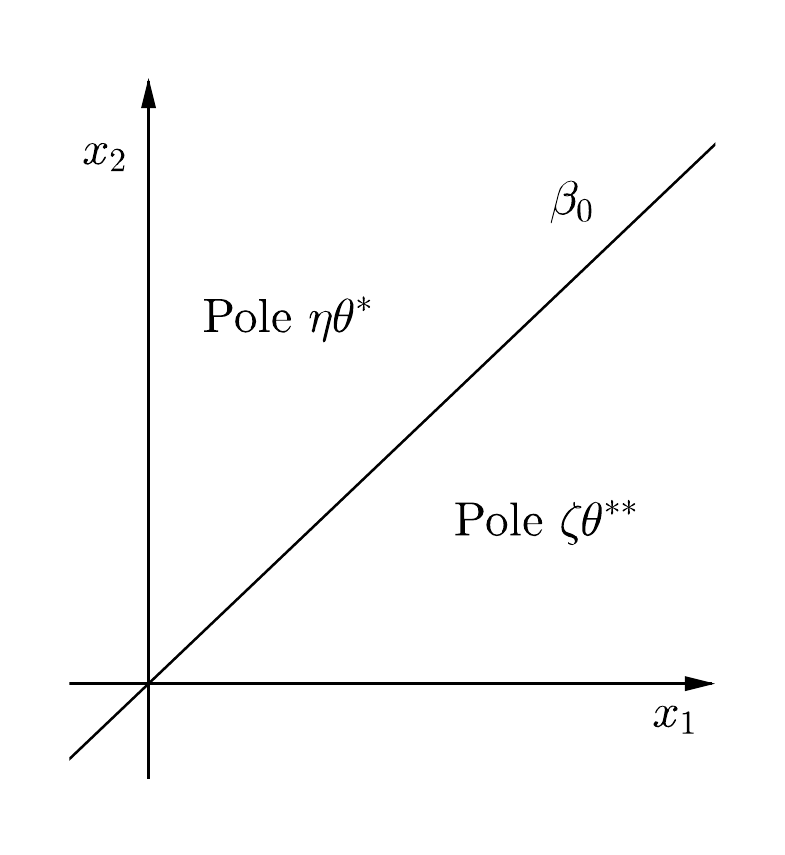}
\caption{Theorem \ref{thmresultsnew1} case (ivc)}
\end{figure}

\begin{figure}[hbtp]
\centering
\includegraphics[scale=0.75]{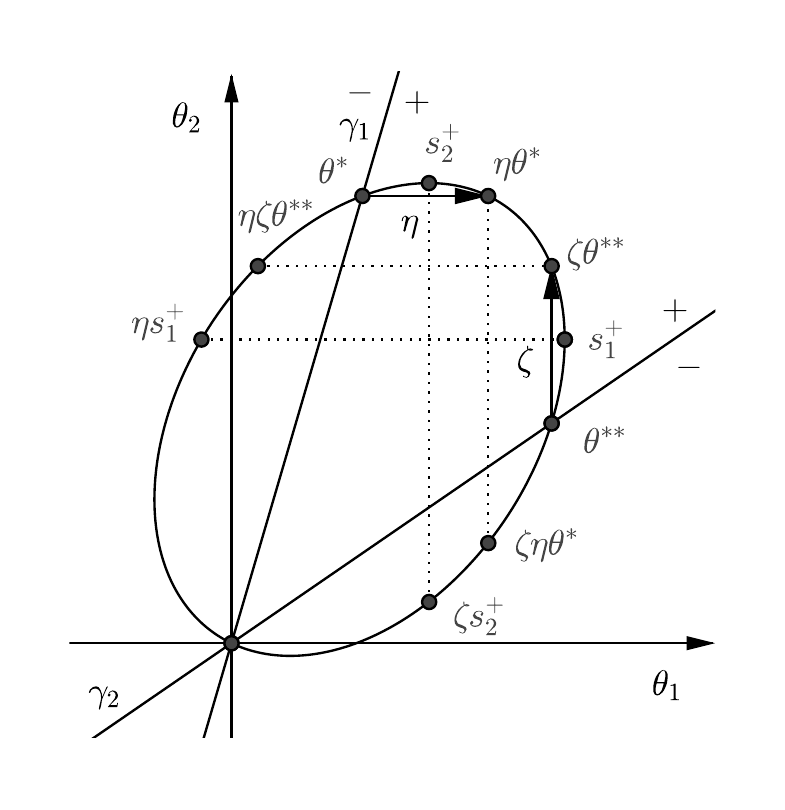}
\hfill
\includegraphics[scale=0.75]{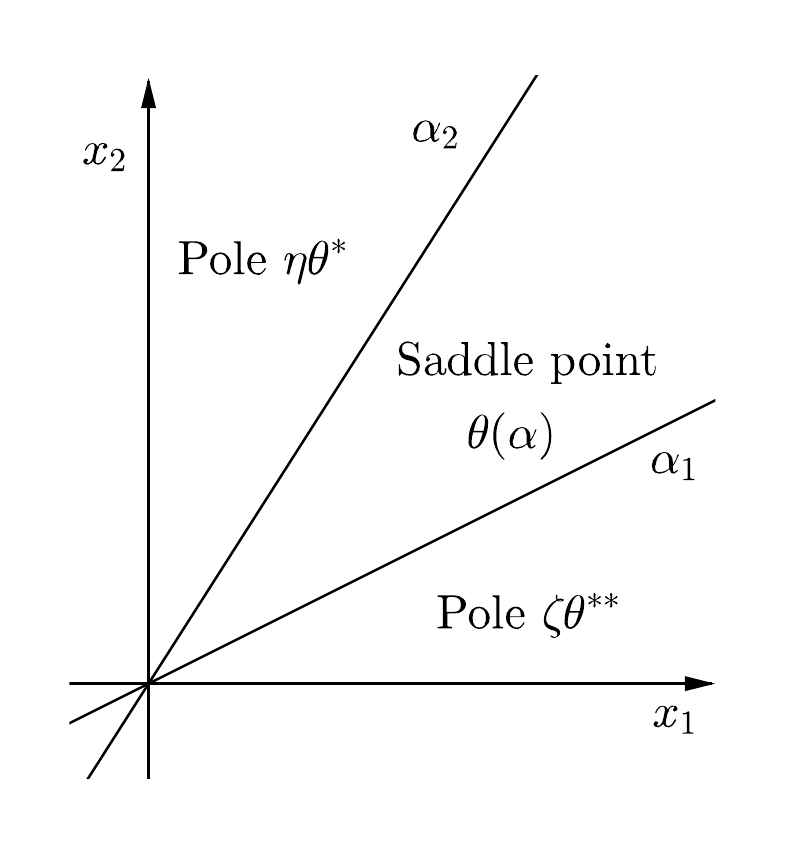}
\caption{Theorem \ref{thmresultsnew1} case (ivd)}
\label{lastfig}
\end{figure}

\clearpage

\bibliographystyle{apalike} 
\bibliography{biblio2} 

\end{document}